\documentclass[12pt,amsymb,fullpage]{amsart}
\usepackage{amssymb,amscd,pstricks}

\newtheorem{theorem}{Theorem}[section]
\newtheorem{defn}[theorem]{Definition}

\newtheorem{lemma}[theorem]{Lemma}

\newtheorem{eple}[theorem]{Example}
\newtheorem{rmk}[theorem]{Remarks}
\newtheorem{dsc}[theorem]{Discussion}
\newtheorem{nota}[theorem]{Notation}

\newsavebox{\indbin}
\savebox{\indbin}{\begin{picture}(0,0)
\newlength{\gnu}
\settowidth{\gnu}{$\smile$} \setlength{\unitlength}{.5\gnu}
\put(-1,-.65){$\smile$} \put(-.25,.1){$|$}
\end{picture}}

\newcommand{\be}{\begin{enumerate}}
\newcommand{\bd}{\begin{defn}}
\newcommand{\bt}{\begin{theorem}}
\newcommand{\bl}{\begin{lemma}}
\newcommand{\ee}{\end{enumerate}}
\newcommand{\ed}{\end{defn}}
\newcommand{\et}{\end{theorem}}
\newcommand{\el}{\end{lemma}}

\begin{document}
\title{Flash Geometry of Algebraic Curves}
\author{Tristram de Piro}
\address{ Flat 1, 98 Prestbury Road, Cheltenham, GL52 2DJ} \email{depiro100@gmail.com}
\thanks{}
\begin{abstract}

In this paper, we develop the theory of flashes of an algebraic curve. We show that the theory is birationally invariant in a sense which we will make more precise below. We also show how the theory provides a foundation for the method of asymptotic degenerations, a particular class of degenerations of plane projective algebraic curves. In particular, we consider the geometrical technique in relation to the Severi problem of degenerating nodal curves to lines in general position, and correct Severi's original proof of his conjecture.

\end{abstract}
\maketitle

\begin{section}{Degenerating Nodal Curves}

In this section, we return to the convention in \cite{depiro6} that a plane algebraic curve $C$ is an \emph{irreducible} subvariety of $P^{2}$, of dimension $1$. By a node, we mean an ordinary double point, as in Definition 1.2. of \cite{depiro8}. By a plane nodal curve, we will mean a plane algebraic curve which has at most \emph{ordinary double points} as singularities, as defined in \cite{depiro8}. In particular, this convention is \emph{different} from the use of the term "node" in \cite{depiro6}, see also the explanation in \cite{depiro8}. We will also occasionally make use of the further assumption on $C$, see \cite{depiro6} and \cite{depiro7};\\

A generic point of $C$ has character $(1,1)$ $(\dag)$\\

This is to exclude certain exceptional curves in non-zero characteristic. We will return to these exceptional cases in the final section.\\

\indent An interesting geometric problem in the theory of nodal curves is to provide a general method of "degenerating" a plane nodal curve $C$ of degree $m$ into a union of $m$ lines, having transverse intersections. More specifically, if $C$ is a plane nodal curve of degree $m$, one can consider degenerations, over an irreducible parameter space, which, for convenience of terminology, we denote by $P^{1}$, of the form $\{C_{t}\}_{t\in P^{1}}$ with the following properties;\\

(i). The family $\{C_{t}:t\in P^{1}\}$ consists of plane reduced algebraic curves.\\

(ii). $C_{0}=C$.\\

(iii). For any $t\in P^{1}$, all the singularities of $C_{t}$ are
nodes.\\

(iv). For any $t\in (P^{1}\setminus\infty)$, $C_{t}$ is a nodal curve (in particularly irreducible) and has the same number of singular points as $C$.\\

(v) $C_{\infty}$ is a union of $m$ lines $\{l_{1},\ldots,l_{m}\}$.\\
 in general position, that is, the lines are distinct and, if $m\geq 3$, there does not exist a triple $(j_{1},j_{2},j_{3})$, with $j_{1}<j_{2}<j_{3}$ such that $(l_{j_{1}}\cap l_{j_{2}})=(l_{j_{2}}\cap l_{j_{3}})$.\\

Severi referred to a plane nodal curve, having such a degeneration, as a "nodal curve of the main stream". He conjectured that \emph{any} nodal curve is a nodal curve of the main stream. However, his proof of the result is erroneous.\\
\indent We do not attempt to answer this conjecture fully in the following paper. However, it is hoped that the general theory of flash geometry, which we will develop here, can help to resolve the fallacious steps in Severi's argument. We will return to this question in a later section.\\

The following lemma will be used repeatedly in arguments concerning nodal curves. The reader should look at \cite{depiro6} or \cite{depiro7} for relevant terminology.\\

\begin{lemma}{Nodal Curve Presentation}\\

Let $C$ be a plane nodal curve of degree $m$, satisfying $(\dag)$. Then there exists an affine coordinate system $(x,y)$ for $P^{2}$, such that, in this coordinate system;\\

(i). The line at $\infty$ cuts $C$ transversely in $m$ distinct non-singular points.\\

(ii). The tangent lines to $C$ parallel to the $y$-axis all have $2$-fold contact(contatto), and are based at non-singular points of $C$, which are in finite position.\\

(iii). Considering $x$ as a rational function on $C$, in the terminology of \cite{depiro7}, the weighted set $G=(x=0)$ consists of $m$ distinct branches, each counted once, lying inside $NonSing(C)$. Moreover, these branches are all in finite position, with base points distinct from the points of contact in $(ii)$.\\

(iv). There is no line $l$, parallel to the $y$-axis, intersecting $C$ in more than one point of the set consisting of its finitely many ordinary double points or the finitely many points named in (ii).

\end{lemma}

\begin{proof}
By the assumption $(\dag)$ and Remarks 6.6 of \cite{depiro6}, there exist only finitely many non-ordinary branches. Hence, there exist finitely many tangent lines $\{l_{\gamma_{1}},\ldots,l_{\gamma_{r}}\}$, based at $\{p_{1},\ldots,p_{r}\}$, (possibly with repetitions), such that;\\

$I_{\gamma_{j}}(C,p_{j},l_{\gamma_{j}})\geq 3$, (for $1\leq j\leq r$)\\

Moreover, $C$ has finitely many ordinary double points $\{q_{1},\ldots,q_{s}\}$ as singularities. Let $\{l_{\gamma_{q_{1}}^{1}},l_{\gamma_{q_{1}}^{2}},\ldots,l_{\gamma_{q_{s}^{1}}},l_{\gamma_{q_{s}^{2}}}\}$
be the $2s$ tangent lines, (possibly with repetitions), corresponding to these ordinary double points. Let $\{l_{q_{i}q_{j}}:1\leq i<j\leq s\}$ be the finitely many lines, (possibly with repetitions), passing through at least $2$ of these ordinary double points. Let $\{l_{1},\ldots,l_{t}\}$ be the finitely many lines which are tangent to at least $2$ distinct branches of $C$. That there exist finitely many such lines follows from duality arguments, the reader should look at Lemma 5.11 of \cite{depiro13}. Finally, let $\{l_{1},\ldots,l_{w}\}$ be the finitely many lines, which have the property that, they pass through one of the ordinary double points $\{q_{1},\ldots,q_{s}\}$, and are also tangent to a branch, centred at a nonsingular point. That there are finitely many such lines follows from consideration of each of the finite number of pencils centred at $\{q_{1},\ldots,q_{s}\}$, and results from \cite{depiro6}.\\
\indent Now choose a point $P$ not lying on $C$ or any of the above defined lines. Let $\Sigma=\{l_{\lambda}^{P}\}_{\lambda\in P^{1}}$ be the pencil defined by all lines passing through the point $P$. Then $\Sigma$ defines a $g_{m}^{1}$ on $C$ without fixed branches. By Lemma 2.17 of \cite{depiro6}, for generic $\lambda$, $l_{\lambda}^{P}$ intersects $C$ transversely in $m$ distinct branches, based at non singular points of $C$. By construction, we also have that, if $l_{Px}$ belongs to $\Sigma$ and defines the tangent line $l_{\gamma_{x}}$ to a branch $\gamma_{x}$ based at $x$, then $x\in NonSing(C)$, $l_{Px}$ has $2$-fold contact (contatto) with the branch $\gamma_{x}$, and $l_{Px}$ cannot pass through either an ordinary double point or be tangent to another branch $\gamma_{x'}$ of $C$, for $x'$ distinct from $x$. $(*)$. Moreover, if $l_{Px}$ belongs to $\Sigma$, it can pass through at most one ordinary double point of $C$, $(**)$. Now choose a homography, sending the point $[0:1:0]$ and the line $Z''=0$, in the original coordinates $[X'':Y'':Z'']$ on $P^{2}$, to $P$ and $l_{\lambda}^{P}$. Let $[X':Y':Z']$ be the new coordinate system defined by this homography. For the affine coordinate system $(x',y')$, defined by $x'={X'\over Z'}$ and $y'={Y'\over Z'}$, we have that the line at $\infty$ has the property $(i)$. The lines parallel to the $y'$-axis correspond to the lines, excluding $l_{\lambda}^{P}$, in the pencil defined by $\Sigma$, in this new coordinate system. Hence, $(ii)$ and $(iv)$ follows immediately from $(*)$ and $(**)$. Now, considering $x'$ as a rational function on $C$, by the same argument as above, for generic $\mu\neq\infty$, the line defined by $(x'=\mu)$, intersects $C$ transversely in $m$ distinct non-singular branches. Moreover, as $P$ ($[0:1:0]$ in this new coordinate system) does not lie on $C$, it follows that all these intersections lie in finite position, $(***)$. Now, let $\theta$ be the homography of $P^{2}$ defined by the affine translation $(x',y')\mapsto (x'-\mu,y')$ and let $(x,y)$ be the new coordinates defined by this map. As $P$ and the line at $\infty$ are fixed by $\theta$, and, moreover, the tangent lines parallel to the $y'$-axis are translated to tangent lines parallel to the $y$-axis, conditions $(i)$, $(ii)$ and $(iv)$ are preserved. By $(***)$, we then have that $(iii)$ holds as well.

\end{proof}

\end{section}

\begin{section}{Newtonian Methods}

We now discuss some geometric methods, originally developed by Isaac Newton, that constitute the main ideas behind the technique of flash geometry.\\

\begin{defn}{Asymptotes, Hyperbolic and Parabolic Branches}\\

Let $C$ be a plane irreducible algebraic curve of degree $m$, defined in the coordinate system $(x,y)$, not equal to a line. We define a line $l_{(a,b,t)}$, given by the equation $ax+by=t$, to be an asymptote if;\\

(i). $l_{(a,b,t)}$ passes through one of the intersections $p\in C\cap l_{\infty}$\\

(ii). $l_{(a,b,t)}$ is tangent to at least one of the branches $\gamma_{p}$ centred at $p$.\\

We define a branch $\gamma_{p}$, centred at one of the intersections $p\in C\cap l_{\infty}$, to be an infinite branch, and a finite branch otherwise. We also define;\\

$\gamma_{p}$ to be parabolic if the line $l_{\infty}$ is the tangent line to the branch.\\

$\gamma_{p}$ to be hyperbolic otherwise.\\

If $C$ is equal to a line $l$, and this line does not coincide with the line $l_{\infty}$, we define $l$ to be its asymptote, and its one infinite branch $\gamma_{\infty}$ to be hyperbolic. If $C$ is the line $l_{\infty}$, we say that it has no asymptotes and all its infinite branches are parabolic.\\

If $C$ is not irreducible, we consider its \emph{distinct} irreducible factors $\{C_{1},\ldots,C_{k}\}$. The asymptotes of $C$ are defined as the union of the asymptotes of each factor $C_{j}$, for $1\leq j\leq k$. Similarly, its branches are classified according to the definition above, for each irreducible factor.\\

\end{defn}

As a simple consequence of this definition, we have;\\

\begin{lemma}
Let notation be as in Definition 2.1 and let $C$ be \emph{any} plane algebraic curve (possibly not irreducible) of degree $m$, having finite intersection with $l_{\infty}$. Then $C$ has $r_{1}\leq m$ parabolic branches and $r_{2}\leq m$ hyperbolic branches with $r_{1}+r_{2}\leq m$. In particular there exist $r_{3}\leq r_{2}\leq m$ distinct asymptotes to the curve $C$. If $C$ is a plane nodal curve and the coordinate system is chosen as in Lemma 1.1, then there exist exactly $m$ asymptotes to the curve $C$.
\end{lemma}

\begin{proof}
Suppose first that $C$ is irreducible. As we may assume that $C$ has finite intersection with $l_{\infty}$, by the Hyperspatial Bezout Theorem (Theorem 2.3 of \cite{depiro6}), we have that;\\

$\sum_{infinite\ branches\ \gamma_{p}}I_{\gamma_{p}}(C,l_{\infty})=m$ $(*)$\\

It follows trivially that $r_{1}+r_{2}\leq m$. If $l_{(a,b,t)}$ is an asymptote, then it defines the tangent line to an infinite branch $\gamma_{p}$. As $l_{(a,b,t)}$ is distinct from the line $l_{\infty}$, such a branch $\gamma_{p}$ must be hyperbolic. Any (hyperbolic) branch has a unique tangent line, hence distinct asymptotes must define tangent lines to distinct hyperbolic branches. Therefore, $r_{3}\leq r_{2}\leq m$. If $C$ is not irreducible, let $\{C_{1},\ldots,C_{j},\ldots,C_{k}\}$ be its distinct irreducible factors, let $r_{1,j}$ be the number of parabolic branches of $C_{j}$ and let $r_{2,j}$ be the number of hyperbolic branches of $C_{j}$. If $deg(C_{j})=m_{j}$, we clearly have that $m_{1}+\ldots+m_{j}+\ldots+ m_{k}\leq m$, $(**)$. Applying the previous considerations to each irreducible factor $C_{j}$,  we have that $r_{1,j}+r_{2,j}\leq m_{j}$. Applying the above Definition 2.1 and $(**)$, we have that $r_{1}+r_{2}\leq m$. Again, by a similar argument to the above, we have that $r_{3}\leq r_{2}\leq m$. For the final part of the lemma, by $(i)$ of Lemma 1.1 and $(*)$, we have that there exist $m$ distinct infinite hyperbolic branches. The tangent lines of these hyperbolic branches must be distinct, otherwise at least one would coincide with the line $l_{\infty}$, hence there exist exactly $m$ asymptotes as required.
\end{proof}

\begin{rmk}
An elegant method of determining the real asymptotes to a given plane real irreducible algebraic curve is given in \cite{Nunemacher}. I hope the author will not mind me essentially restating his results in the context of algebraically closed fields.\\

Theorem. (adapted almost verbatim from Nunemacher)\\

Let $C$ be any any plane algebraic curve and let $C$ be defined in the coordinate system $(x,y)$ by;\\

$P(x,y)=\sum_{k=0}^{m}P_{k}(x,y)$\\

where $P_{k}(x,y)$ is a homogeneous polynomial of degree $k$. Suppose that $(ax+by)$ is a factor of the top degree form $P_{m}(x,y)$ of multiplicity $r$. Let $s\leq r$ denote the largest integer with the property that there exist polynomials $Q_{j}(x,y)$ for $m-s+1\leq j\leq m$ satisfying the conditions;\\

$P_{m}(x,y)=(ax+by)^{s}Q_{m}(x,y)$\\

$P_{m-1}(x,y)=(ax+by)^{s-1}Q_{m-1}(x,y)$\\

$\ldots\ldots$\\

$P_{m-s+1}(x,y)=(ax+by)Q_{m-s+1}(x,y)$ $(*)$\\

Then associated with the factor $(ax+by)$ is a set of at most $s$ asymptotes $ax+by=t_{0}$, where $t_{0}$ is a root of the equation;\\

$t^{s}Q_{m}(b,-a)+t^{s-1}Q_{m-1}(b,-a)+\ldots+tQ_{m-s+1}(x,y)+P_{m-s}(b,-a)=0$\\

All the asymptotes to the curve $C$ in the coordinate system $(x,y)$ arise in this way as $(ax+by)$ ranges over the linear factors of $P_{m}(x,y)$\\

Corollary. (adapted almost verbatim from Nunemacher)\\

With the notation and hypotheses of the above theorem, if $(ax+by)$ is a simple factor of $P_{m}(x,y)$, then associated with this factor is the single asymptote to $C$, defined by the equation;\\

$(ax+by)Q_{m}(b,-a)+P_{m-1}(b,-a)=0$\\

The proof of the Theorem and its Corollary follow easily from the paper \cite{Nunemacher}. We have strengthened the hypotheses of the Theorem to include any plane algebraic curve, the corresponding Theorem in \cite{Nunemacher} \emph{excludes} the degenerate case that $C$ contains a line as one of its irreducible factors. However, given Definition 2.1, which includes this case, one can check that his proof
is unaffected in this greater degree of generality.\\

In order to give an illustration of the theorem and its corollary, consider the following curves;\\

(i). The hyperbola defined by $xy=\lambda$.\\

In this case, we have that $P_{2}(x,y)=xy$, whose simple linear factors are $x$ and $y$, and $P_{1}(x,y)=0$. By the corollary, the asymptotes are given by;\\

$x.y(0,-1)+0(0,-1)=0$ that is $x=0$\\

$y.x(1,0)+0(1,0)=0$ that is $y=0$\\

(Observe that this gives the correct result even in the degenerate case when $\lambda=0$)\\

(ii). The parabola defined by $x^{2}-y=0$.\\

In this case, we have that $P_{2}(x,y)=x^{2}$, which has $x$ as a linear factor of multiplicity $2$, and $P_{1}(x,y)=-y$. We have that $x$ does not divide $P_{1}(x,y)$, hence, by the theorem, the asymptotes are given by $x=t_{0}$, where $t_{0}$ is a root of;\\

$t.x(0,-1)-y(0,-1)=0$ that is $1=0$\\

As this has no solutions, we conclude that the parabola has no asymptotes.\\

(iii). The curve defined by;\\

$xy^{2}+ey=a^{2}x^{3}+bx^{2}+cx+d$, see (1) of Remarks 2.4.\\

Case 1. $a\neq 0$.\\

In this case, $P_{3}(x,y)=xy^{2}-a^{2}x^{3}$, which, assuming $char(L)\neq 2$, has three distinct linear factors $\{x,(y+ax),(y-ax)\}$, and $P_{2}(x,y)=-bx^{2}$. By the corollary, the asymptotes are given by;\\

$x.(y^{2}-a^{2}x^{2})(0,-1)+(-bx^{2})(0,-1)=0$ that is $x=0$\\

$(y-ax)(xy+ax^{2})(1,a)+(-bx^{2})(1,a)$ that is $y-ax={b\over 2a}$\\

$(y+ax)(xy-ax^{2})(1,-a)+(-bx^{2})(1,-a)$ that is $y+ax={-b\over 2a}$\\

The curve therefore has three distinct asymptotes, see (2) of Remarks 2.4. If $b=0$, these asymptotes intersect at the origin $(0,0)$, see (3) of Remarks 2.5. If $b\neq 0$, these asymptotes intersect in a triangle with vertices centred at $\{(0,{b\over 2a}),(0,{-b\over 2a}), ({-b\over 2a^{2}},0)\}$, see (4) of Remarks 2.4.\\

Case 2. $a=0, b\neq 0$.\\

In this case, $P_{3}(x,y)=xy^{2}$, which has $x$ as a simple factor and $y$ as a linear factor of multiplicity $2$, again $P_{2}(x,y)=-bx^{2}$. By the corollary, one asymptote is given by;\\

$x.y^{2}(0,-1)+(-bx^{2})(0,-1)=0$ that is $x=0$.\\

As $y$ does not divide $P_{2}(x,y)$, the theorem gives the other asymptotes as $y=t_{0}$, where $t_{0}$ is a root of;\\

$t.(xy)(1,0)+(-bx^{2})(1,0)=0$ that is $b=0$\\

As this has no solutions, we conclude there are no further asymptotes, see (5) of Remarks 2.4.\\

Case 3. $a=0,b=0$.\\

By the same reasoning as Case 2, we obtain $x=0$ as an asymptote. In this case, however, $y$ divides $P_{2}(x,y)=0$, so we have to consider $P_{1}(x,y)=ey-cx$. By the theorem, the other asymptotes are then given by $y=t_{0}$, where $t_{0}$ is a root of;\\

$t^{2}.x(1,0)+t.0(1,0)+(ey-cx)(1,0)$, that is $t^{2}-c=0$\\

In the case when $c=0$, we obtain one further asymptote $y=0$, otherwise, assuming $char(L)\neq 2$, we obtain $2$ further asymptotes $y=c^{1\over 2}$ and $y=-c^{1\over 2}$, see (6) of Remarks 2.4.\\

(iv). The curve defined by;\\

$xy=ax^{3}+bx^{2}+cx+d=0$, see (7) of Remarks 2.4.\\

 We have that $P_{3}(x,y)=ax^{3}$, which has $x$ as a linear factor of multiplicity $3$, $P_{2}(x,y)=bx^{2}-xy$. By the theorem, the asymptotes are given by $x=t_{0}$, where $t_{0}$ is a root of;\\

$t^{2}.x(0,-1)+t.(bx-y)(0,-1)+cx(0,-1)=0$ that is $t=0$\\

Hence, $x=0$ is the only asymptote to this curve.\\

(v). The curve defined by;\\

 $y^{2}=ax^{3}+bx^{2}+cx+d$, with $a\neq 0$, see (8) of Remarks 2.4.\\

 We have that $P_{3}(x,y)=ax^{3}$, which has $x$ as a linear factor of multiplicity $3$, $P_{2}(x,y)=bx^{2}-y^{2}$. By the theorem, the asymptotes are given by $x=t_{0}$, where $t_{0}$ is a root of;\\

$t.x^{2}(0,-1)+(bx^{2}-y^{2})(0,-1)=0$, that is $-1=0$\\

This has no solutions, hence the curve has no asymptotes.\\

(vi). The curve defined by;\\

 $y=ax^{3}+bx^{2}+cx+d$, with $a\neq 0$, see (9) of Remarks 2.4\\

 We have that $P_{3}(x,y)=ax^{3}$, which has $x$ as a linear factor of multiplicity $3$, $P_{2}(x,y)=bx^{2}$, and $P_{1}(x,y)=cx-y$. By the theorem, the asymptotes are given by $x=t^{0}$, where $t_{0}$ is a root of;\\

$t^{2}.x(0,-1)+t.bx(0,-1)+(cx-y)(0,-1)=0$, that is $1=0$\\

This has no solutions, hence the curve has no asymptotes.\\

\end{rmk}

\begin{rmk}

Isaac Newton gives the first effective method of determining the asymptotes of a real algebraic curve $C$ in \cite{Anal}. The technique is part of a more general construction called  "Newton's parallelogram". An excellent account of this construction is given in \cite{Talbot}. Moreover, Newton makes a systematic use of the method of asymptotes in his classification of real cubic curves, see \cite{Cubics} and \cite{Talbot}. His analysis of such curves in \cite{Cubics} begins with the observation;\\

"All lines of the 1st, 3rd, 5th, 7th, or odd orders, have at least two infinite branches extending in opposite directions"\\

In the case of a real cubic curve $C$, this translates to the fact that the line $l_{\infty}$ intersects $C$ in at least $1$ point, counted with multiplicity (appropriately defined). The observation is easily seen to be true from the fact that any real polynomial $p\in {\mathbb R}[y]$ of degree $3$ has at least $1$ real solution and consideration the limiting behavior of such a solution as $x\rightarrow\infty$ or $x\rightarrow -\infty$ in the defining equation $P(x,y)$ of $C$.  Excluding the degenerate case when $C$ has a parabolic branch, that is the line $l_{\infty}$ intersects $C$ in a point with multiplicity $2$, we obtain a hyperbolic branch and at least $1$ asymptote to the curve $C$. In this case, by taking such an asymptote to be the coordinate axis $x=0$, the general equation of $C$ reduces to the form;\\

$xy^{2}+(bx^{2}+cx+d)y=cx^{3}+dx^{2}+ex+f$ $(*)$\\

Such a quadratic equation can be analysed using the simple method of completing the square. This construction forms the basis for Newton's division of real cubics into four Cases. These Cases are then further analysed according to the behaviour of their asymptotes and the existence of a diameter (see \cite{Cubics} and \cite{Talbot} for a precise definition.) This leads to Newton's division of real cubics into a total of $72$ species. The curves considered as examples in the previous remark constitute a number of these species;\\

(1). Newton refers to this equation as The Case I cubic, see \cite{Cubics}.\\

(2). Newton refers to such a curve, in \cite{Cubics}, as a redundant or triple hyperbola (if $a^{2}>0$), and as a defective hyperbola (if $a^{2}<0$).\\

(3). The triple hyperbolas with this property are considered as Species 24-32 in Newton's classification of cubics, see \cite{Cubics}.\\

(4). The triple hyperbolas with this property are considered as Species 1-23 in Newton's classification of cubics (and are subdivided further according to the number of diameters). The defective hyperbolas are considered as Species 33-45 in Newton's classification, see \cite{Cubics}.\\

(5). Newton refers to such curves as parabolic hyperbolas, they account for Species 46-56 of his classification, see \cite{Cubics}.\\

(6). Newton refers to such curves as the hyperbolisms of the hyperbola ,$(c>0)$, ellipse, $(c<0)$, or the parabola, $(c=0)$, they account for species 57-60, species 61-63 and species 64-65 respectively, see \cite{Cubics}.\\

(7). Newton calls this equation The Case II cubic, see \cite{Cubics}, also known as Newton's trident.\\

(8). Newton calls this equation The Case III cubic, see \cite{Cubics}, referred to there as the Divergent Parabola.\\

(9). Newton calls this equation The Case IV cubic, see \cite{Cubics}, referred to there as the Cubic Parabola.\\

Perhaps the most important ingredient in Newton's analysis of cubics is the use that he makes of asymptotes as a means of representing cubic curves. A purely algebraic
proof of his result would be inconceivable without the supporting intuition of such a representation. The reader is strongly encouraged to look at the illustrations that Newton originally gave of each species in his classification, see \cite{Talbot}. The fact that Newton referred to a cubic curve (curve of degree $m$) as a "line of the
third order" (a "line of the m'th order"), clearly shows that he thought of such curves in terms of a system of lines, $(\dag)$, (\footnote{In the following section, this viewpoint will become clearer when we introduce the notion of flashes, as a means of representing an algebraic curve.}). I discuss this idea in greater detail in my forthcoming book "Christian Geometry", where I consider the aesthetics ideas underlying Newton's work. The principle $(\dag)$ is also illustrated by the problem of degenerating curves of degree $m$ to a union of $m$ lines, which we discussed in the first section.
\end{rmk}

Using the coordinate system of Lemma 1.1, we can find the following decomposition of $C$;\\

\begin{lemma}{Newton's Theorem}\\

Let hypotheses be as in Lemma 1.1, and let $F(x,y)$ define $C$ in the coordinate system $(x,y)$ given by Lemma 1.1. Then we can find algebraic power series $\{\eta_{1}(x),\ldots,\eta_{m}(x)\}$ in $L[[x]]$, with $\{\eta_{1}(0),\ldots,\eta_{m}(0)\}$ distinct, such that;\\

$F(x,y)=(y-\eta_{1}(x))\ldots (y-\eta_{m}(x))$\\

\noindent as an identity in the ring $L[[x]][y]$.

\end{lemma}

\begin{proof}

Let $\{a_{1},\ldots,a_{m}\}$ enumerate the $m$ distinct solutions to $F(0,y)=0$. Making the linear change of coordinates $y'=y-a_{j}$, the equation for $C$ in this system is given by;\\

$F_{j}(x,y')=F(x,y'+a_{j})=0$ $(\dag)$\\

By construction, we have that $F_{j}(0,0)=0$ and, moreover, by $(iii)$ of Lemma 1.1, ${\partial F_{j}\over \partial y}(0,0)\neq 0$. It follows, applying the implicit function theorem, that we can find an algebraic power series $\delta_{j}(x)\in L[[x]]$, with $\delta_{j}(0)=0$, such that $F_{j}(x,\delta_{j}(x))=0$, $(*)$, see especially the following remark. Now let $\eta_{j}(x)=a_{j}+\delta_{j}(x)$. By $(\dag)$, it follows immediately that $F(x,\eta_{j}(x))=0$, for $1\leq j\leq m$. Hence, we have that;\\

$(y-\eta_{j}(x))|F(x,y)$, for $1\leq j\leq m$, in $L((x))[y]$\\

using the fact that the division property holds in the polynomial ring $L((x))[y]$, as $L((x))$ is a field. We clearly have that each $(y-\eta_{j}(x))$ is irreducible in $L((x))[y]$, hence prime. Moreover, as the power series $\{\eta_{1}(x),\ldots,\eta_{m}(x)\}$ are distinct, the factors $(y-\eta_{j}(x))$ are coprime in $L((x))[y]$, for $1\leq j\leq m$. It, therefore, follows that;\\

$\prod_{1\leq j\leq m}(y-\eta_{j}(x))|F(x,y)$ in $L((x))[y]$ $(**)$\\

Now, using the fact that both sides of the relation have degree $m$, we can find $g(x)\in L((x))$ such that;\\

$F(x,y)=g(x)(y-\eta_{1}(x))\ldots(y-\eta_{m}(x))$ in $L((x))[y]$\\

As $F(0,y)=0$ has $m$ distinct solutions in $L$, we can write $F$ in the form;\\

$F(x,y)=y^{m}+p_{1}(x)y^{m-1}+\ldots+p_{j}(x)y^{m-j}+\ldots+p_{m}(x)$\\

\noindent with $p_{j}\in L[x]$ and $deg(p_{j})=j$. It follows that we can take $g(x)=1$ in $(**)$, so that the identity holds in $L[[x]][y]$. This shows the lemma.

\end{proof}

\begin{rmk}
The existence of power series solutions to polynomial equations $G(x,y)=0$, satisfying $(*)$ of the previous lemma, was first shown by Isaac Newton. He gives the following method of constructing a solution by successive approximations, in "The Praxis of Resolution", Paragraph 36 of \cite{Flux};\\

The assumption that $G(0,0)=0$ and ${\partial G\over\partial y}(0,0)\neq 0$ allows us to write;\\

$G(x,y)=p_{m}(x)y^{m}+\ldots+ p_{1}(x)y+p_{0}(x)$ $(*)$\\

with $p_{0}(0)=0$ and $p_{1}(0)\neq 0$. For $(x,y)$ "small", Newton observes that it, therefore, makes sense to take as a first approximate solution;\\

$y_{0}={-\lambda_{0}x^{i_{0}}\over p_{1}(0)}$\\

where $\lambda_{0}x^{i_{0}}$, $(i_{0}\geq 1$), is the first term in the expression for $p_{0}(x)$. Now, Newton makes the substitution $y=(y'+y_{0})$
in $(*)$, this results in a further polynomial equation of the same form;\\

$q_{m}(x)y'^{m}+\ldots+q_{1}(x)y'+q_{0}(x)=0$ $(**)$\\

By a straightforward algebraic calculation, using the fact that $i_{0}\geq 1$, one checks that $q_{1}(0)\neq 0$ and $ord(q_{0}(x))>ord(p_{0}(x))$. Hence, one can take as the second quote;\\

$y_{1}={-\lambda_{1}x^{i_{1}}\over q_{1}(0)}$\\

where $\lambda_{1}x^{i_{1}}$, $(i_{1}>i_{0})$, is the first term in the expression for $q_{0}(x)$. Continuing in this way, one obtains a sequence of approximate solutions;\\

$s_{n}(x)=y_{0}(x)+y_{1}(x)+\ldots +y_{n}(x)$, for $n\geq 0$\\

Either this process terminates after a finite number of approximations, giving a polynomial solution to $(*)$, or one obtains an infinite power series;\\

$s(x)=\sum_{i\geq 0}y_{i}(x)$\\

A straightforward algebraic calculation shows that;\\

$ord(G(x,s_{n+1}(x)))\geq ord(G(x,s_{n}(x)))+1$\\

Hence, by elementary completeness arguments for the power series ring $L[[x]]$, we are guaranteed that $G(x,s(x))=0$.\\

This method of constructing a solution to the polynomial equation $G(x,y)=0$ is closely related to the possibly more familiar Newton-Raphson method. Namely, one considers the function;\\

$G:L[x]\rightarrow L[x]$, $G(y)=p_{m}(x)y^{m}+\ldots+ p_{1}(x)y+p_{0}(x)$ $(\dag)$\\

Having obtained a first approximation $y_{0}=s_{0}$ to the equation $G(y)=0$, the Newton-Raphson method gives the further approximation;\\

$s_{1}=y_{0}-{G(y_{0})\over G'(y_{0})}=y_{0}(x)-{q_{0}(x)\over q_{1}(x)}$\\

where $q_{0}(x)$ and $q_{1}(x)$ are obtained from the transformed polynomial $(**)$ above. By a similar argument to the above, replacing the successive approximations $-{\lambda_{1}x^{i_{1}}\over q_{1}(0)}$ by $-{q_{0}(x)\over q_{1}(x)}$ and noting that;\\

${q_{0}(x)\over q_{1}(x)}={\lambda_{1}x^{i_{1}}\over q_{1}(0)}u(x)$ for a unit $u(x)\in L[[x]]$\\

one is similarly guaranteed that this method also yields the same power series solution $s(x)$ to the equation $G(y)=0$ in $(\dag)$. The Newton-Raphson method is usually applied to functions of a real variable, rather than the rings $L[x]$ or $L[[x]]$. However, that Newton intended his method of finding power series solutions to polynomial equations of the form $G(x,y)=0$, ("species"), to be a partial generalisation of this method is borne out by his consideration of the case of "affected equations", at the beginning of \cite{Flux}. An affected equation is just the case where the indeterminate $x$ is replaced by an explicit numerical value.\\
\indent Newton also gives a general method for finding power series solutions to polynomial equations of the form $G(x,y)=0$, without the simplifying assumption that ${\partial G\over\partial y}(0,0)\neq 0$. This is done by the introduction of "Newton's parallelogram", in Paragraph 29 of \cite{Flux}, and fractional exponents $x^{1\over t}$, for $t\geq 1$. The parallelogram method is further explained in \cite{Anal}. This leads to a general method of finding $n$ solutions $y_{j}(x)\in L((x^{1\over t}))$, for $1\leq j\leq n$, of a polynomial equation $G(x,y)=0$ of degree $n$, see \cite{Aby1} for a more precise statement of "Newton's Theorem". However, it seems clear from Newton's highly geometric approach, that he intended the theorem to provide a method of understanding the global structure of an algebraic curve, rather than as a means of analysing the nature of local singularities. This view is supported by consideration of his work on the classification of plane cubics, in \cite{Cubics}. Here, his power series method is used to find the asymptotes of such curves, which he refers to as "infinite branches", and branches at finite distances (those crossing the axis $x=0$), also referred to in the literature as "satellite branches". Much of the work in \cite{Cubics} is concerned with "filling in" the rest of the curve, from this initial information. This is also the approach taken by a number of Newton's followers in England, see, for example, the papers \cite{Smart},\cite{Tracing} and the commentary on \cite{Cubics}, in \cite{Talbot}. As one can always choose the "satellite branches" to be non-singular, see Lemma 1.1, one can obtain a solution to "Newton's Theorem" without introducing "Puiseux Series", as, indeed, is done in the previous Lemma 2.5. It is extremely unclear how to give a geometric interpretation of "Puiseux Series". It seems, therefore, to be a rather unfortunate historical accident that these were introduced, as a response to Newton's work, in order to analyse algebraic curves. The proofs that we give in the following section, which develop flash geometry, closely follows Newton's original approach.\\
 \indent I also discuss the general aesthetics behind Newton's representation of algebraic curves in "Christian Geometry" (in preparation). As already mentioned, the reader would benefit greatly from looking at Newton's original "sketches" of cubic curves in \cite{Cubics}.
\end{rmk}

\end{section}

\begin{section}{Flash Geometry of Nodal Curves}
 \begin{defn}

We will denote the coordinates of $P^{2}$ by $\{X,Y,W\}$ or $\{X,Z,W\}$, the coordinates of $P^{3}$ by $\{X,Y,Z,W\}$ and the points $\{[0:0:1:0],[0:1:0:0]\}$ by $\{Q_{1},Q_{2}\}$. We let $\{pr_{1},pr_{2}\}$ be the canonical conic projection morphisms defined by;\\

$pr_{1}:(P^{3}\setminus Q_{1})\rightarrow P_{1}^{2}$; $pr_{1}([X:Y:Z:W])=[X:Y:0:W]$\\

$pr_{2}:(P^{3}\setminus Q_{2})\rightarrow P_{2}^{2}$; $pr_{2}([X:Y:Z:W])=[X:0:Z:W]$\\

We let $A_{1}^{2}\subset P_{1}^{2}$ be the open subset defined by the coordinate system $(x,y)$, where $x={X\over W}$ and $y={Y\over W}$, and $l_{1,\infty}\subset P_{1}^{2}$ be the line at infinity defined by $(W=0)\cap P_{1}^{2}$. Similarly, we let $A_{2}^{2}\subset P_{2}^{2}$ be the open subset defined by the coordinate system $(x,z)$, where $x={X\over W}$ and $z={Z\over W}$, and $l_{2,\infty}\subset P_{2}^{2}$ the line at infinity defined by $(W=0)\cap P_{2}^{2}$. We let $A^{3}\subset P^{3}$ be the open subset defined by the coordinate system $(x,y,z)$, where $x={X\over W}$, $y={Y\over W}$ and $z={Z\over W}$. The restriction of the projection morphisms $\{pr_{1},pr_{2}\}$ are then given in affine coordinates by;\\

$pr_{1}:A^{3}\rightarrow A_{1}^{2}$; $pr_{1}(x,y,z)=(x,y)$\\

$pr_{2}:A^{3}\rightarrow A_{2}^{2}$; $pr_{2}(x,y,z)=(x,z)$\\

We will also refer to the projected points $\{pr_{2}(Q_{1}),pr_{1}(Q_{2})\}$ as $\{Q_{1},Q_{2}\}$, and the restriction of the projections $\{pr_{1},pr_{2}\}$ to $\{A^{2}_{2},A^{2}_{1}\}$ by $pr$.

\end{defn}

We observe the following straightforward property;\\

\begin{lemma}
There does not exist a pair $\{a,b\}$, distinct from $\{Q_{1},Q_{2}\}$, with $\overline{pr_{1}^{-1}(a)}=\overline{pr_{2}^{-1}(b)}$.
\end{lemma}

\begin{proof}

Observe that a fibre of $pr_{1}$ consists of a line passing through $Q_{1}$, with the point $Q_{1}$ removed. Similar considerations apply to $pr_{2}$, $(\dag)$. Suppose a pair, as in the lemma, existed. If $\{a,b\}$ are distinct, then, by $(\dag)$, we would clearly have that $\overline{pr_{1}^{-1}(a)}=\overline{pr_{2}^{-1}(b)}=l_{ab}=l_{Q_{1}Q_{2}}$. As the image of $pr_{1}\circ pr_{2}$ is $P_{1}^{2}\cap P_{2}^{2}$, this implies that $l_{ab}=P_{1}^{2}\cap P_{2}^{2}$. As $P_{1}^{2}\cap P_{2}^{2}$ is fixed both by $pr_{1}$ and $pr_{2}$, $(*)$, it follows that $a=pr_{2}(a)=pr_{1}(b)=b$. Hence, we have that $pr_{1}^{-1}(a)=pr_{2}^{-1}(a)=P_{1}^{2}\cap P_{2}^{2}$, which clearly contradicts $(*)$.

\end{proof}

We now show the following;\\

\begin{lemma}

Let hypotheses be as in Lemma 2.5 and notation as in Definition 3.1, then we can find an irreducible plane projective curve $C'\subset P^{2}$, relative to the coordinate system $(x,z)$, such that;\\

(i). There exists an open $V'\subset C'$, with $(0,0)\in V'$, such that;\\

$pr:(V',(00))\rightarrow (A^{1},0)$, $(*)$\\

is an etale morphism.\\

(ii). $\{\delta_{1}(x),\ldots,\delta_{m}(x)\}\subset R(V')$, for a given embedding of $R(V')$ in $L[[x]]\cap L(x)^{alg}$.\\

(iii). The function field $L(V')$ is generated over $L(x)$ by the power series $\{\delta_{1}(x),\ldots,\delta_{m}(x)\}$.\\

(iv). The cover defined by $(*)$ is Galois, with the extension $L(V')/L(x)$ equal to the Galois closure of the extension $L(C)/L(x)$.

\end{lemma}

\begin{proof}

From the previous construction of Lemma 2.5, the power series $\{\delta_{1}(x),\ldots,\delta_{m}(x)\}$ belong to $L[[x]]\cap L(x)^{alg}$. By Theorem 1.3 and subsequent remarks at the beginning of Section 3 of \cite{depiro5}, we have that $L[[x]]\cap L(x)^{alg}$ is isomorphic to the local ring for the etale topology ${\mathcal O}_{(A^{1},0)}^{et}$.\\
\indent By the definition of this local ring, we can find irreducible varieties;\\

$\{(U_{1},(0)^{lift}_{1}),\ldots,(U_{i},(0)^{lift}_{i}),\ldots,(U_{m},(0)^{lift}_{m})\}$\\

and etale morphisms;\\

$\phi_{i}:(U_{i},(0)^{lift}_{i})\rightarrow (A^{1},0)$, for $1\leq i\leq m$\\

with $\delta_{i}(x)\in R(U_{i})$. It follows, from algebraic considerations, see, for example, \cite{M}, that the fibre product;\\

$(U_{1\ldots m},(0)^{lift})=(U_{1}\times_{A^{1}}\times\ldots\times_{A^{1}}U_{m},(0)^{lift}_{1}\times_{(0)}\ldots\times_{(0)}(0)^{lift}_{m})$\\

is irreducible, nonsingular, and defines etale covers;\\

$\psi_{1\ldots m}:(U_{1\ldots m},(0^{lift}))\rightarrow (A^{1},0)$\\

$\psi_{i}:(U_{1\ldots m},(0^{lift}))\rightarrow (U_{i},(0)^{lift}_{i})$, for $1\leq i\leq m$\\

such that $\phi_{i}\circ\psi_{i}=\psi_{1\ldots m}$. By the definition of the etale local ring, in terms of maps and equivalence relations, we can assume that all the power series $\{\delta_{1}(x),\ldots,\delta_{m}(x)\}$ are represented in the coordinate ring $R(U_{1\ldots m})$.\\
\indent The dominant morphism $\psi_{1\ldots m}$ induces an inclusion;\\

$L(x)\subset L(U_{1\ldots m})$\\

and a corresponding factorisation;\\

$L(x)\subset_{i_{1}}L(x,\delta_{1}(x),\ldots,\delta_{m}(x))\subset_{i_{2}}L(U_{1\ldots m})$,(\footnote{When $char(L)\neq 0$, using the fact that the roots ${\eta_{1}(x),\ldots,\eta_{m}(x)}$, found in Lemma 2.5, are distinct, it is straightforward to check that $L(x,\delta_{1}(x),\ldots,\delta_{m}(x))$ is a seperable extension of $L(x)$.})\\

Let $C''$ be an irreducible nonsingular model of $L(x,\delta_{1}(x),\ldots,\delta_{m}(x))$, with corresponding dominant morphisms;\\

$\alpha_{2}:U_{1\ldots m}\rightarrow C''$,\ $\alpha_{1}:C''\rightarrow A^{1}$\\

corresponding to the inclusions $\{i_{1},i_{2}\}$, such that $\alpha_{1}\circ\alpha_{2}=\psi_{1\ldots m}$. Let $(0)^{lift}\in C''$ also denote the image $\alpha_{2}((0)^{lift})$, for $(0)^{lift}\in U_{1\ldots m}$. As the composition $\alpha_{1}\circ\alpha_{2}$ is etale, it follows, using an elementary Zariski structure argument, that the morphism $\alpha_{1}$ is Zariski unramified at $(0)^{lift}\in C''$. By nonsingularity of $C''$, Footnote 2 and Theorems (6.10,6.11) of \cite{depiro7}, there exists an open subset $W\subset C''$, containing $(0)^{lift}$, such that the restriction;\\

$\alpha_{1}:(W,(0)^{lift})\rightarrow (A^{1},0)$\\

is etale.\\

We now use the local presentation lemma for etale morphisms, given in Fact 1.5 of \cite{depiro5}, that is we can find an open subset $U'\subset A^{1}$, containing $(0)$, and a monic polynomial $F(z)\in R(U')[z]$, such that $\alpha_{1}$ can be presented for an open subset $V'\subset W\subset C''$ in the form;\\

$Spec(({R(U')[z]\over F(z)})_{d})\rightarrow Spec(R(U'))$ $(**)$\\

with $F'(z)$ invertible in $({R(U')[z]\over F(z)})_{d}$. Now let $p(x)\in L[x]$ be a polynomial, with $p(0)\neq 0$, such that $R(U')=L[x]_{p}$. Then we can write the equation for $F(z)$ in the form;\\

$F(z)=z^{n}+{c_{1}(x)\over p^{r_{1}}(x)}z^{n-1}+\ldots+{c_{n-1}(x)\over p^{r_{n-1}}(x)}z+{c_{n}(x)\over p^{r_{n}}(x)}=0$\\

for some $n\geq 1$ and $0\leq r_{j}\leq r$, with $c_{j}(x)$ coprime to $p(x)$ for $1\leq j\leq n$.  Clearing coefficients and using the fact that $V'$ is irreducible, the polynomial;\\

$G(x,z)=p^{r}(x)z^{n}+c_{1}(x)p^{r-r_{1}}(x)z^{n-1}+\ldots+c_{n}(x)p^{r-r_{n}}(x)$\\

defines an irreducible curve $C'$ in $A_{2}^{2}$, whose restriction to $V'=(pr_{1}^{-1}(U')\cap(G=0)\cap (d\neq 0))$ corresponds to the cover defined in $(**)$. By making a vertical translation of the curve, we can assume that $G(0,0)=0$ and $(0,0)$ corresponds to the point $(0)^{lift}$ of $V'\subset W\subset C''$. We finally check the properties $(i)$ to $(iv)$ for $(V',C',(0,0))$.\\

(i). Follows immediately from the presentation $(**)$.\\

(ii). As there exists a birational map between $(V',C'')$ and $(V',C')$ over $(A^{1},0)$, we can assume that the power series $\{\delta_{1}(x),\ldots,\delta_{m}(x)\}\subset R(V')\subset L(C')$. That $\delta_{1}(0,0)=\ldots=\delta_{m}(0,0)=0$, follows from the definition of the maximal ideal $\mathfrak{m}^{et}_{(A^{1},0)}\subset\mathcal{O}^{et}_{(A^{1},0)}$.\\

(iii). Follows from the construction of $C''$ and the fact that $\{C',C''\}$ are birational.\\

(iv). Follows from $(iii)$ and the fact that $L(x,\delta_{1}(x),\ldots,\delta_{m}(x))=L(x,\eta_{1}(x),\ldots,\eta_{m}(x))$ is a splitting field for the polynomial $F(x,y)$, given in Lemma 2.5.\\

\end{proof}

The existence of a curve $C'\subset P^{2}$, satisfying the conditions of the lemma, is far from unique. More precisely, we have the following;\\

\begin{lemma}

Let $C_{1}$ satisfy the conditions of Lemma 3.3 and let;\\

$g(x,z)={a(x)z+b(x)\over c(x)z+d(x)}$\\

with $\{a(x),b(x),c(x),d(x)\}$ polynomials in $L[x]$, such that $b(0)=0$ and $a(0)d(0)\neq 0$\\

Then $g(x,z)$ determines a birational,(\footnote{Unless otherwise stated, when speaking of a birational morphism $\theta$ between curves $C_{1}$ and $C_{2}$, satisfying the conditions of Lemma 3.3, we will always assume that $\theta$ is an isomorphism in the etale topology with $(0,0)$ as labelled points, and $(pr\circ\theta)=pr$. Obvious examples of birational morphisms which do not satisfy this requirement are non-identity elements of the automorphism group $(L(C_{1})/L(x))$. It is not immediately obvious that, for any two curves $C_{1}$ and $C_{2}$, satisfying the conditions of Lemma 3.3, there does exist a birational morphism between $C_{1}$ and $C_{2}$, with the above extra requirement. However, this will be shown later in the paper.}), morphism $\theta_{g}:C_{1}\leftrightsquigarrow C_{2}$, with $C_{2}$ satisfying the conditions of Lemma 3.3.

\end{lemma}

\begin{proof}
As $a(0)d(0)\neq b(0)c(0)$, there exists an open subset $U\subset A^{1}$, containing $0$, such that, for all $x\in U$, $g(x,z)={a(x)z+b(x)\over c(x)z+d(x)}$ defines an invertible Mobius transformation. Let;\\

$V_{1}=\{(x,z)\in U\times A^{1}:c(x)z+d(x)\neq 0\}$\\

$V_{2}=\{(x,z)\in U\times A^{1}:c(x)z-a(x)\neq 0\}$\\

Then the morphism;\\

$\theta_{g}:V_{1}\rightarrow V_{2}$; $(x,z)\mapsto (x,g(x,z))$\\

has an inverse provided by the transformation;\\

$\theta_{g}^{-1}:V_{2}\rightarrow V_{1}$; $(x,z)\mapsto (x,g^{-1}(x,z))$\\

where $g^{-1}$ is the inverse Mobius transformation to $g$. Moreover, we have that $pr_{1}\circ\theta_{g}=Id_{(U\setminus c(x)=0)}$.\\

The image of the irreducible open subset $(C_{1}\cap V_{1})\subset C_{1}$ is constructible, hence, we can form its projective closure $C_{2}$. It follows immediately, that $\theta_{g}$ defines an isomorphism between the open subsets $(C_{1}\cap V_{1})$ and $(C_{2}\cap V_{2})$, $(\dag)$, therefore, defines a birational map $\theta_{g}:C_{1}\leftrightsquigarrow C_{2}$. We check conditions $(i)$ to $(iv)$ for the curve $C_{2}$.\\

(i). We have that $(0,0)\in (V_{1}\cap V_{2})$ and $g(0,0)=0$. By $(\dag)$, $(0,0)$ defines a nonsingular point of the curve $C_{2}$. In order to verify that $pr_{1}$ defines an etale morphism in an open neighborhood of $(0,0)\in C_{2}$, it is then sufficient to verify that the morphism is Zariski unramified at $(0,0)$, by Theorems (6.10,6.11) of \cite{depiro7}. Suppose not, then we can find $\epsilon\in({\mathcal V}_{0}\cap A^{1})$, and distinct $\{\epsilon_{1},\epsilon_{2}\}\subset{\mathcal V}_{0}$ such that $\{(\epsilon,\epsilon_{1}),(\epsilon,\epsilon_{2})\}\subset{\mathcal{V}}_{(0,0)}\cap C_{2}\cap V_{1}\cap V_{2}$. It follows immediately that $\{(\epsilon,g^{-1}(\epsilon,\epsilon_{1})),(\epsilon,g^{-1}(\epsilon,\epsilon_{2}))\}\subset {\mathcal{V}}_{(0,0)}\cap C_{1}\cap V_{1}\cap V_{2}$. Moreover, the pair $\{(\epsilon,g^{-1}(\epsilon,\epsilon_{1})),(\epsilon,g^{-1}(\epsilon,\epsilon_{2}))\}$ is distinct, by birationality. This contradicts the fact that $pr_{1}$ defines an etale morphism in an open neighborhood of $(0,0)\in C_{1}$.\\

The remaining conditions $(ii),(iii),(iv)$ are straightforward to check, using the fact that;\\

$\theta_{g}:(C_{1}\cap V_{1},(0,0))\rightarrow (C_{2}\cap V_{2},(0,0))$\\

is an isomorphism, with the property that $(pr_{1}\circ\theta_{g})=pr_{1}$.

\end{proof}

The verticalised, polar geometry of such birational transformations, suggests that we should give special consideration to the effect on branches, centred at the point $Q_{1}$, determined by the coordinate system $(x,z)$;\\ \\

\begin{lemma}{Effect on hyperbolic branches at $Q_{1}$}\\

Let $C_{1},C_{2}$ and $\theta_{g}$ satisfy the conditions of the previous lemma and suppose that $\gamma$ is a hyperbolic branch of $C_{1}$, centred at $Q_{1}$, with tangent line $x=\alpha$. Then, the corresponding branch $\theta_{g}(\gamma)$ of $C_{2}$, determined by the birational map $\theta_{g}$, is centred at $Q_{1}$, if $c(\alpha)=0$ and $a(\alpha)\neq 0$, and at $(\alpha,{a(\alpha)\over c(\alpha)})$, if $c(\alpha)\neq 0$.\\

\end{lemma}

\begin{proof}

As $x=\alpha$ is tangent to the branch $\gamma$, and not equal to $l_{\infty}$, by results of \cite{depiro6}, given an infinitesimal $\epsilon\in\mathcal{V}_{0}$, there exists $z(\epsilon)$ such that;\\

$(\alpha+\epsilon,z(\epsilon))\in({\mathcal V}_{Q_{1}}\cap C_{1}\cap \gamma)$\\

By the explicit construction of specialisations, given in \cite{depiro3}, we may assume that $z(\epsilon)={u(\epsilon)\over\epsilon^{i}}$, for a unit $u(x)\in L[[x^{1\over r}]]$, $(r\geq 1)$, and $i>0$, $i\in{\mathcal Q}$, (\footnote{The use of Puiseux series, here, is necessary, due to the existence of blue branches (Definition 3.11). We will find a more effective method of understanding the geometry of such branches, later in the paper.})\\

If $\theta_{g}(\gamma)$ is centred at $P$, then, using results of \cite{depiro6};\\

 $\theta_{g}((\alpha+\epsilon,z(\epsilon)))\in ({\mathcal V}_{P}\cap C_{2}\cap \theta_{g}(\gamma))$\\

We explicitly determine $P$ by specialisation in the following cases;\\

Case 1. $c(\alpha)=0$ and $a(\alpha)\neq 0$\\

We have that;\\

$g(\alpha+\epsilon,z(\epsilon))={a(\alpha+\epsilon)z(\epsilon)+b(\alpha+\epsilon)\over c(\alpha+\epsilon)z(\epsilon)+d(\alpha+\epsilon)}={a(\alpha+\epsilon)u(\epsilon)+b(\alpha+\epsilon)\epsilon^{i}\over c(\alpha+\epsilon)u(\epsilon)+d(\alpha+\epsilon)\epsilon^{i}}$\\

hence, the projective coordinates of $\theta_{g}(\epsilon,z(\epsilon))$, see Definition 3.1, are given by;\\

$X=(\alpha+\epsilon)(c(\alpha+\epsilon)u(\epsilon)+d(\alpha+\epsilon)\epsilon^{i})$\\

$Z=a(\alpha+\epsilon)u(\epsilon)+b(\alpha+\epsilon)\epsilon^{i}$\\

$W=c(\alpha+\epsilon)u(\epsilon)+d(\alpha+\epsilon)\epsilon^{i}$\\

Using the fact that $c(\alpha)=0$ and $a(\alpha)\neq 0$, the specialised coordinates are given by;\\

$X=0$,\indent $Z=1$,\indent $W=0$\\

hence, the result follows.\\

Case 2. $c(\alpha)\neq 0$.\\

Writing $z(\epsilon)={1\over t(\epsilon)}$, where $t(\epsilon)=\epsilon^{i}u(\epsilon)$, we have that;\\

$\theta_{g}(\alpha+\epsilon,z(\epsilon))=(\alpha+\epsilon, {a(\alpha+\epsilon)+b(\alpha+\epsilon)t(\epsilon)\over c(\alpha+\epsilon)+d(\alpha+\epsilon)t(\epsilon)})$\\

which, clearly specialises to $(\alpha,{a(\alpha)\over c(\alpha)})$, using the fact that $c(\alpha)\neq 0$.\\

The lemma is then shown.

\end{proof}

\begin{lemma}{Effect on parabolic branches at $Q_{1}$}\\

Let $C_{1},C_{2}$ and $\theta_{g}$ satisfy the conditions of the previous lemma and suppose that $\gamma$ is a parabolic branch of $C_{1}$, centred at $Q_{1}$. Then the corresponding branch $\theta_{g}(\gamma)$ is centred at;\\

$[1:0:0]$ if $max\{ord(a(x)),ord(b(x))\}<<min\{ord(c(x)),ord(d(x))\}$\\

$[0:1:0]$ if $max\{ord(c(x)),ord(d(x))\}<<min\{ord(a(x)),ord(b(x))\}$\\

\end{lemma}

\begin{proof}

As $\gamma$ is parabolic and centred at $Q_{1}$, given an infinitesimal $\epsilon\in{\mathcal V}_{0}$, there exists $z(\epsilon)$ such that $({1\over\epsilon},z(\epsilon))\in({\mathcal V}_{Q_{1}}\cap C_{1}\cap\gamma)$. By a straightforward calculation, we can assume that $z(\epsilon)$ is of the form ${u(\epsilon)\over\epsilon^{i}}$, for a unit $u(x)\in L[[x^{1\over r}]]$, $r\geq 1$, and $1<i\leq N$, $i\in{\mathcal Q}$, where $N$ depends only on the degree of $C_{1}$. We have that;\\

$g({1\over\epsilon},z(\epsilon))={a({1\over\epsilon})u(\epsilon)+b({1\over\epsilon})\epsilon^{i}\over c({1\over\epsilon})u(\epsilon)+d({1\over\epsilon})\epsilon^{i}}=q(\epsilon)$\\

and\\

$\theta_{g}({1\over\epsilon},z(\epsilon))=[{1\over\epsilon}:q(\epsilon):1]=[1:\epsilon q(\epsilon):\epsilon]$, $(\star)$\\

In the case that $ord(q(\epsilon))>-1$, the specialisation, determined from $(\star)$, is given by $[1:0:0]$, and, in the case that $ord(q(\epsilon))< -1$, the specialisation is given by $[0:1:0]$. By inspection of $q(\epsilon)$, these conditions are clearly met, given the corresponding conditions of the lemma.

\end{proof}

It is convenient, in what follows, to consider the effect of the following transformation;\\

\begin{lemma}

Let $C_{1}$ satisfy the conditions of the previous lemma, and let $\alpha\neq 0$ be chosen so that $x=\alpha$ does not correspond to the tangent line of any hyperbolic branch of $C_{1}$, centred at $Q_{1}$. Then the transformation;\\

$\theta_{g}:(x,z)\mapsto (x,(x-\alpha)z)$\\

has the following properties;\\

(i). $\theta_{g}$ defines a birational map $\theta_{g}:C_{1}\leftrightsquigarrow C_{2}$, with $C_{2}$ satisfying the conditions of Lemma 3.3.\\

(ii). If $\gamma$ is a branch of $C_{1}$, centred along $({l_{\infty}\setminus [1:0:0]})$, then $\theta_{g}(\gamma)$ is centred at $Q_{1}\in C_{2}$.\\

Moreover,\\

(iii). If $\gamma$ is a branch of $C_{1}$, centred at $[1:0:0]$, then, after finitely many repetitions of transformations, having the form $\theta_{g}$, we obtain a birational map $\theta:C_{1}\leftrightsquigarrow C_{2}$, such that $\theta(\gamma)$ is centred at $Q_{1}\in C_{2}$. \\

(iv). In particular, $C_{2}$ intersects the line $l_{\infty}$ only at $Q_{1}$.\\

\end{lemma}

\begin{proof}

The property $(i)$ follows immediately from Lemma 3.4 and the choice of $\alpha\neq 0$. By explicit calculation, the projective formulation of $\theta_{g}$ is given by;\\

$\theta_{g}:[X:Z:W]\mapsto [XW:XZ-\alpha WZ:W^{2}]$ $(\star\star)$\\

Any point $P$, lying on $(l_{\infty}\setminus (Q_{1}\cup [1:0:0])$, satisfies $W=0$ and $XZ\neq 0$, hence, by $(\star\star)$, is
mapped to $Q_{1}$. In particular, if $\gamma$ is a branch of $C_{1}$, centred along $(l_{\infty}\setminus (Q_{1}\cup [1:0:0])$, then $\theta_{g}(\gamma)$ is centred at $Q_{1}$. By the choice of $\alpha$ in relation to tangent lines, and Lemma 3.5, the hyperbolic branches of $C_{1}$, centred at $Q_{1}$, are fixed. One finds the effect on the parabolic branches of $C_{1}$, centred at $Q_{1}$, by explicit calculation. Using the notation of Lemma 3.6, for $z(\epsilon)={u(\epsilon)\over \epsilon^{i}}$, with $i>1$, we have that;\\

$\theta_{g}({1\over\epsilon}, z(\epsilon))=[\epsilon^{i}:u(\epsilon)-\alpha u(\epsilon)\epsilon:\epsilon^{i+1}]$, $(\star\star\star)$\\

Clearly, the specialisation of this point is $Q_{1}=[0:1:0]$. Hence, the parabolic branches of $C_{1}$, centred at $Q_{1}$, are also fixed. $(ii)$ is then shown.\\

Again, we find the effect on a branch $\gamma$, centred at $[1:0:0]$, by explicit calculation. Using the notation of Lemma 3.6, for $z(\epsilon)={u(\epsilon)\over \epsilon^{i}}$, with $i<1$, by considering $(\star\star\star)$, we have that $\theta_{g}(\gamma)$ is centred along $(l_{\infty}\setminus [1:0:0])$ iff $0\leq i<1$. Otherwise, we obtain a new $z_{1}(\epsilon)={u_{1}(\epsilon)\over \epsilon^{i_{1}}}$, with $i_{1}=i+1$. Repeating the calculation, for new transformations of the form $\theta_{g}$, satisfying the previous conditions of the lemma, we are guaranteed that this process eventually produces a birational map $\theta:C_{1}\leftrightsquigarrow C_{2}$, with $\theta(\gamma)$ centred along $(l_{\infty}\setminus [1:0:0])$. Using the previous properties $(i)$ and $(ii)$, we can assume that    $\theta(\gamma)$ is centred at $Q_{1}$. Hence, $(iii)$ is shown.\\

The final statement follows straightforwardly from $(i)$, $(ii)$ and $(iii)$.\\

\end{proof}

We now consider the following method of moving branches in finite position, not situated along the axis $x=0$, to $Q_{1}$;\\

\begin{lemma}
Let $C_{1},C_{2},g$ and $\theta_{g}$ satisfy the conditions of Lemma 3.4, with the additional property that, for a given $\alpha\neq 0$, $c(\alpha)=d(\alpha)=0$, $a(\alpha)\neq 0$ and $a(\alpha)\beta+b(\alpha)\neq 0$, for any of the finitely many branches of $C_{1}$ in finite position, centred along the axis $x=\alpha$, with coordinates $(\alpha,\beta)$. Then, if $\gamma$ is a branch of $C_{1}$, centred in finite position along the axis $x=\alpha$, $\theta_{g}(\gamma)$ is centred at $Q_{1}\in C_{2}$, and, any hyperbolic branch of $C_{1}$, centred at $Q_{1}$ with tangent line $x=\alpha$, is fixed.
\end{lemma}

\begin{proof}

The case when $\gamma$ is centred in finite position, with coordinates $(\alpha,\beta)$, such that $a(\alpha)\beta+b(\alpha)\neq 0$ follows by a direct calculation. The statement, concerning hyperbolic branches, follows immediately from Lemma 3.5.
\end{proof}

As a consequence of the previous arguments, we have the following;\\

\begin{lemma}

There exists an irreducible plane projective curve $C'\subset P^{2}$, satisfying the conditions of Lemma 3.3, with the additional properties;\\

(i). $(C'\cap l_{\infty})=Q_{1}$.\\

(ii). All the singularities of $C'$ lie along the axis $x=0$.\\

\end{lemma}

\begin{proof}

By Lemma 3.7(ii), we can assume there exists a curve $C_{1}$, satisfying the conditions of Lemma 3.3, with the additional property that $(C_{1}\cap l_{\infty})=\{Q_{1}\}$. In particular, all the singularities of $({C_{1}\setminus\{Q_{1}\}})$ are in finite position. Let $A=\{\alpha_{1},\ldots,\alpha_{j},\ldots,\alpha_{r}\}$ enumerate the axes $x=\alpha_{j}$, containing the singularities in finite position, away from $x=0$. Let  $B=\{\alpha_{r+1},\ldots,\alpha_{k},\ldots,\alpha_{s}\}$ enumerate the axes $x=\alpha_{k}$, corresponding to tangent lines of hyperbolic branches, centred at $Q_{1}$, away from $x=0$, not contained in $A$. Let $C_{j}=\{\beta_{j1},\ldots,\beta_{ji},\ldots,\beta_{jm}\}$, for $1\leq j\leq r$, $1\leq i\leq m$, where $deg(C_{1})=m$,  enumerate the $z$-coordinates of the intersections in finite position of $C_{1}$ with the axis $x=\alpha_{j}$.  It is an elementary exercise to show, using the fact that $0\notin (A\cup B)$, that there exist polynomials $\{a(x),b(x),c(x),d(x)\}\in L[x]$, with the properties;\\

(a). $a(\alpha_{1})\ldots a(\alpha_{r})\ldots a(\alpha_{s})\neq 0$.\\

(b). $c(\alpha_{1})=\ldots=c(\alpha_{r})=\ldots=c(\alpha_{s})=0$.\\

(c). $d(\alpha_{1})=\ldots=d(\alpha_{r})=0$.\\

(d). $b(\alpha_{j})\neq -a(\alpha_{j})\beta_{ji}$, $1\leq j\leq r$, $1\leq i\leq m$.\\

(e). $a(0)d(0)\neq 0$ and $b(0)=0$.\\

(f). $max\{ord(c(x)),ord(d(x))\}<<min\{ord(a(x)),ord(b(x))\}$.\\

By (e) and Lemma 3.4, these polynomials determine a birational morphism $\theta_{g}:C_{1}\leftrightsquigarrow C_{2}$, such that $C_{2}$ also satisfies the conditions of Lemma 3.3. By (a),(b),(c),(d) and Lemma 3.8, all the singularities of $C_{1}$, in finite position, away from the axis $x=0$, are moved to $Q_{1}\in C_{2}$. By (a),(b) and Lemma 3.5, the position of all the hyperbolic branches of $C_{1}$ at $Q_{1}$ is fixed. By (f) and Lemma 3.6, the position of all the parabolic branches of $C_{1}$ at $Q_{1}$ is fixed. Denoting $C_{2}$ by $C'$, it is then clear that conditions $(i)$ and $(ii)$ of the lemma are satisfied for $C'$.

\end{proof}

We recall from \cite{depiro5} that the notion of a branch is invariant under birational transformations. The following result allows us to move any finite number of branches away from the critical point $Q_{1}$.\\

\begin{lemma}

Let $C_{1}$ satisfy the conditions of Lemma 3.3 and let $\{\gamma_{1},\ldots,\gamma_{t}\}$ enumerate a finite subset of its branches. Then there exists a birational map $\theta_{g}:C_{1}\leftrightsquigarrow C_{2}$, with $C_{2}$ also satisfying the conditions of Lemma 3.3, such that the corresponding branches $\{\theta_{g}(\gamma_{1}),\ldots,\theta_{g}(\gamma_{t})\}$ are either in finite position or centred at the point $[1:0:0]$ along the axis $l_{\infty}$. In particular, none of the branches are centred at the point $Q_{1}$.

\end{lemma}

\begin{proof}

By birationality, we can assume that $C_{1}$ satisfies the conditions of Lemma 3.9. Given the above enumeration, let;\\

(i). $\{\gamma_{1},\ldots,\gamma_{r}\}$ consist of those in finite position.\\

(ii). $\{\gamma_{r+1},\ldots,\gamma_{s}\}$ consist of the hyperbolic branches, centred at $Q_{1}$.\\

(iii). $\{\gamma_{s+1},\ldots,\gamma_{t}\}$ consist of the parabolic branches, centred at $Q_{1}$.\\

Let $\{(\alpha_{1},\beta_{1}),\ldots,(\alpha_{j},\beta_{j}),\ldots,(\alpha_{r},\beta_{r})\}$ enumerate the coordinates of the branches given in $(i)$. Let $\{\delta_{r+1},\ldots,\delta_{k},\ldots,\delta_{s}\}$ enumerate the axes $x=\delta_{k}$, appearing as tangent lines to the branches given in $(ii)$. It is an elementary exercise to show that there exist polynomials $\{a(x),b(x),c(x),d(x)\}\in L[x]$, with the properties;\\

(a). $c(\alpha_{1})\ldots c(\alpha_{j})\ldots c(\alpha_{r})c(\delta_{r+1})\ldots c(\delta_{k})\ldots c(\delta_{s})\neq 0$.\\

(b). $a(\alpha_{j})\beta_{j}+b(\alpha_{j})=\beta_{j}$, $c(\alpha_{j})\beta_{j}+d(\alpha_{j})=1$, $(1\leq j\leq r)$\\

(c). $a(0)d(0)\neq 0$ and $b(0)=0$.\\

(d).$max\{ord(a(x)),ord(b(x))\}<<min\{ord(c(x)),ord(d(x))\}$.\\

By $(c)$ and Lemma 3.4, there exists a corresponding birational map $\theta_{g}:C_{1}\leftrightsquigarrow C_{2}$, with $C_{2}$ also satisfying the conditions of Lemma 3.3. By $(b)$, the branches $\{\theta_{g}(\gamma_{1}),\ldots,\theta_{g}(\gamma_{r})\}$ are all centred at the same points as $\{\gamma_{1},\ldots,\gamma_{r}\}$ respectively. In particular, the branches from $(i)$ remain in finite position. By $(a)$ and Lemma 3.5, the branches $\{\theta_{g}(\gamma_{r+1}),\ldots,\theta_{g}(\gamma_{s})\}$ are centred at the points in finite position;\\

$\{(\delta_{r+1},{a(\delta_{r+1})\over c(\delta_{r+1})}),\ldots,(\delta_{s},{a(\delta_{s})\over c(\delta_{s})})\}$\\

\noindent respectively. In particular, the branches from $(ii)$ are moved to finite position. By $(d)$ and Lemma 3.6, the branches $\{\theta_{g}(\gamma_{s+1}),\ldots,\theta_{g}(\gamma_{t})\}$ are centred at the point $[1:0:0]$. Hence, the lemma is shown.

\end{proof}

We now make the following definition;\\

\begin{defn}
Fixing an enumeration of branches for some curve $C'$, satisfying the conditions of Lemma 3.3, we define a branch $\gamma$ to be blue, if there exists a birational map $\theta:C\leftrightsquigarrow C_{1}$, with the curve $C_{1}$ also satisfying the conditions of Lemma 3.3, such that;\\

Either\\

(i). $\theta(\gamma)$ is centred in finite position $(a,b)$, with $I_{\gamma}(C_{1},x=a)\geq 2$.\\

Or\\

(ii). $\theta(\gamma)$ is centred at $[1:0:0]$, with $I_{\gamma}(C_{1},l_{\infty})\geq 2$.\\

where $I_{\gamma}$ is defined as in \cite{depiro6}. We define a branch $\gamma$ to be silver otherwise.

\end{defn}

We observe the following;\\

\begin{rmk}

(a). The choice of colours is partly motivated by consideration of the optical spectrum of visible light, as well as theoretical views of the author on chemical processes that occur in the production of light by stars. I hope to explain my views on the connections between geometry and a 4-fold model of the theory of light in a forthcoming paper.\\

(b). On a geometrical level, the above construction is well-defined. For suppose that $C_{1}$ and $C_{2}$ are both curves, for which Definition 3.11 is made, relative to $C'$. Then there exists a birational map $\theta_{1}:C_{1}\leftrightsquigarrow C_{2}$, for which $(pr_{1}\circ\theta_{1})=pr_{1}$, $(\star)$. It is a simple exercise to check, by $(\star)$, the definition of $I_{\gamma}$ using infinitesimals, see \cite{depiro6}, and birationality of the map $\theta_{1}$, that, if $\gamma$ is a branch of $C_{1}$, centred at $(a,b)$, with $\theta(\gamma)$ centred at $(a,b')$, then $I_{\gamma}(C_{1},x=a)=I_{\theta(\gamma)}(C_{2},x=a)$. Similarly, if $\gamma$ is a branch of $C_{1}$, centred at $[1:0:0]$, with $\theta(\gamma)$ centred at $[1:0:0]$, then $I_{\gamma}(C_{1},l_{\infty})=I_{\theta(\gamma)}(C_{2},l_{\infty})$.\\

(c). Using Lemmas 3.5, 3.6 and 3.9, any branch $\gamma$ of $C'$ may be centred in finite position or at $[1:0:0]$, by an appropriate birational map $\theta$. Hence, we can give an equivalent characterisation of a silver branch as in Definition 3.11, replacing the either/or clause by;\\

Either\\

(i). $\theta(\gamma)$ is centred in finite position $(a,b)$, with $I_{\gamma}(C_{1},x=a)=1$.\\

Or\\

(ii). $\theta(\gamma)$ is centred at $[1:0:0]$, with $I_{\gamma}(C_{1},l_{\infty})=1$.\\

On an intuitive level, this means that a silver branch is characterised by the condition that it is non-singular and transverse to a vertical axis or to the line $l_{\infty}$, when centred away from $Q_{1}$. Similarly, a blue branch is characterised by the condition that it is either singular or tangent to a vertical axis/the line $l_{\infty}$, when centred away from $Q_{1}$.\\

(d). There exist infinitely many silver branches and finitely many blue branches belonging to $C'$. This follows immediately from the above characterisation and the fact that, there exists a unique non-singular branch, transverse to the axis $x=a$, for a generic point $(a,b)$ belonging to $C'$. Clearly, there exist finitely many silver branches and finitely many blue branches, belonging to $C'$, centred along the axis $x=0$.\\

\end{rmk}

We now show the following;\\

\begin{lemma}{Uniformity of branches}\\

Let $C'$ satisfy the conditions of Lemma 3.3. Then the distribution of silver and blue branches along any axis is uniform. More precisely;\\

(i). If $x=\alpha$ is a vertical axis, then, either, every branch centred in finite position along $x=\alpha$ is silver, or, every branch centred in finite position along $x=\alpha$ is blue.\\

(ii). Either, every branch centred along ${l_{\infty}\setminus\{Q_{1}\}}$ is silver, or, every branch centred along
${l_{\infty}\setminus\{Q_{1}\}}$ is blue.\\

(iii). If $x=\alpha$ is a vertical axis, then, either, every hyperbolic branch, centred at $Q_{1}$, tangent to the line $x=\alpha$, is silver, or, every hyperbolic branch, centred at $Q_{1}$, tangent to the line $x=\alpha$, is blue.\\

(iv). Either, every parabolic branch, centred at $Q_{1}$, is silver, or, every parabolic branch, centred at $Q_{1}$, is blue.\\

\end{lemma}

\begin{proof}

It is elementary that, for generic $\delta$, the curve $C'$ intersects the axis $x=\delta$ transversely at $t$ distinct points in finite position, where $t=deg(L(C')/L(x))$. As, by Lemma 3.3(iv), the extension $L(C')/L(x)$ is Galois, we have that $|G|=t$, where $G=Gal(L(C')/L(x))$, and $G$ acts transitively on the fibre ${(C'\cap (x=\alpha))\setminus\{Q_{1}\} }$, $(\star)$. As this condition is definable, we can find an open subset $U\subset A^{1}$, with the property that, for any $\delta\in U$, $(\star)$ holds for the fibre $x=\delta$. Suppose that there exists a silver branch $\gamma_{1}$ and a blue branch $\gamma_{2}$ situated on the axis $x=\alpha$, centred at the coordinates $(\alpha,\beta_{1})$ and $(\alpha,\beta_{2})$ respectively. Choose an infinitesimal $\epsilon\in{\mathcal V}_{0}$, such that $\alpha+\epsilon\in U$, and $(\alpha+\epsilon,\beta_{1}')\in\gamma_{1},(\alpha+\epsilon,\beta_{2}')\in\gamma_{2}$. As $(\star)$ holds for the axis $x=\alpha+\epsilon$, we can find a birational morphism $\theta_{h}:C'\leftrightsquigarrow C'$, representing the action of $h\in G$, with $pr\circ\theta_{h}=pr$, and $\theta_{h}(\alpha+\epsilon,\beta_{1}')=(\alpha+\epsilon,\beta_{2}')$. It follows, by construction, that $\theta_{h}(\gamma_{1})=\gamma_{2}$. However, by previous remarks 3.12(b), we have that $I_{\theta_{h}(\gamma_{1})}(C',x=\alpha)=I_{\gamma_{1}}(C',x=\alpha)=1$, contradicting the fact that $\theta_{h}(\gamma_{1})=\gamma_{2}$ is a blue branch. Hence, $(i)$ is shown. In order to show $(ii)$, one should employ a similar argument, using the fact that, for $\epsilon\in{\mathcal V}_{0}$, with $\epsilon\neq 0$, $x={1\over\epsilon}$ lies in $U$. For $(iii)$ and $(iv)$, one can use Lemmas 3.5, Lemmas 3.6 and Definition 3.11, to reduce to the cases $(i)$ and $(ii)$ respectively.\\

\end{proof}

\begin{lemma}{Separation of silver branches}\\

Let $C'$ satisfy the conditions of Lemma 3.3, then any silver branch $\gamma$ of $C'$, not centred along the axis $x=0$, (\footnote{We follow the convention that a branch, centred along $x=0$, includes hyperbolic branches, centred at $Q_{1}$, with tangent line $x=0$.}), can be isolated away from $Q_{1}$. That is, there exists a birational map $\theta:C'\leftrightsquigarrow C_{1}$, with $C_{1}$ also satisfying the conditions of Lemma 3.3, such that;\\

Either (i). $\theta(\gamma)$ is centred at a point $(\alpha,\beta)$ in finite position, with $\alpha\neq 0$, and is the unique branch of $C_{1}$ to be centred at this point.\\

Or (ii). $\theta(\gamma)$ is centred at a point of $(l_{\infty}\setminus\{Q_{1}\})$, and is the unique branch of $C_{1}$ to be centred at this point.\\

Moreover, given any finite set $\{\gamma_{1},\ldots,\gamma_{r}\}$ of silver branches, centred at $Q_{1}$, with distinct tangent directions, not including $x=0$, there exists a birational map $\theta:C'\leftrightsquigarrow C_{1}$, with $C_{1}$ also satifying the conditions of Lemma 3.3, such that $\{\theta(\gamma_{1}),\ldots,\theta(\gamma_{r})\}$ are all isolated away from $Q_{1}$.

\end{lemma}

\begin{proof}
By Lemmas 3.5 and 3.6, we can assume that the silver branch $\gamma$ is centred either in finite position or at $[1:0:0]$.\\

 Case 1. Assume that $\gamma$ is centred at $(\alpha,\beta)$, with $\alpha\neq 0$.\\

By Lemma 3.13, all the branches situated along the axis $x=\alpha$ are silver. In particular, all the branches centred at $(\alpha,\beta)$ are silver. Let $\{\gamma,\gamma_{1},\ldots,\gamma_{i},\ldots,\gamma_{r}\}$ be an enumeration of these branches. By results of \cite{depiro6}, see also Remarks 2.6, we can find parametrisations of the form;\\

$\{(\alpha+x,\beta+\eta(x)),(\alpha+x,\beta+\eta_{1}(x)),\ldots,(\alpha+x,\beta+\eta_{r}(x))\}$\\

for this enumeration, where $\{\eta(x),\eta_{1}(x),\ldots,\eta_{i}(x),\ldots,\eta_{r}(x)\}$ are distinct power series, belonging to $L[[x]]$, of the form;\\

$\eta(x)=\sum_{j=1}^{\infty}c_{j}x^{j}$, $\eta_{i}(x)=\sum_{j=1}^{\infty}c_{ij}x^{j}$, $1\leq i\leq r$\\

Let $g(x,z)$ be defined by;\\

 $g(x,z)={(z-\beta)\over (x-\alpha)}-{\beta\over\alpha}={(\alpha z-\beta x)\over \alpha(x-\alpha)}$, $(\star)$\\

We have that $\alpha.\alpha(x-\alpha)|_{(0,0)}=-\alpha^{3}\neq 0$ and $(-\beta x)|_{(0,0)}=0$, hence, by Lemma 3.4, $g(x,z)$ defines a birational map $\theta_{g}:C'\leftrightsquigarrow C''$, with $C''$ also satisfying the conditions of Lemma 3.3. Using Lemma 3.8, if $\gamma'$ is a branch of $C'$ in finite position, along the axis $x=\alpha$, not at $(\alpha,\beta)$, then $\theta_{g}(\gamma)$ is centred at the point $Q_{1}\in C''$. By Lemma 3.5, if $\gamma'$ is a hyperbolic branch of $C'$, centred at $Q_{1}$, with tangent line $x=\alpha$, then $\theta(\gamma')$ is also centred at $Q_{1}\in C''$. It follows that the only branches, in finite position along the axis $x=\alpha$ of $C''$, are $\{\theta_{g}(\gamma),\theta_{g}(\gamma_{1}),\ldots,\theta_{g}(\gamma_{r})\}$. We now consider the positions of these branches;\\

By results of \cite{depiro6} and the definition of $g(x,z)$, the branches\\ $\{\theta_{g}(\gamma),\theta_{g}(\gamma_{1}),\ldots,\theta_{g}(\gamma_{r})\}$ are parametrised by;\\

 $\{(\alpha+x,\delta(x)-{\beta\over\alpha}),(\alpha+x,\delta_{1}(x)-{\beta\over\alpha}),\ldots,(\alpha+x,\delta_{r}(x)-{\beta\over\alpha})\}$\\

where $\eta(x)=\alpha+x\delta(x)$ and $\eta_{i}(x)=\alpha+x\delta_{i}(x)$, for $1\leq i\leq r$. It follows that $\theta_{g}(\gamma)$ is centred in finite position, at $(\alpha,\beta_{1})$, where $\beta_{1}$ depends only on the first coefficient $c_{1}$ of the power series $\eta(x)$. In particularly, a branch $\theta_{g}(\gamma_{i})$, for $1\leq i\leq r$, is also centred at $(\alpha,\beta_{1})$, if and only if, $c_{i1}=c_{1}$. We then repeat the above construction for the new centre $(\alpha,\beta_{1})$, obtaining a new birational map $\theta_{g_{1}}$. The only branches remaining in finite position along the axis $x=\alpha$ after the composition $(\theta_{g_{1}}\circ \theta_{g})$, being those for which $c_{i1}=c_{1}$, and, the only branches whose centre coincides with that of $(\theta_{g_{1}}\circ \theta_{g})(\gamma)$, being those for which $c_{i1}=c_{1}$ and $c_{i2}=c_{2}$. As the power series $\eta(x)$ is distinct from $\{\eta_{1}(x),\ldots,\eta_{r}(x)\}$, clearly the construction terminates after a finite number of steps, and, we obtain a curve $C_{1}$, satisfying the conditions of Lemma 3.3, together with a birational map $\theta:C\leftrightsquigarrow C_{1}$, such that $\theta(\gamma)$ is the only branch, centred in finite position, along the axis $x=\alpha$. Hence, (i) is shown for this case.\\

Case 2. Assume that $\gamma$ is centred at $[1:0:0]$.\\

Choosing $\alpha\neq 0$, we consider the rotation $\nu$ about $Q_{1}$,  defined by;\\

$\nu:(x,z)\mapsto ({1\over x+\alpha}-{1\over\alpha},z)=({-x\over \alpha(x+\alpha)},z)$\\

By construction, we have that $\nu(0,0)=(0,0)$ and $\nu$ fixes the axis $x=0$. It is clear that $\nu$ determines a birational morphism $\theta_{\nu}:C'\leftrightsquigarrow C''$, which is etale in an open neighborhood of $(0,0)$, (\footnote{This is not a morphism in the sense of Lemma 3.4 and footnote 3, as, clearly, $(pr\circ\theta)\neq pr$}).
It is a straightforward exercise, using infinitesimals, and the fact that $\nu$ permutes the vertical axes, to check that $\theta_{\nu}$ interchanges;\\

 (a). The silver branches of $C'$, centred along $({l_{\infty}\setminus([1:0:0]\cup Q_{1})})$, and the silver parabolic branches of $C'$, centred at $Q_{1}$, with the silver hyperbolic branches of $C''$, centred at $Q_{1}$, with tangent line $x={-{1\over\alpha}}$.\\

 (b). The silver branches of $C'$, centred at $[1:0:0]$, with the silver branches of $C''$, centred in finite position along the axis $x=-{1\over\alpha}$. $(\star\star)$\\

We now apply Case 1, to obtain a birational morphism $\theta_{1}:C''\leftrightsquigarrow C'''$, in the stronger sense of Lemma 3.4, such that the silver branch $\theta_{\nu}(\gamma)$ is isolated at $(-{1\over\alpha},0)$, that is $(\theta_{1}\circ\theta_{\nu})(\gamma)$ is the unique branch of $C'''$ to be centred at $(-{1\over\alpha},0)$, $(\star\star\star)$.\\

 Then, we apply the inverse rotation ${\nu}^{-1}$ about $Q_{1}$, defined by;\\

 ${\nu}^{-1}:(x,z)\mapsto ({1\over x+{1\over \alpha}}-\alpha,z)=({-\alpha^{2}x\over\alpha x+1},z)$\\

which determines a birational map $\theta_{{\nu}^{-1}}:C'''\leftrightsquigarrow C''''$. It is simple to check that the composition $(\theta_{{\nu}^{-1}}\circ\theta_{1}\circ\theta_{\nu}):C'\leftrightsquigarrow C''''$ is a birational map, in the stronger sense of Lemma 3.4. Using $(\star\star)$, we have that $(\theta_{\nu^{-1}}\circ\theta_{1}\circ\theta_{\nu})(\gamma)$ is centred at $[1:0:0]$. If $({-1\over\alpha}+x,\delta(x))$, with $\delta(x)\in L[[x]]$ and $ord(\delta(x))\geq 1$, is a parametrisation of the silver branch $\theta_{1}\circ\theta_{\nu})(\gamma)$, then the corresponding parametrisation of $(\theta_{\nu^{-1}}\circ\theta_{1}\circ\theta_{\nu})(\gamma)$ is given by $({1\over x}-\alpha,\delta(x))$. Moreover, using $(\star\star)(b)$ and $(\star\star\star)$, if $({1\over x}-\alpha,\delta_{1}(x))$ is a parametrisation of any other silver branch of $C''''$, centred at $[1:0:0]$, then we must have that $ord(\delta_{1}(x))=0$, $(\star\star\star\star)$. We, finally, apply a transformation of the form $\theta_{g}:C''''\leftrightsquigarrow C'''''$, considered in Lemma 3.7. By $(\star\star\star\star)$, it is elementary to check, using infinitesimals, that the branch $(\theta_{g}\circ\theta_{\nu^{-1}}\circ\theta_{1}\circ\theta_{\nu})(\gamma)$ is fixed at $[1:0:0]$ and the other branches of $C''''$, centred at $[1:0:0]$, are moved to positions along the axis $(l_{\infty}\setminus{(Q_{1}\cup[1:0:0])})$, see also the proof of Lemma 3.7. By the properties of $\theta_{g}$, considered in Lemma 3.7, $(\theta_{g}\circ\theta_{\nu^{-1}}\circ\theta_{1}\circ\theta_{\nu})(\gamma)$ is the unique branch of $C'''''$ to be centred at $[1:0:0]$. Letting $\theta=(\theta_{g}\circ\theta_{\nu^{-1}}\circ\theta_{1}\circ\theta_{\nu})$ and $C_{1}=C'''''$, we have that $(ii)$ is shown for this case.\\

For the final part of the lemma, assume, first, that all of the branches $\{\gamma_{1},\ldots,\gamma_{r}\}$ are hyperbolic. By Lemma 3.10, we can assume that they are all in finite position. Let $\{\alpha_{1},\ldots,\alpha_{r}\}$ enumerate the distinct $x$-coordinates of these branches. Using Case 1 above, we can isolate the branch $\gamma_{1}$, away from $Q_{1}$, by applying finitely many transformations of the form, given in $(\star)$, with $\alpha=\alpha_{1}$. Clearly, the branches $\{\gamma_{2},\ldots,\gamma_{r}\}$ remain in finite position after these transformations. Suppose, inductively, that we have isolated the branches $\{\gamma_{1},\ldots,\gamma_{i}\}$, with $1\leq i<r$, away from $Q_{1}$. We can then apply finitely many transformations of the form, given in $(\star)$, with $\alpha=\alpha_{i+1}$, in order to isolate the branch $\gamma_{i+1}$, away from $Q_{1}$. As the restriction of these transformations to the axes $\{x=\alpha_{1},\ldots,x=\alpha_{i}\}$ is of the form $a(\alpha_{j})z+b(\alpha_{j})$, with $a(\alpha_{j})=(\alpha_{i+1}-\alpha_{j})^{-1}\neq 0$, for $1\leq j\leq i$, the branches $\{\gamma_{1},\ldots,\gamma_{i}\}$ also remain isolated away from $Q_{1}$. Hence, we have isolated the branches $\{\gamma_{1},\ldots,\gamma_{i},\gamma_{i+1}\}$. By induction, we can, therefore, isolate all the branches $\{\gamma_{1},\ldots,\gamma_{r}\}$, away from $Q_{1}$, as required, $(\star\star\star\star\star)$. Assume, now, that $\{\gamma_{1},\ldots,\gamma_{r}\}$ are hyperbolic branches and $\gamma_{r+1}$ is parabolic. By Lemma 3.10, we can assume that $\{\gamma_{1},\ldots,\gamma_{r}\}$ are all centred in finite position, and $\gamma_{r+1}$ is centred at $[1:0:0]$. By $(\star\star\star\star\star)$, we can isolate the branches $\{\gamma_{1},\ldots,\gamma_{r}\}$, away from $Q_{1}$, using finitely many transformations of the type considered in $(\star)$. By a simple projective calculation, such transformations fix the point $[1:0:0]$, in particular the branch $\gamma_{r+1}$ remains centred at $[1:0:0]$. We now follow the method of Case 2, in order to isolate the branch $\gamma_{r+1}$, away from $Q_{1}$, while, retaining the property $(\star\star\star\star\star\star)$, that the remaining branches $\{\gamma_{1},\ldots,\gamma_{r}\}$ are isolated, away from $Q_{1}$, in finite position. This amounts to showing that $(\star\star\star\star\star\star)$ is preserved in the following steps from Case 2;\\

(1). For the rotation $\nu$, with $\alpha\neq\alpha_{i}$, for $1\leq i\leq r$.\\

(2). For the construction of the birational map $\theta_{1}$, with the additional use of the previous inductive proof.\\

(3). For the inverse rotation $\nu^{-1}$, with $\alpha\neq\alpha_{i}$, for $1\leq i\leq r$.\\

(4). For transformations of the form $\theta_{g}$, considered in Lemma 3.7, with $\alpha\neq\alpha_{i}$, for $1\leq i\leq r$.\\

Hence, the final part of the lemma is shown.

\end{proof}

\begin{rmk}
This seems to be a reasonable geometrical property of a theory of light, motivated by theoretical considerations of the author on the chemical production of energy by fission in stars, see also Remarks 3.12(a). The increased level of  radioactivity along the axis $x=0$ is, in fact, an observable phenomenon of pulsars, a type of neutron star, see \cite{GG}.

\end{rmk}

The following result is straightforward, see the previous footnote 3;\\

\begin{lemma}
Let $C_{1}$ and $C_{2}$ satisfy the conditions of Lemma 3.3, then there exists a birational map $\theta:C_{1}\leftrightsquigarrow C_{2}$, which is an isomorphism in the etale topology with $(0,0)$
as labelled points, and $(pr\circ\theta)=\theta$.
\end{lemma}

\begin{proof}
By the definition of the etale topology, there exists an irreducible projective algebraic curve $(C_{3},*)$, together with an open subset $U_{3}\subset C_{3}$, containing the marked point $*$, and open subsets $U_{1}\subset C_{1}$, $U_{2}\subset C_{2}$, containing $(0,0)$, with etale morphisms;\\

$i_{13}:(C_{3},U_{3},*)\rightarrow (C_{1},U_{1},(0,0))$\\

$i_{23}:(C_{3},U_{3},*)\rightarrow (C_{2},U_{2},(0,0))$\\

such that $i_{13}^{*}(\delta_{i}(x))=i_{23}^{*}(\delta_{i}(x))$, for $1\leq i\leq m$, $(*)$, and $(pr\circ i_{13})=(pr\circ i_{23})$, $(**)$. Using $(**)$, the morphisms $i_{13}$ and $i_{23}$ induce inclusions;\\

$i_{13}^{*}:L(C_{1})\rightarrow L(C_{3})$\\

$i_{23}^{*}:L(C_{2})\rightarrow L(C_{3})$\\

with $(i_{13}^{*}\circ pr^{*})=(i_{23}^{*}\circ pr^{*}):L(x)\rightarrow L(C_{3})$, $(***)$. By $(*)$, $(***)$,
and Lemma 3.3(iii), we can define an isomorphism $\phi^{*}:L(C_{2})\rightarrow L(C_{1})$, such that $(\phi^{*}\circ pr^{*})=pr^{*}:L(x)\rightarrow L(C_{1})$. Hence, there exists a birational map $\phi:C_{1}\leftrightsquigarrow C_{2}$, with $(pr\circ\phi)=pr$, $(****)$. By Lemma 3.3(i), there exists a unique silver branch $\gamma_{1}$ of $C_{1}$, centred at the origin $(0,0)$. Using $(****)$, it is straightforward that $\theta(\gamma_{1})$ is a silver branch of $C_{2}$, centred along the axis $x=0$. Using Lemma 3.3(iv), and the method of Lemma 3.13, we can find a birational map $\theta_{h}:C_{2}\leftrightsquigarrow C_{2}$, with $(pr\circ\theta_{h})=pr$, $(*****)$ representing $h\in Gal(L(C_{2})/L(x))$, such that $(\theta_{h}\circ\phi)(\gamma_{1})=\gamma_{2}$, where $\gamma_{2}$ is the unique silver branch of $C_{2}$, centred at $(0,0)$. It is a simple calculation, using infinitesimals, Theorems (6.10,6.11) of \cite{depiro4} and the fact that the branches $\{\gamma_{1},\gamma_{2}\}$ are nonsingular, to show that $(\theta_{h}\circ\phi)$ is etale in an open neighborhood of $C_{1}$, containing $(0,0)$. Similar considerations apply to the morphism $(\theta_{h}\circ\phi)^{-1}$, hence, $(\theta_{h}\circ\theta)$ induces an isomorphism in the etale topology, with $(0,0)$ as labelled points. Letting $\theta=(\theta_{h}\circ\phi)$, the fact that $(pr\circ\theta)=pr$ follows immediately from $(****)$ and $(*****)$. Hence, the lemma is shown.

\end{proof}

\begin{rmk}
It is a straightforward exercise, to show that a morphism satisfying the conditions of Lemma 3.17 is unique.
Hence, we can unambigiously identify the branches of any two curves, satisfying the conditions of Lemma 3.3. The previous Lemmas 3.5 to 3.15, support the intuition that the orbits of each branch resemble silver or blue flashes of light, emanating from and towards a central point $Q_{1}$. Given the high gravitational field of a neutron star, this is, possibly, a good description of the behaviour of light in certain situations.
\end{rmk}

\begin{rmk}
Rose windows are another geometrically interesting source of examples, that exhibit both radial and rotational symmetry, as well as intricate light effects. Some of the best can be found in the Rayonnant designs of French medieval cathedrals. An excellent account of the historical development of rose windows is given in \cite{painton}.
For the reader, interested in making rose windows, one should look at \cite{helga}. A number of pictures of rose windows should be available soon on the website, http://www.magneticstrix.net, where you can find some good examples of neutron star style windows in the Rayonnant Rose of Saint Denis, and the North and South Roses of Notres Dame cathedral.

\end{rmk}

We now make the following refinement of the branch terminology that we have used;\\

\begin{defn}

Let conditions be as in Definition 3.11, then we define a branch $\gamma$ to be violet, if;\\

$\gamma$ is a non-singular blue branch.\\

\noindent and we define a branch $\gamma$ to be green, if;\\

$\gamma$ is a silver branch, centred along the axis $x=0$, (see also footnote 5), or
$\gamma$ is a blue branch, but not a violet branch.\\

\end{defn}

\begin{rmk}

In order to see that this is a good definition,

\end{rmk}

We then have the following geometric interpretation of Lemma 2.5.\\

\begin{lemma}
Let hypotheses be as in Lemma 2.5 and notation as in Definition 3.1, then we can find an irreducible closed algebraic surface $S\subset P^{3}$, such that the restriction of $pr_{1}$ to $(S\setminus Q_{1})$ defines a quasi-finite morphism, $pr_{1}:S\rightarrow P_{1}^{2}$, $(\heartsuit)$, with the following further additional properties;\\

(i). There exist open subsets $U\subset A_{1}^{2}\subset P_{1}^{2}$ and $V\subset S\cap A^{3}$ such that;\\

$pr_{1}:(V,(00)^{lift})\rightarrow (U,(00))$ is etale.\\

(ii). There exist distinct irreducible closed algebraic curves $\{C_{1},\ldots,C_{m}\}$ contained in $S$, such that;\\

a. $pr_{1}:(C_{j}\setminus\{Q_{1}\})\rightarrow C$ is a finite cover, for $1\leq j\leq m$, (\footnote{We use the term cover, in the sense that $pr_{1}$ defines a dominant morphism from $(C_{j}\setminus\{Q_{1}\})$ onto an open subset of $C$. The projection naturally extends to define a surjective map from the branches of $C_{j}$, including those centred at $Q_{1}$, to the branches of $C$. The reader should look at the paper \cite{depiro6}})\\

b. $\overline{(pr_{1}^{-1}(C)\cap S)}=\cup_{1\leq j\leq m}C_{j}$.\\

c. The coordinate ring $R(V)$ embeds in $L[[x]][y]\cap L(x,y)^{alg}$ and the defining equations of $C_{j}\cap V$ are given by;\\

 $y-\eta_{j}(x)=0$, for $1\leq j\leq m$.\\

(iii). There exists an irreducible algebraic curve, $C'\subset P_{2}^{2}$, satisfying the conditions of Lemma 3.3, with respect to the coordinate system $(x,z)$, such that;\\

a. $pr_{2}:C_{j}\rightarrow C'$ is a finite cover, for $1\leq j\leq m$.\\

b. $pr_{2}:C_{j}\leftrightsquigarrow C'$ is a birational map, with inverse $\eta_{j}:C'\leftrightsquigarrow C_{j}$\\

and an open subset $V'\subset C'$, such that;\\

c. $pr_{1}:(V',(00))\rightarrow (A^{1},0)$ is a Galois cover, with the extension $L(V')/L(x)$ equal to the Galois closure of the extension $L(C)/L(x)$.\\

\end{lemma}

\begin{proof}

From the previous construction of Lemma 2.5, the power series $\{\delta_{1}(x),\ldots,\delta_{m}(x)\}$ belong to $L[[x]]\cap L(x)^{alg}$. By Theorem 1.3 and remarks at the beginning of Section 3 of \cite{depiro5}, we have that $L[[x]]\cap L(x)^{alg}$ is isomorphic to the local ring for the etale topology ${\mathcal O}^{et}_{(A^{1},0)}$. By the construction in Lemma 3.3, we can find an irreducible plane projective curve $C'\subset P^{2}$, such that;\\

(i). There exists an open subset $U_{1}\subset C'$, such that $(0,0)\in U_{1}$ and;\\

$pr:(U_{1},(0,0))\rightarrow (A^{1},0)$\\

is an etale morphism, $(\dag)$\\

(ii). $\{\delta_{1}(x),\ldots,\delta_{m}(x)\}$ all belong to $R(U_{1})$, for a given embedding of $R(U_{1})$ in $L[[x]]\cap L(x)^{alg}$, and $\delta_{1}(0,0)=\ldots=\delta_{m}(0,0)=0$, $(\dag')$\\

(iii). The function field $L(U_{1})$ is generated over $L(x)$ by the power series $\{\delta_{1}(x),\ldots,\delta_{m}(x)\}$, (\footnote{In non-zero characteristic, one should take $L(U_{1})$ to be the seperable closure of $L(x,\delta_{1}(x),\ldots,\delta_{m}(x))$}). $(\dag'')$.\\

By Fact 1.5 of \cite{depiro5}, we can find an open subset $U'\subset A^{1}$, containing $(0)$, and a monic polynomial $F(z)\in R(U')[z]$ such that the etale morphism $pr$ can be presented, for an open neighborhood $V'\subset U_{1}\subset C'$ of $(0)^{lift}=(0,0)$ in the form;\\

$Spec(({R(U')[z]\over F(z)})_{d})\rightarrow Spec(R(U'))$ $(*)$\\

with $F'(z)$ invertible in $({R(U')[z]\over F(z)})_{d}$. As $(V',(0,0))$ is a localisation of $(U_{1},(0,0))$, the coordinate ring $R(U_{1})\subset R(V')\subset L[[x]]\cap L(x)^{alg}$ and the conditions $(\dag)$, $(\dag')$ and $(\dag'')$ are preserved, replacing $U_{1}$ by $V'$.\\

Let $G(x,z)$ define the projective curve $C'$, restricted to $A_{2}^{2}$, see Definition 3.1, which we also denote by $C'$. The restriction of $C'$ to $V'=(pr_{1}^{-1}(U')\cap(G=0)\cap (d\neq 0))$ corresponds to the cover defined in $(*)$. We also consider $G(x,z)$ as defining an irreducible algebraic surface $C'\times A^{1}$ in $A^{3}$, using the coordinates $(x,y,z)$. We now projectivize the affine polynomial $G(x,z)$, by making the substitutions $\{x={X\over W},z={Z\over W}\}$, and obtain an irreducible homogeneous polynomial $R(X,Z,W)$. This defines an irreducible algebraic surface $S\subset P^{3}$. Observe that, as $C'$ passes through $(0,0)$, the surface $S$ contains the point $Q_{2}$. It is clear from the construction that $pr_{2}(S\setminus \{Q_{2}\})=C'$ and $pr_{1}(S\setminus Q_{1})\subset P_{1}^{2}$. It is straightforward to see that the fibres of $pr_{1}$ restricted to $(S\setminus Q_{1})$ are all finite, showing $(\heartsuit)$ in the statement of the lemma.\\

We now set $U=U'\times A^{1}\subset A_{1}^{2}$ and $V=V'\times A^{1}\subset S\cap A^{3}$. By the fact that $R(V')$ embeds in $L[[x]]\cap L(x)^{alg}$, we have that $R(V)$ embeds in $L[[x]][y]\cap L(x,y)^{alg}$, $(**)$. The etale morphism $pr_{1}:(V',(00))\rightarrow (U',(0))$ may be presented in the form;\\

$R(U')\rightarrow {R(U')[z,w]\over <F(z),w\delta(z)-1>}$ with ${\partial(F(z),w\delta(z)-1)\over\partial (z,w)}|_{p}\neq 0$, for $p\in V'$\\

The morphism $pr_{1}:(V,(000)\rightarrow (U,(00))$ may then be presented in the same form, replacing the ring $R(U')$ by the ring $R(U)=R(U')[y]$. By the definition of an etale morphism in Definition 1.1 of \cite{depiro5}, we obtain immediately that this morphism is etale, hence (i) of the Lemma is shown.\\

By $(\dag')$ and $(**)$, we now have that the algebraic power series\\ $\{y-\delta_{1}(x),\ldots,y-\delta_{m}(x)\}$ define  irreducible algebraic curves $\{D_{1},\ldots,D_{m}\}$ on $V$, passing through $(00)^{lift}=(000)$. By the construction of $V$, we have that the algebraic power series $\{y-a_{1}-\delta_{1}(x),\ldots,y-a_{m}-\delta_{m}(x)\}$ define distinct irreducible algebraic curves $\{D_{1}',\ldots,D_{j}',\ldots,D_{m}'\}$ on $V\subset S$, passing through $(0,a_{j},0)$, for $1\leq j\leq m$. We let $C_{j}=\overline{D_{j}'}$, for $1\leq j\leq m$. By elementary facts on Zariski closure, each $C_{j}\subset S$ is irreducible and $(ii)(c)$ of the lemma holds, by the definition of $\eta_{j}(x)=a_{j}+\delta_{j}(x)$ in Lemma 2.5, and the fact that $D_{j}'=C_{j}\cap V$. Again, by Lemma 2.5, $(pr_{1}^{-1}(C)\cap V)=\cup_{1\leq j\leq m}(C_{j}\cap V)$, $(***)$. It follows immediately, from elementary facts about quasi-finite morphisms, that $(ii)(a)$ of the lemma holds and $\cup_{1\leq j\leq m}C_{j}\subseteq \overline{(pr_{1}^{-1}(C)\cap S)}$. Now suppose that $C_{0}\subset S$ is an irreducible component of $\overline{(pr_{1}^{-1}(C)\cap S)}$. By elementary dimension theory, $C_{0}$ is an algebraic curve and must define a finite cover of $C$. We may suppose that $(C_{0}\cap A^{3})\neq\emptyset$, otherwise $C_{0}$ is contained in the plane $(W=0)$ and, therefore, $pr_{1}(C_{0})=C$ is contained in the line $l_{1,\infty}$, contradicting the presentation of $C$ in Lemma 1.1. Suppose that $(C_{0}\cap V)=\emptyset$, then $C_{0}\cap A^{3}$ must be a line of the form $\{p\}\times A^{1}$, where $p=(p_{1},p_{2})$ is one of the finitely many exceptional points on the affine curve $C'\subset A_{2}^{2}$, obtained by removing $V'$. It follows that $pr_{1}(C_{0})=C$ is contained in the closure of the line $x=p_{1}$, in the coordinate system $(x,y)$ of $A_{1}^{2}$, again contradicting the presentation of $C$ in Lemma 1.1. Hence, we may assume that $(C_{0}\cap V)\neq\emptyset$. It follows that $(C_{0}\cap V)$ defines an irreducible component curve of $(pr^{-1}(C)\cap V)$, hence, by $(***)$, coincides with $(C_{j}\cap V)$, for some $1\leq j\leq m$. It follows, by elementary facts on Zariski closure, that $C_{0}=C_{j}$. Therefore, $(ii)(b)$ of the lemma holds and $(ii)$ of the lemma is shown.\\

By the presentation given in Lemma 1.1, the point $Q_{2}$ does \emph{not} lie on $C$. By (ii)(a) of the lemma, and the fact that $Q_{2}$ is \emph{fixed} by the projection $pr_{1}$,  $Q_{2}$ cannot lie on $C_{j}$ for any $1\leq j\leq m$ either. It follows that $pr_{2}$ is defined on all of $C_{j}$, for $1\leq j\leq m$. As we observed earlier, $pr_{2}(S\setminus \{Q_{2}\})=C'$, hence, as each curve $C_{j}$ belongs to $(S\setminus\{Q_{2}\})$, and, using the argument above to exclude the exceptional case that $pr_{2}(C_{j})$ is a point, we obtain $(iii)(a)$ of the lemma. Now, observe that the algebraic power series $\{\eta_{1}(x),\ldots,\eta_{m}(x)\}$ define rational functions in $L(C')$, and belong to the coordinate ring $R(V')$, for the open subset $V'\subset C'$. We may, therefore, define a morphism, $\theta_{j}:V'\rightarrow V$\\

$\theta_{j}(x,z)=(x,\eta_{j}(x,z),z)$, for $(x,z)\in V'$\\

The defining equation of the image of this morphism in $R(V)$ is given by $y-\eta_{j}(x,z)=0$, hence, by (ii)(c) of the lemma, corresponds to $(C_{j}\cap V)$. It follows that $\theta_{j}:V'\rightarrow C_{j}$ must define a morphism as well. By the explicit definition of $pr_{2}$ in the affine coordinates $(x,y,z)$, we have that $pr_{2}\circ\theta_{j}=Id_{V'}$. It follows immediately that $pr_{2}$ defines an isomorphism between the open subsets $(C_{j}\cap V)\subset C_{j}$ and $V'\subset C'$. Hence, (iii)(b) of the lemma is shown. The property (iii)(c) follows immediately from the construction of $V'$, by observing that $L(V')=L(x,\delta_{1}(x),\ldots,\delta_{m}(x))=L(x,\eta_{1}(x),\ldots,\eta_{m}(x))$ is a splitting field for the polynomial $F(x,y)$ over $L(x)$, hence, the cover defined in $(iii)(c)$ is Galois and the extension $L(V')/L(x)$ is equal to the Galois closure of $L(C)/L(x)$, see also footnote 1. Therefore, (iii) of the lemma, is shown. This completes the proof.

\end{proof}

\begin{defn}
We define the set of curves $\{C_{1},\ldots,C_{j},\ldots,C_{m}\}$, found in the previous Lemma 3.21, as flashes of the original algebraic curve $C$, (\footnote{More accurately, a flash $C_{j}$ should be considered as representing the class of curves, corresponding to each power series $\eta_{j}(x)$. That is, given two curves $\{C^{1},C^{2}\}$, satisfying the conditions of Lemma 3.3, and a birational map $\alpha_{12}:C^{1}\leftrightsquigarrow C^{2}$, as in Lemma 3.16, if $\{C_{1}^{1},\ldots,C_{j}^{1},\ldots,C_{m}^{1}\}$ and $\{C_{1}^{2},\ldots,C_{j}^{2},\ldots,C_{m}^{2}\}$ denote the sets of curves obtained in Lemma 3.21, then we obtain naturally defined birational maps $\theta_{j}=(\eta_{j}\circ\alpha_{12}\circ(pr_{2})^{-1}):C_{j}^{1}\leftrightsquigarrow C_{j}^{2}$. We can then define a flash as the equivalence class of a given curve $C_{j}$ under these birational maps; by Lemma 3.16, this accounts for all such curves.})

\end{defn}

We then have;\\

\begin{lemma}

There exists a finite group $G$ and an algebraically definable action of $G$ on an open subvariety $W'\subset V'$, with $pr_{1}:W'\rightarrow A^{1}$ a quasi-finite morphism, such that;\\

(i). For generic $x\in A^{1}$, $G$ acts sharply transitively on the fibre $pr_{1}^{-1}(x)$.\\

(ii). $G$ induces an algebraically definable, transitive action on the set of flashes $\{C_{1},\ldots,C_{m}\}$.\\

\end{lemma}

\begin{proof}

We let $G=Gal(L(C')/L(x))$. By standard considerations, a given $g\in G$ induces a birational morphism $\theta_{g}:C'\leftrightsquigarrow C'$, with the property that $pr_{1}\circ\theta_{g}=pr_{1}:C'\rightsquigarrow P^{1}$, as rational maps. More specifically, we can find polynomials $\{p_{g}(x,z),q_{g}(x,z)\}$ with the property that $\theta_{g}$ is defined, in affine coordinates, by;\\

$\theta_{g}:(x,z)\mapsto (x,{p_{g}(x,z)\over q_{g}(x,z)})=(x,g\centerdot z)$ $(*)$\\

Letting $G(x,z)$ define the curve $C'$, restricted to the coordinate system $(x,z)$, set $R_{g}=((G=0)\cap (q_{g}=0))$, and $S_{g}=pr_{1}(R_{g})$. Let $C''$ define the curve obtained by removing the finitely many coordinate lines $\{x=s:s\in S_{g}\}$, from the curve $C'$, realised by $G(x,z)$. It is straightforward to check that $C''$ is invariant under the action of $\theta_{g}$. By repeating the construction, we can choose $C''$ to be invariant under the action of $\theta_{g}$, for any $g\in G$, $(**)$. If $g\in G$, then $C''$ is invariant under the action of both $\theta_{g}$ and $\theta_{g^{-1}}=(\theta_{g})^{-1}$, it follows that $\theta_{g}$ maps $C''$ isomorphically onto itself, $(***)$. By a similar argument, removing also coordinate lines of the form $\{x=t:t\in T_{g}\}$, from $V'$, where;\\

 $T_{g}=\{t\in A^{1}:Card(V'(t))<Card(V'(x)),\ for\ generic\ x\in A^{1}\}$\\

 we can find an open subvariety $W'\subset V'$, with the same properties $(**)$ and $(***)$. We consider the action of $G$ on a generic fibre $W'(x)$. Assuming $deg_{z}G(x,z)=l$, this consists of $l$ distinct points, in finite position. It is also straightforward to check that $deg(L(W')/L(x))=l$. As the extension $L(W')/L(x)$ is Galois, $Card(G)=l$. If the action fails to be transitive, then, by a simple counting argument, we can find a non-trivial $g_{0}\in G$ and $x_{0}\in W'(x)$, with $g_{0}\centerdot x_{0}=\theta_{g_{0}}(x_{0})=x_{0}$. As the morphisms $\theta_{g_{0}}$ and $\theta_{Id}=Id$ are etale, it follows that $\theta_{g_{0}}=Id$, see Proposition 3.15 of \cite{M}, which is a contradiction. Sharp transitivity is obtained by a similar argument. This proves $(i)$.\\

By Lemma 2.5, we have that the power series $\{\eta_{1}(x),\ldots,\eta_{m}(x)\}$ constitute a complete set of distinct roots for the polynomial $F(x,y)$, considered as belonging to $L(x)[y]$.  By $(\dag)$ of Lemma 3.3, and the later remark there, the power series $\{\eta_{1}(x),\ldots,\eta_{m}(x)\}$ also belong to the coordinate ring $R(V')\subset L(C')$. In particular, we have that the function field $L(C')$ is a splitting field for the polynomial $F(x,y)$, over the subfield $L(x)$, and $G$ acts transitively on the power series $\{\eta_{1}(x),\ldots,\eta_{m}(x)\}\subset R(V')\subset R(W')$, $(\sharp)$. We can naturally extend the action of $G$ to the coordinate ring $R(W)=R(W')[y]$. By $(\sharp)$, this defines a transitive action on the irreducible curves $\{D_{1}',\ldots,D_{m}'\}$, which extends to an action on the flashes $\{C_{1},\ldots,C_{m}\}$, using elementary facts on Zariski closure and the definition of $C_{j}=\overline{D_{j}'}$, for $1\leq j\leq m$. (This action clearly also respects the equivalence relation, defined in footnote 9.) Hence, $(ii)$ is shown. Algebraic definability of the group action follows easily, by enumerating the parameters involved in the action defined by $(*)$, and restricting to the variety $W'$.
\end{proof}

We now give a more refined version of Lemma 3.3, by analysing the intersections and singularities of the flashes $\{C_{j}:1\leq j\leq m\}$.\\

We first observe the following, we recall the notion of $val_{\gamma}$, for a birational morphism, from \cite{depiro7}.\\

\begin{lemma}
Let $C'$ satisfy the conditions of Lemmas 3.3 and 3.21, then, if a branch $\gamma$ of $C'$ is centred in finite position along the axis $x=0$, the values of the birational morphisms $val_{\gamma}(\eta_{j})$ are distinct, belonging to the set $\{a_{1},\ldots,a_{m}\}$.

\end{lemma}

\begin{proof}
 Suppose that $\gamma$ is a branch, centred along $x=0$. If $b_{j}$ denotes the value $val_{\gamma}(\eta_{j})$, and $\gamma$ is centred at $(0,c)$, then $(0,b_{j},c)$ belongs to the curve $C_{j}$. By $(ii)(a)$ of Lemma 3.21, the projection $(0,b_{j})$ belongs to $C$, hence, $b_{j}$ belongs to the set $\{a_{1},\ldots,a_{m}\}$. Suppose that $val_{\gamma}(\eta_{j_{0}})=val_{\gamma}(\eta_{j_{1}})=a_{j}$, for $j_{0}\neq j_{1}$.
The birational morphisms $\{\eta_{j_{0}},\eta_{j_{1}}\}$ are distinct, by Lemma 3.21, hence, there exists an open subset $U\subset C'$, for which $(\eta_{j_{0}}-\eta_{j_{1}})$ is non-zero. Let $(x,z)\in(\gamma\cap U)$, then the values $\eta_{j_{0}}(x,z)$ and $\eta_{j_{1}}(x,z)$ are distinct and the points $\{(x,\eta_{j_{0}}(x,z),z),(x,\eta_{j_{1}}(x,z),z)\}$ belong to the curves $(C_{j_{0}}\cap{\mathcal V}_{(0,a_{j},c)})$ and $(C_{j_{1}}\cap{\mathcal V}_{(0,a_{j},c)})$ respectively. Applying Lemma 3.21, the projected points $\{(x,\eta_{j_{0}}(x,z)),(x,\eta_{j_{0}}(x,z))\}$ belong to $C\cap{\mathcal V}_{(0,a_{j})}$. As $x\in{\mathcal V}_{0}$, it follows that $I_{(0,a_{j})}(C,x=0)\geq 2$. This clearly contradicts the presentation of $C$, given in Lemma 1.1(iii).

\end{proof}

We recall that $C$ is a nodal curve, satisfying the assumptions of Lemma 1.1.

\begin{defn}

 Let $\wp=\{\rho_{1},\ldots,\rho_{i},\ldots,\rho_{r}\}$ enumerate the branches of $C$ with the following properties, relative to the coordinate system $(x,y)$;\\

Either (i). The tangent line of $\rho_{i}$ to $C$ is vertical.\\

Or (ii). $\rho_{i}$ is centred at a node, belonging to $C$.\\

By Lemma 1.1, all such branches are in finite position, and the conditions are mutually exclusive.

\end{defn}

 In the terminology of Lemma 3.21, we observe that the projection $pr_{1}$ defines a quasi finite map from $(C_{j}\setminus\{Q_{1}\})$ to $C$, which extends to give a well-defined map from the branches $\gamma$ of $C_{j}$ to $C$, and that the branches of $C_{j}$ can be naturally identified with the branches of $C'$, using the birational map $\eta_{j}$. This motivates the following definition;\\

\begin{defn}
For the branches in Definition 3.25, we define,\\
 $\Gamma_{j}=\{\gamma_{1,j},\ldots,\gamma_{t,j}\}$ to be the lifting of these branches to the flash $C_{j}$, in Lemma 3.21, and also of the projective curve $C'$, given in Lemma 3.24. We define $\Gamma=\{\gamma_{1},\ldots,\gamma_{t}\}=\bigcup_{1\leq j\leq m}\Gamma_{j}$.
\end{defn}

We have the following birational invariance;\\

\begin{lemma}
Let $\alpha_{12}:C^{1}\leftrightsquigarrow C^{2}$ satisfy the conditions of Lemma 3.3 and Lemma 3.16, with equivalent flashes $\theta_{j}:C_{j}^{1}\leftrightsquigarrow C_{j}^{2}$, defined by Lemma 3.21, see also footnote 9. Then the branches $\{\gamma,\eta_{j}(\gamma)\}$ of $\{C^{1},C_{j}^{1}\}$ lift the same branch $\rho$ of $C$, as $\{\alpha_{12}(\gamma),\theta_{j}(\eta_{j}(\gamma))\}$. In particular, the sets $\Gamma_{j}$ and $\Gamma$ are birationally invariant, and $val_{\gamma}(\eta_{j})=val_{\alpha_{12}(\gamma)}(\eta_{j})$.

\end{lemma}

\begin{proof}
We denote the representatives of the power series $\eta_{j}$ by the rational functions $f\in L(C^{1})$ and $g\in L(C^{2})$. By the definition of $\alpha_{12}$, we have that $(g\circ\alpha_{12})=f$, and $(pr_{1}\circ\alpha_{12})=pr_{1}$, as rational functions on $C^{1}$. It is sufficient to show that $(pr_{1}\circ\eta_{j})=(pr_{1}\circ\eta_{j}\circ\alpha_{12}):C^{1}\rightarrow C$, as rational maps. For a generic point $(x,z)\in C^{1}$, we have that;\\

$(pr_{1}\circ\eta_{j})(x,z)=pr_{1}(x,f(x,z),z)=(x,f(x,z))$\\

$(pr_{1}\circ\eta_{j}\circ\alpha_{12})(x,z)=(pr_{1}\circ\eta_{j})(x,(pr\circ \alpha_{12})(x,z))$\\
\indent \ \ \ \ \ \ \ \ \ \ \ \ \ \ \ \ \ \ \ \ \ \ \ \ \ \ \ $=pr_{1}(x,g(x,(pr\circ \alpha_{12})(x,z)),z)$\\
\indent \ \ \ \ \ \ \ \ \ \ \ \ \ \ \ \ \ \ \ \ \ \ \ \ \ \ \ $=pr_{1}(x,(g\circ\alpha_{12})(x,z),z)$\\
\indent \ \ \ \ \ \ \ \ \ \ \ \ \ \ \ \ \ \ \ \ \ \ \ \ \ \ \ $=(x,f(x,z))$ (\footnote{We have temporarily used the notation $pr$ for the projection onto the $z$-coordinate.})\\

Hence, the result is shown.
\end{proof}

We make the following definition;\\

\begin{defn}
Let $C'$ satisfy the conditions of Lemmas 3.3 and 3.21, then, we define a branch to be geometric, if it is either finite, or, hyperbolic and centred at $Q_{1}$. We define a branch to be perpendicular, if it is infinite, and also parabolic, if centred at $Q_{1}$. See Definition 2.1 for relevant terminology.
\end{defn}

\begin{rmk}
A branch is geometric if and only if it is not perpendicular. This follows immediately from Definition 2.1. The notions of geometric and perpendicular are also birationally invariant, in the sense of Lemma 3.16.

\end{rmk}

\begin{lemma}
Let $C'$ satisfy the conditions of Lemmas 3.3 and 3.21, then, for a geometric branch $\gamma$ of $C'$, we have that;\\

$val_{\gamma}(\eta_{j_{0}})=val_{\gamma}(\eta_{j_{1}})$, for $j_{0}\neq j_{1}$, iff $\gamma\in(\Gamma_{j_{0}}\cap \Gamma_{j_{1}})$\\

There does not exist a geometric branch $\gamma$ of $C'$, on which $\{\eta_{j_{0}},\eta_{j_{1}},\eta_{j_{2}}\}$ take the same $val_{\gamma}$,  for a distinct triple $(j_{0},j_{1},j_{2})$\\

\end{lemma}

\begin{proof}
By Lemma 3.10 and Remarks 3.29, we can assume that $\gamma$ is in finite position. If $val_{\gamma}(\eta_{j})=\infty$, then the point $[0:1:0:0]$ belongs to the flash $C_{j}$, hence, the point $[0:1:0]$ belongs to the curve $C$. This contradicts the presentation of $C$ in Lemma 1.1(iii). Suppose $\gamma$ is centred at $(x_{0},z_{0})$, with $val_{\gamma}(\eta_{j_{0}})=val_{\gamma}(\eta_{j_{1}})=y_{0}<\infty$. Then the point $(x_{0},y_{0},z_{0})$ belongs to the intersection $C_{j_{0}}\cap C_{j_{1}}$, and projects to a point $(x_{0},y_{0})$ on $C$. We can choose $(x,z)\in{\mathcal V}_{(x_{0},z_{0})}$, such that $\eta_{j_{0}}(x,z)\neq\eta_{j_{1}}(x,z)$, and, therefore, find distinct points $\{y,y'\}$, with $(x,y,z)\in ({\mathcal V}_{(x_{0},y_{0},z_{0})}\cap C_{j_{0}})$ and $(x,y',z)\in ({\mathcal V}_{(x_{0},y_{0},z_{0})}\cap C_{j_{1}})$. The distinct projected points $(x,y)$ and $(x,y')$ belong to $({\mathcal V}_{(x_{0},y_{0})}\cap C)$. It follows that $I_{(x_{0},y_{0})}(C,x=x_{0})\geq 2$. By the presentation of $C$ in Lemma 1.1, this can only occur if the branch(es), centred at $(x_{0},y_{0})$, belong to $\wp$. It follows, from Definition 3.26, that the branch $\gamma\in (\Gamma_{j_{0}}\cap \Gamma_{j_{1}})$. Conversely, suppose that $\gamma\in(\Gamma_{j_{0}}\cap \Gamma_{j_{1}})$, so the projected branches $pr_{1}(\eta_{j_{0}}(\gamma))$ and $pr_{1}(\eta_{j_{1}}(\gamma))$, $(*)$, belongs to $\wp$. Supposing that $val_{\gamma}(\eta_{j_{0}})\neq val_{\gamma}(\eta_{j_{1}})$, then the projections, $(*)$, are centred at distinct points $(x_{0},y_{0})$, $(x_{0},y_{1})$ of $C$. This contradicts Lemma 1.1(iv). Finally, suppose that $val_{\gamma}(\eta_{j_{0}})=val_{\gamma}(\eta_{j_{1}})=val_{\gamma}(\eta_{j_{2}})=y_{0}<\infty$. By a similar argument to the above, we would find a point $(x_{0},y_{0})\in C$, with $I_{(x_{0},y_{0})}(C,x=x_{0})\geq 3$. This contradicts Lemma 1.1(ii).

\end{proof}

As a consequence of the preceding lemma, we obtain;\\

\begin{lemma}
Let $C'$ satisfy the conditions of Lemma 3.30, then we have that;\\

$\Gamma=\bigsqcup_{1\leq j_{0}<j_{1}\leq m}(\Gamma_{j_{0}}\cap \Gamma_{j_{1}})$\\

is a disjoint union of pairwise intersections.

\end{lemma}

\begin{proof}
Suppose that $\gamma\in\Gamma$, therefore, by Definition 3.26, belongs to $\Gamma_{j_{0}}$, for some $1\leq j_{0}\leq m$. We claim that $\gamma$ is geometric. For suppose that $\gamma$ is perpendicular, then, by Lemma 3.10, it can be centred at $[1:0:0]$. One can verify, that the projected branch $pr_{1}(\eta_{j_{0}}(\gamma))$ is centred on the line $l_{\infty}$, (\footnote{The case that $val_{\gamma}(\eta_{j_{0}})$ is infinite, being the more difficult calculation. We approximate the point $[1:0:0]$, by $[{1\over\epsilon}:1:1]$ and obtain the approximation $[{1\over\epsilon}:{1\over\epsilon^{j}}u(\epsilon):1:1]$ for the centre of the branch $\eta_{j_{0}}(\gamma)$. This specialises to $[1:u(0):0:0]$ and projects to $[1:u(0):0]$, on $l_{\infty}$. See similar calculations, previously in the paper.}). However, $(\wp\cap l_{\infty})=\emptyset$, by Lemma 1.1(i),(ii). By Lemma 3.10 and Remarks 3.29, the branch $\gamma$ can be centred in finite position, $(x_{0},z_{0})$, and the projected branch $pr_{1}(\eta_{j_{0}}(\gamma))$ is centred, in finite position, along the axis $x=x_{0}$. Suppose that the values $val_{\gamma}(\eta_{j})$, $(*)$, are distinct from $val_{\gamma}(\eta_{j_{0}})$, for $j\neq j_{0}$. If these values, $(*)$, are, themselves, different, then, using the observation that the values are finite, we clearly obtain $m$ seperate points of $C$, along the axis $x=x_{0}$. However, the projected branch $\rho_{0}=pr_{1}(\eta_{j_{0}}(\gamma))$ belongs to $\wp$, and, therefore, $I_{\rho_{0}}(C,x=x_{0})=2$. It follows, by Bezout's theorem, that there can only exist $m-1$ different points along the axis $x=x_{0}$. It follows that there must exist $\{\eta_{j_{1}},\eta_{j_{2}}\}$, with $(j_{0},j_{1},j_{2})$ a distinct triple, such that $val_{\gamma}(\eta_{j_{1}})=val_{\gamma}(\eta_{j_{2}})$. Now, by Lemma 3.30, the projected branch $\rho_{1}=pr_{1}(\eta_{j_{1}}(\gamma))$ belongs to $\wp$. It follows that we obtain two different branches $\{\rho_{0},\rho_{1}\}\subset\wp$, centred at seperate points along the axis $x=x_{0}$. This contradicts Lemma 1.1(iv). Hence, there must exist $j_{0}\neq j_{1}$, with $val_{\gamma}(\eta_{j_{0}})=val_{\gamma}(\eta_{j_{1}})$. In particular, $pr_{1}(\eta_{j_{1}}(\gamma))$ belongs to $\wp$. Therefore, $\gamma\in(\Gamma_{j_{0}}\cap \Gamma_{j_{1}})$. If the intersection of two of the sets in $\bigcup_{1\leq j_{0}<j_{1}\leq m}(\Gamma_{j_{0}}\cap \Gamma_{j_{1}})$, $(**)$, were non-empty, then, using the first part of Lemma 3.30, we would obtain a distinct triple $(j_{0},j_{1},j_{2})$, with $val_{\gamma}(\eta_{j_{0}})=val_{\gamma}(\eta_{j_{1}})=val_{\gamma}(\eta_{j_{2}})$. This contradicts the second part of Lemma 3.30. Hence, the union $(**)$ is disjoint.
\end{proof}

We make the following definition;\\

\begin{defn}
Let $\gamma$ be a blue branch of $C'$, in finite position, centred at $(x_{0},z_{0})$. We define a function $h\in L(C')$ to be symmetric at $\gamma$, if;\\

For generic $x_{1}\in{\mathcal V}_{x_{0}}$, and $\{(x_{1},y_{1}),\ldots,(x_{1},y_{t})\}=(\gamma\cap x=x_{1})$, we have that;\\

 $h(x_{1},y_{1})=\ldots=h(x_{1},y_{t})$, $(*)$.

\end{defn}

\begin{rmk}
This is a good definition. Without loss of generality, assume that $val_{\gamma}(h)=y_{0}$, and let $\gamma'$ be the corresponding branch, centred at $(x_{0},y_{0},z_{0})$, of the image curve $C'_{h}\subset P^{3}$, with $h:C'\leftrightsquigarrow C'_{h}$, and $\gamma''$ the corresponding branch, centred at $(x_{0},y_{0})$, of the projected curve $pr_{1}(C'_{h})$, with $pr_{1}:C'_{h}\rightsquigarrow pr_{1}(C'_{h})$. We are then able to observe that\\
 $Card(\gamma''\cap (x=x_{1}))=1$, witnessed by the projections;\\

 $(x_{1},h(x_{1},y_{1}))=\ldots=(x_{1},h(x_{1},y_{t}))$\\

 It follows that $I_{\gamma''}(pr_{1}(C'_{h}),x=x_{0})=1$, by definitions in \cite{depiro6}. In particular, the branch $\gamma''$ is non-singular, and cuts the axis $x=x_{0}$ transversely. By reversing the argument, we deduce easily that, for \emph{any} $x_{1}\in{\mathcal V}_{x_{0}}$, the property $(*)$ of Definition 3.32 holds.

\end{rmk}
We observe the following symmetry property;\\

\begin{lemma}

Let $C'$ satisfy the above conditions and let $\gamma$ be a blue branch of $C'$ in finite position. Then, it is not the case that every function $\{\eta_{1}(x),\ldots,\eta_{m}(x)\}$ is symmetric at $\gamma$.
\end{lemma}

\begin{proof}
Suppose, for contradiction, that each function $\{\eta_{1}(x),\ldots,\eta_{m}(x)\}$ is symmetric at $\gamma$. Let $\gamma$ be centred at $(x_{0},y_{0})$, and choose $x_{1}\in{\mathcal V}_{x_{0}}$ generically, with $\{(x_{1},y_{1}),\ldots,(x_{1},y_{t})\}=(\gamma\cap (x=x_{1}))$. As the branch $\gamma$ is blue, we have that $t\geq 2$. By Lemma 3.23(i), we can find a birational morphism $\theta_{g}:C'\leftrightsquigarrow C'$, with $\theta_{g}(x_{1},y_{1})=(x_{1},y_{2})$. $\theta_{g}^{*}$ defines an isomorphism of the function field $L(C')/L(x)$, in particular it permutes the set of roots $\{\eta_{1}(x),\ldots,\eta_{m}(x)\}$. The permutation is non-trivial, as $L(C')$ is generated by these roots, over $L(x)$, and $\theta_{g}^{*}\neq Id$. Without loss of generality, we can, therefore, assume that $\theta_{g}^{*}(\eta_{1})=\eta_{2}$. As both $(x_{1},y_{1})$ and $\theta_{g}(x_{1},y_{1})$ belong to $\gamma$, we must have that $\gamma$ is invariant under the action of $\theta_{g}$. In particular, for \emph{any} $x_{1}\in{\mathcal V}_{x_{0}}$, with $x_{1}\neq x_{0}$, $\theta_{g}$ defines a non trivial permutation of the fibre $\gamma\cap(x=x_{1})$ Now, for \emph{any} point $(x_{1},y_{1})$ on the branch $\gamma$, distinct from $(x_{0},y_{0})$, using the assumption of symmetry, we have that;\\

$\eta_{2}(x_{1},y_{1})=\theta_{g}^{*}(\eta_{1})(x_{1},y_{1})=\eta_{1}(\theta_{g}(x_{1},y_{1}))=\eta_{1}(x_{1},y_{1})$\\

This implies that the functions $\eta_{1}$ and $\eta_{2}$ are identical, which is not the case. Hence, one of the functions $\{\eta_{1}(x),\ldots,\eta_{m}(x)\}$ is not symmetric at $\gamma$.

\end{proof}

\begin{rmk}
Definition 3.32 can be easily extended to include the case of a blue branch centred at $[1:0:0]$. Namely, we define a function $h\in L(C')$ to be symmetric at $\gamma$, if $(*)$ holds, with the requirement that $x_{1}\in{\mathcal V}_{\infty}$ is generic, (in which case the intersections $\gamma\cap x=x_{1}$ are all in finite position). It is a straightforward exercise, left to the reader, to check that Remarks 3.33, and Lemma 3.34 hold for this definition.

\end{rmk}
We can now specify the positions of the blue and silver branches of $C'$;\\

\begin{lemma}
Suppose that the $r$ vertical tangents of $C$ are situated along the axes $\{x=a_{1},\ldots,x=a_{r}\}$, in the coordinate system $(x,y)$, then the geometric blue branches of $C'$, up to birationality, see Definitions 3.11 and 3.28, are situated exactly along the same axes, in the coordinate system $(x,z)$, and the geometric silver branches of $C'$, up to birationality, are situated exactly along the complementary axes, that is axes of the form $x=b$, for $b$ distinct from $\{a_{1},\ldots,a_{r}\}$. In particular, the geometric branches, which, up to birationality, are situated along the line $x=0$ are silver.
\end{lemma}

\begin{proof}
Suppose, for contradiction, that a blue branch $\gamma$, in finite position, is situated along a complementary axis $x=b$. By Lemma 3.34, we can assume that the function $\eta_{1}(x)$ is not symmetric at $\gamma$. As we saw in Lemma 3.30, $val_{\gamma}(\eta_{1})=d<\infty$. We, therefore, obtain, by the standard method, a corresponding branch $\gamma'$ of $C$, centred at $(b,d)$. As $\eta_{1}$ is not symmetric, it is a straightforward calculation, to show that $I_{\gamma'}(C,x=b)\geq 2$. This implies that $\gamma'$ is a vertical tangent of $C$, contradicting the supposition of the lemma. Any geometric blue branch can be centred in finite position, using Lemma 3.10 and Definition 3.11. If $\gamma'$ defines a vertical tangent of $C$, centred along $x=a_{j}$, for a given $1\leq j\leq r$, then, using Lemma 3.30, we can find a corresponding branch $\gamma$, up to birationality, in finite position, along the same axis $x=a_{j}$. It is a simple calculation, to show that $\gamma$ must be a blue branch. This show the first part of the lemma. The second part follows immediately from Lemma 3.13. The final part follows from Lemma 1.1(iii).
\end{proof}

\begin{lemma}
All the perpendicular branches of $C'$ are silver.
\end{lemma}

\begin{proof}
Suppose, for contradiction, that $\gamma$ is a perpendicular blue branch. By Definition 3.28, Lemma 3.10 and Definition 3.11, we can assume that $\gamma$ is centred at $[1:0:0]$. As $I_{\gamma}(C',l_{\infty})\geq 2$, by Definition 3.11, and, without loss of generality, the function $\eta_{1}(x)$ is not symmetric at $\gamma$, by Remarks 3.35, we can find distinct points $[{1\over\epsilon}:z_{1}:1]$ and $[{1\over\epsilon}:z_{2}:1]$, belonging to $\gamma\cap x={1\over\epsilon}$, with $\eta_{1}({1\over\epsilon},z_{1})\neq \eta_{1}({1\over\epsilon},z_{2})$. The distinct points $\{[{1\over\epsilon}:\eta_{1}({1\over\epsilon},z_{1}),:1], [{1\over\epsilon}:\eta_{1}({1\over\epsilon},z_{1}),:1]\}$ belong to the projected image branch $\gamma'$ of $C'$, centred along the line $l_{\infty}$. It follows that $I_{\gamma'}(C',l_{\infty})\geq 2$. This contradicts Lemma 1.1(i).
\end{proof}

We can use the above specification to simplify the analysis of the nodes of $C$. We make the following refinement of terminology;\\

\begin{defn}
Let $\wp_{\nu}$ enumerate the branches of $C$, satisfying Definition 3.25(ii), $\Gamma_{j,\nu}$ the lifting of these branches to $C_{j}$ and $C'$, as in Definition 3.26, and $\Gamma_{\nu}=\bigcup_{1\leq j\leq m}\Gamma_{j,\nu}$.
\end{defn}

It is a straightforward exercise to refine Lemma 3.27 and Lemma 3.31, namely;\\

\begin{lemma}
The sets $\Gamma_{j,\nu}$ and $\Gamma_{\nu}$ are birationally invariant, and $\Gamma_{\nu}=\bigsqcup_{1\leq j_{0}<j_{1}\leq m}(\Gamma_{j_{0},\nu}\cap \Gamma_{j_{1},\nu})$ is a disjoint union of pairwise intersections.\\

\end{lemma}

\begin{defn}
If $\gamma\in\Gamma_{\nu}$, we say that $\gamma$ is associated to a node $\nu_{0}$ of $C$, if,  $pr_{1}(\eta_{j_{0}}(\gamma))$ and $pr_{1}(\eta_{j_{1}}(\gamma))$, determined uniquely by Lemma 3.39, form the branches of
$\nu_{0}$.

\end{defn}
We can define an equivalence relation on $\Gamma_{\nu}$ as follows;\\

\begin{defn}
Let $\gamma_{1}$ and $\gamma_{2}$ belong to $\Gamma_{\nu}$. We define $\gamma_{1}\sim\gamma_{2}$ if they are associated to the same node of $C$.
\end{defn}

\begin{rmk}
It is an interesting exercise to show that, if the $s$ nodes of $C$ are situated along the axes $\{x=b_{1},\ldots,x=b_{s}\}$, in the coordinate system $(x,y)$, then any $2$ branches of $C'$, which, up to birationality, can be centred along a given axis $x=b_{j}$, for some $1\leq j\leq s$, in the coordinate system $(x,z)$, belong to $\Gamma_{\nu}$ and are equivalent. It is also clear, from the above, that the equivalence relation on $\Gamma_{\nu}$ has exactly $s$ classes.
\end{rmk}

With these definitions, we have;\\

\begin{lemma}
There exists $C'$, satisfying the conditions of Lemma 3.3 and Lemma 3.21, with $s$ branches $\{\gamma_{1},\ldots,\gamma_{s}\}$ belonging to $\Gamma_{\nu}$, representing each equivalence classes in Definition 3.41, such that each branch is isolated and in finite position.
\end{lemma}

\begin{proof}
By Lemma 3.36, and Lemma 1.1(iv), all the branches belonging to $\Gamma_{\nu}$ are silver. We can represent each equivalence class by $s$ distinct silver branches, which, up to birationality, can be centred along the distinct coordinate axes $\{x=b_{1},\ldots,x=b_{s}\}$. By Lemma 3.14, we can simultaneously isolate these branches in finite position.
\end{proof}

The following result is geometrically clarifying;\\

\begin{lemma}{Flash Intersections}\\

Let hypotheses and notation be as in the preceding lemmas, then;\\

(i). All the branches of the flashes $\{C_{1},\ldots,C_{m}\}$ are non-singular.\\

(ii). No three flashes intersect in a point. That is there is no point $p\in S$ and distinct flashes $\{C_{j_{1}},C_{j_{2}},C_{j_{3}}\}$, with $j_{1}<j_{2}<j_{3}$, such that $p\in (C_{j_{1}}\cap C_{j_{2}}\cap C_{j_{3}})$.\\

(iii). There exists a mapping from the intersections of flashes to the vertical tangent points and nodes of $C$. That is;\\

a. For every point $p\in S$ and distinct flashes $\{C_{j_{1}},C_{j_{2}}\}$, with $j_{1}<j_{2}$, such that $p\in (C_{j_{1}}\cap C_{j_{2}})$, we have that $pr_{1}(p)$ defines a vertical tangent point or a node of $C$ in the sense of Lemma 1.1(ii).\\

b. For every $p_{1}\in C$, defining a vertical tangent point or a node, there exists a $p\in S$ and distinct flashes $\{C_{j_{1}},C_{j_{2}}\}$, with $j_{1}<j_{2}$, such that $p\in (C_{j_{1}}\cap C_{j_{2}})$ and $pr_{1}(p)=p_{1}$.\\

(iv). There exists a bijection between the intersections of flashes and the spirit lines. That is;\\

a. Every intersection between two distinct flashes $\{C_{j_{1}},C_{j_{2}}\}$, with $j_{1}<j_{2}$, lies on a unique spirit line $l\subset S$.\\

b. For every spirit line $l\subset S$, there exists a unique point $p\in S$ and distinct flashes $\{C_{j_{1}},C_{j_{2}}\}$, with $j_{1}<j_{2}$, such that $p\in (l\cap C_{j_{1}}\cap C_{j_{2}})$.\\

(v). In particular, there exists a mapping from the spirit lines to the vertical tangent lines of $C$.\\

\end{lemma}

\begin{proof}
We adopt the coordinate of Definition 3.1. The coordinate projection $pr_{1}$ defines a quasi finite map from $(C_{i}\setminus\{Q_{1}\})$ to $C$, which extends to give a well-defined map from the branches $\gamma$ of $C_{i}$ to $C$. This follows from the observation that there exists a unique morphism $\phi$ from the desingularisation $C_{i}^{ns}$ to $C^{ns}$, for which $(pr_{1}\circ\theta_{i})=(\theta\circ\phi)$, where $\{\theta,\theta_{i}\}$ are the desingularization morphisms of $\{C,C_{i}\}$ respectively. We claim that if $\gamma$ is a singular branch of $C_{i}$, then $pr_{1}(\gamma)$ is a singular branch of $C$, $(*)$. Let $pr_{1}(\gamma)$ be centred at $A\in P^{2}$, and let $\gamma$ be centred at $B\in P^{3}$. Let $l$ be a line passing through $A$. Let $P_{l}$ be the plane spanned by the lines $\{l,l_{AQ_{1}}\}$, which must pass through the point $B$. If $\gamma$ is singular, then $I_{italian,\gamma}(C_{i},P_{l})\geq 2$, hence, we can choose a plane $P_{l}'$, infinitely close to $P_{l}$, with $C_{i}\cap P_{l}'\cap{\mathcal V}_{B}$, containing two distinct points $\{B_{1},B_{2}\}$. By an appropriate choice of $P_{l}'$, we can assume that $B_{2}\notin l_{B_{1}Q_{1}}$. Let $l'=P^{2}\cap P_{l}'$, then $\{pr_{1}(B_{1}),pr_{1}(B_{2})\}$ are distinct points belonging to $C\cap l'\cap{\mathcal V}_{A}$. It follows that $pr_{1}(\gamma)$ must be singular, showing $(*)$. By the hypothesis that $C$ is non-singular, it follows that all the branches of $C_{i}$ are non-singular, giving $(i)$. Suppose, for contradiction, that three flashes intersect in a point $p\in S$. Without loss of generality, assume that $q=pr_{2}(p)\in C'$ is in finite position. Choosing generic $q'\in{\mathcal V}_{q}$, we can assume that the functions $\{\eta_{j_{1}},\eta_{j_{2}},\eta_{j_{3}}\}$ are defined at $q'$, with $\{\eta_{j_{1}}(q'),\eta_{j_{2}}(q'),\eta_{j_{3}}(q')\}\subset{\mathcal V}_{p}$. As the flashes $\{C_{j_{1}}, C_{j_{2}},C_{j_{3}}\}$ are distinct, and therefore intersect in finitely many points, by assuming $q'$ is generic, we must have that the $y$-coordinates $\{y_{j_{l}},y_{j_{2}},y_{j_{3}}\}$ of the triple $\{\eta_{j_{1}}(q'),\eta_{j_{2}}(q'),\eta_{j_{3}}(q')\}$ are also distinct, $(\dag)$. Suppose that $pr_{1}(p)\in C$ has coordinates $(a,b)$ in the system $(x,y)$, then, the coordinates of $\{pr_{1}(\eta_{j_{1}}(q')),pr_{1}(\eta_{j_{2}}(q')),pr_{1}(\eta_{j_{3}}(q'))\}\subset C$ are of the form $\{(a'y_{j_{l}}),(a',y_{j_{2}}),(a',y_{j_{3}})\}$, with $a'\in{\mathcal V}_{a}$, and $\{y_{j_{l}},y_{j_{2}},y_{j_{3}}\}\subset\mathcal{V}_{b}$. It follows that $I_{(a,b)}(C,x=a)=3$, contradicting non-singularity of $C$ or Lemma 1.1(ii). This shows $(ii)$. The remaining cases $(iii)$ and $(iv)$ are left as an exercise for the reader, It is mainly a question of converting the results of Lemmas 3.30 and Lemma 3.31. One should be careful, to exclude exceptional cases, arising from the images of perpendicular branches.

\end{proof}

\end{section}

\begin{section}{Asymptotic Degenerations}

The flashes model of a non-singular algebraic curve $C$ is useful for studying the problem of degenerations of curves to lines in general position, as, such degenerations can be viewed as a continuous family of flashes, $(\sharp)$. (\footnote{The modern catastrophe theory approach to the problem seems, at least from an aesthetic and practical point of view, to have certain deficiencies.}) This is made clearer by the following;\\

\begin{defn}{Asymptotic Degenerations}\\

Let $\{C_{t}\}_{t\in P^{1}}$ be a family of plane projective algebraic curves of degree $m\geq 2$, with the following properties, relative to a fixed coordinate system $\{x,y\}$;\\

(i). The intersection of each curve $C_{t}$ in the family with the axis $x=0$ consists of a fixed set of $m$ distinct points $\{(0,a_{1}),\ldots,(0,a_{j}),\ldots,(0,a_{m})\}$ in finite position, distinct from $(0,0)$. (\footnote{It is, on occasion, useful to impose the following additional assumptions;\\

(a). Using projective coordinates $\{X,Y,W\}$, the intersection of each curve $C_{t}$ in the family, with the axis $W=0$, consists of a fixed set of $m$ different points, distinct from $[1:0:0]$.\\
(b). The tangent line to one of the branches, centred at a point in $(b)$, is fixed, and distinct from the axis $W=0$.\\

This is technically convenient in terms of later terminology, and is required to exclude extremely exceptional cases, which will be mentioned later.}).\\

(ii). For generic ${t\in P^{1}}$, $C_{t}$ is irreducible.\\

(iii). Each curve $C_{t}$ has $m$ distinct non-singular branches, centred at the points $\{(0,a_{1}),\ldots,(0,a_{m})\}$, and, the tangent lines $\{l_{a_{1}},\ldots,l_{a_{m}}\}$ at these branches are all fixed and distinct from the axis $x=0$.\\

We call such a family of curves an asymptotic degeneration.

\end{defn}

\begin{lemma}{Existence of Asymptotic Degenerations}\\

Any nodal curve $C$ is part of a non-trivial asymptotic degeneration.

\end{lemma}

\begin{proof}

Given a plane nodal curve $C$, the existence of asymptotic degeneration of $C$, follows from;\\

\noindent (i). Lemma 1.1.\\
 (ii). Dimension calculations for the space of such curves, with fixed degree and number of nodes.\\

 The reader should look at \cite{Sernesi} and \cite{Zariski} for more details about the calculations in (ii). Any such curve $C$ is part of an irreducible space of dimension $3m-1+k$, where $m=deg(C)$, and $k$ is a non-negative integer depending on the number of nodes. The condition on asymptotes imposes at most $m+1\leq 3m-1\leq 3m-1+k$ conditions on such a space, $(*)$. This follows from the observation that the condition on tangent lines in 3.28(iii) is a consequence, $(**)$, of 3.28(i), 3.28(ii), Lemma 1.1 and the extra single condition that the tangent line to one branch, centred along the axis $x=0$ is fixed. The proof of $(**)$ is straightforward, suppose that we have a family of curves $\{C_{t}\}_{t\in P^{1}}$, satisfying the conditions 3.28(i),(ii), with the branches of the generic curve, centred at $\{(0,a_{1}),\ldots,(0,a_{m})\}$, transverse to $x=0$. Using an analytic representation of a given branch, at $(0,a_{j})$, see \cite{depiro6}, the gradient of the tangent line at $(0,a_{j})$ is given by the rational function $\theta_{j}(t)=-{f_{x}(x,y;t)\over f_{y}(x,y;t)}|_{(0,a_{j})}$. If the condition on tangent lines in 3.28(iii) fails, then $\theta_{j_{0}}(t)$ is non-constant, for some $1\leq j_{0}\leq m$. In this case, by compactness of the parameter space $P^{1}$, it attains the value $\infty$, for finitely many $\{t_{1},\ldots,t_{r}\}$. This implies that $f_{y}(x,y;t_{j})=0$, for $1\leq j\leq r$, and the branch, at $(0,a_{j_{0}})$ is either tangent to the axis $x=0$, or singular for the curve $C_{t_{j}}$. By Bezout's theorem, it follows that $C_{t_{j}}$ either has at most $m-1$ distinct points of intersection with the axis $x=0$, contradicting the condition 3.28(i), or contains the axis $x=0$ as a component. This last case is excluded by the condition that the tangent line to one branch, centred along the axis $x=0$ is fixed. The stipulation that the curve does not contain the point $(0,0)$ is simply achieved by a translation.(\footnote{To obtain the extra exceptional requirements in footnote 11, one needs to impose a further $m+1$ conditions. Lemma 1.1 allows one to find two distinct axes, which cut the given curve transversely. The existence of such \emph{strongly asymptotic} degenerations follows again from the fact that $2(m+1)\leq 3m-1\leq 3m-1+k$, if $m\geq 2$. The case $m=1$ is just a line.})

\end{proof}

\begin{lemma}{Uniformity of Asymptotic Degenerations}\\

Let $\{C_{t}\}_{t\in P^{1}}$ be an asymptotic degeneration, of fixed degree, then, there exists a corresponding algebraic family of plane curves, $\{D_{r}: r\in R \}$, in a fixed coordinate system $\{x,z\}$ with $R\rightarrow P^{1}$ a finite cover, such that;\\

(i). For generic $r\in R$, $D_{r}$ is irreducible, $(0,0)\in D_{r}$, and $(D_{r},(0,0))$ is an etale cover of $(A^{1},(0,0))$.\\

(ii). For non-generic $r_{\infty}\in R$, corresponding to $t_{\infty}$, with $C_{t_{\infty}}$ having components $\{C_{t_{\infty}}^{1},\ldots,C_{t_{\infty}}^{k}\}$,  $D_{r_{\infty}}$ is a union of component curves $\{D_{r_{\infty}}^{1},\ldots,D_{r_{\infty}}^{l}\}$, with $l\geq k$.\\

There exists an algebraic family of parameters, $s\in S$, with $S\rightarrow R\rightarrow P^{1}$ a finite cover, and a sequence of morphisms $\{\eta_{1,s},\ldots,\eta_{j,s},\ldots,\eta_{m,s}:s\in S\}$, with the properties;\\

(iii). For corresponding generic $s\in S$, $r\in R$, $t\in P^{1}$;\\

(a). $\eta_{j,s}|D_{r}$ is birational, of fixed degree, and $\eta_{j,s}(0,0)=(0,a_{j},0)$, relative to the coordinate systems $(x,z)\subset (x,y,z)$.\\

(b). The sequence of image curves $\{C_{1,r,s},\ldots,C_{j,r,s},\ldots,C_{m,r,s}\}$ are distinct irreducible covers of $C_{t}$, using the projection $pr_{1}$ from the coordinate system $\{x,y,z\}$ to $\{x,y\}$.\\

(iv). For corresponding non-generic $(s_{\infty},r_{\infty},\infty)\in S\times R\times P^{1}$, the reduced, (\footnote{By which we mean the functions contain no factors, vanishing or taking the value $\infty$, entirely on a component of $D_{r_{\infty}}$; removing such factors may lower their degree.}), functions $\{\eta_{1,s_{\infty}},\ldots,\eta_{j,s_{\infty}},\ldots,\eta_{m,s_{\infty}}\}$ restrict to morphisms on each component $\{D^{1}_{r_{\infty}},\ldots,D^{i}_{r_{\infty}},\ldots,D^{t}_{r_{\infty}}\}$ of $D_{r_{\infty}}$, with the properties that;\\

(a). $\eta_{j,s_{\infty}}|D^{i}_{r_{\infty}}$ is birational.\\

(b). Each image curve $C_{i,j,\infty}$ is a cover of an irreducible component of $C_{\infty}$.\\

(c). The intersections $(C_{i_{0},j_{0},\infty}\cap C_{i_{1},j_{1},\infty})$ are finite for $(i_{0},j_{0})\neq (i_{1},j_{1})$.\\

(d). For any given $j$, every irreducible component of $C_{\infty}$ is covered by an image curve $C_{i,j,\infty}$.\\

(v). For $1\leq j\leq m$, the family of image curves $\{C_{j,r,s}\}$, including the non-generic limit curves $\{C_{i,j,\infty}\}$, form an irreducible closed variety over a $1$-parameter space $S_{j}$, with $S\rightarrow S_{j}\rightarrow R\rightarrow P^{1}$ finite covers.
\end{lemma}

\begin{proof}
Let $P^{N}$ denote the parameter space for plane curves of degree $n$, in the coordinate system $\{X,Z,W\}$, where $N={n(n+3)\over 2}$. We, first, observe that there exists a constructible set $V\subset P^{N}$, with the properties that, for $\overline{r}\in V$;\\

 (i).   The corresponding curve $D_{\overline{r}}$ is irreducible.\\

 (ii).  $D_{\overline{r}}$ passes through $(0,0)$, in the coordinate system $(x,z)$, where $\{x={X\over W}, z={Z\over W}\}$.\\

 (iii). $(D_{\overline{r}},(0,0))$ defines an etale cover of $(A^{1},0)$.  $(*)$\\

Let $F_{n}(X,Z,W,\overline{r})=\sum_{i+j+k=n}r_{ijk}X^{i}Y^{j}W^{k}$ enumerate all plane curves of degree $n$, and let $F_{p}$ be similar polynomials, enumerating all plane curves of degree $1\leq p<n$. Let\\

$\Phi_{1}(\overline r)\equiv\neg\exists_{1\leq p<n}\overline x_{1}\ldots\overline x_{p}\ldots\overline x_{n-1}\exists_{1\leq q<n}\overline y_{1}\ldots\overline y_{q}\ldots\overline y_{n-1}$\\

\indent \ \ \ \ \ \ \ \ \ \ \ $[\bigvee_{p+q=n}F_{n}(x,z,\overline r)=F_{p}(x,z,\overline x_{p})F_{q}(x,z,\overline y_{q})]$\\

By standard properties, $\Phi_{1}$ defines a constructible subset $V_{1}\subset P^{N}$, such that, $(i)$ of (*), holds for ${\overline r}\in V_{1}$. Let\\

$\Phi_{2}(\overline r)\equiv [\Phi_{1}(\overline r)\wedge F_{n}(0,0,\overline r)=0]$\\

Then $\Phi_{2}$ defines a constructible set $V_{2}\subset V_{1}\subset P^{N}$, for which $(ii)$ of (*), holds as well. Let\\

$\Phi(\overline r)\equiv [\Phi_{2}(\overline r)\wedge{{\partial}F\over {\partial}Z}(X,Z,W,\overline r)|_{[0:0:1]}\neq 0]$\\

It is an easy exercise, using the definition of an etale morphism, to see that $\Phi$ defines a constructible set $V\subset V_{2}$, for which all the conditions of (*) are satisfied. We adopt the coordinate notation of Definition 3.1. Let $P^{2N_{j}}$ be a parameter space for the set of rational functions of degree $n_{j}$ on $P^{2}$, where $N_{j}={n_{j}(n_{j}+3)\over 2}$. We observe, that, for a pair $(t,\overline{r})\in (P^{1}\times V)$, there exists at most one rational function $\eta_{j}(x,z)$, of fixed degree $n_{j}$,(\footnote{Having a uniquely determined parameter $\overline{s_{j}}\in P^{2N_{j}}$}), with the properties that;\\

(i). $\eta_{j}(x,z)$ is defined at $(0,0)$ and $\eta_{j}(0,0)=a_{j}$\\

(ii). The image curve $C_{j,t,\overline{r}}$, defined as the projective closure of;\\

$\{(x,y,z):y=\eta_{j}(x,z), (x,z)\in D_{\overline{r}}\}$\\

is a cover of $C_{t}$, using the projection $pr_{1}$, (**), (\footnote{The conic projection on $P^{3}$ is defined, excluding the exceptional point $[0:0:1:0]$, in the case that a curve passes through this point, one still obtains a well defined projected curve, see the paper \cite{depiro6} for the explicit construction.})\\

For a pair $(x,z)$ belonging to the infinitesimal neighborhood $({\mathcal V}_{(0,0)}\cap D_{\overline{r}})$, there can clearly exist at most one possible point $(x,y,z)$, for which $(x,y)$ belongs to $({\mathcal V}_{(0,a_{j})}\cap C_{t})$, by the definition 3.28(iii). Hence, $\eta_{j}$ is uniquely determined on this neighborhood, and, therefore, on the curve $D_{\overline{r}}$, showing (**). By Lemma 3.21, for generic $t\in P^{1}$, $(\dag)$, and appropriate $\{n,n_{j}\}$, there exists $\overline{r}\in V$ and $\overline{s_{j}}\in P^{2N_{i}}$, such that (**) is satisfied for the triple $(t,\overline{r},\overline{s_{j}})$. Let $G_{n_{j}}(X,Z,W,\overline{z_{i}})$ enumerate all rational functions of degree $n_{j}$, and let $\{r,s_{1},\ldots,s_{j},\ldots,s_{m}\}$ denote new parameters, with each $s_{j}$ abbreviating $\overline{s_{j}}$, and $r$ abbreviating $\overline{r}$. Let\\

$\Psi_{j}(t,r,s_{j})\equiv[\Phi(r)\wedge G_{n_{j}}(0,0,s_{j})=a_{j}\wedge((F_{n}(x,z,r)=0)\rightarrow$\\

\indent \indent \indent \indent \indent \indent \indent $C_{t}(x,G_{n_{j}}(x,z,s_{j})))]$\\

and;\\

$\Psi(t,r,s_{1},\ldots,s_{j},\ldots s_{m})\equiv\wedge_{1\leq j\leq m}\Psi_{j}(t,r,s_{j})$\\

By elementary considerations, $\Psi$ defines a constructible subset $U$ of $(P^{1}\times V\times P^{2N_{1}}\times\ldots P^{2N_{i}}\times\ldots P^{2N_{m}})$. We let\\

$\Theta(t)\equiv\exists r\exists s_{1}\ldots\exists s_{j}\ldots \exists s_{m}\Psi(t,r,s_{1},\ldots,s_{j},\ldots,s_{m})$.\\

By $(\dag)$, $\Theta$ defines an open subset of $P^{1}$. We define $S={\overline U}$ to be the closure of $U$ in $(P^{1}\times V\times P^{2N_{1}}\times\ldots P^{2N_{j}}\times\ldots P^{2N_{m}})$, and also $R\subset (P^{1}\times \overline V)$ to be the projection $pr(S)$. Without loss of generality, we can assume that $S$ and $R$ are irreducible, with the properties;\\

(i). $pr:S\rightarrow R\rightarrow P^{1}$ is onto.\\

(ii). For a generic pair $(t,r)\in R$, there exists a unique corresponding triple $(t,r,\overline{s})\in S$\\

(iii).For a generic pair $(t,r)\in R$, $r\in V$. (***)\\

where $\overline s=(s_{1},\ldots,s_{j},\ldots,s_{m})$. This follows from elementary considerations on dimension, the above mentioned property of $\Theta$, and $(**)$. We assume that $dim(S)\geq 2$. We then make the simple observation that, for a suitable embedding of $S$ in $P^{L}$, it is possible to find a hyperplane $\mathcal{H}\subset P^{L}$, with the properties that;\\

(a). $(\mathcal{H}\cap S)$ is irreducible and $(dim(\mathcal{H})\cap S)=dim(S)-1$.\\

(b). $pr(\mathcal{H}\cap S)=P^{1}$\\

(c). $\mathcal{H}$ passes through a generic triple $(t,r,\overline{s})\in S$.\\

This follows from a generalised version of Bertini's Theorem, see \cite{Ha} or \cite{K}. It follows, by induction, that we can obtain the statement (***), with the additional property that $dim(S)=1$. We then rephrase the property $(iii)$ of $(*)$;\\

(iii)'. For generic $t\in P^{1}$, the fibre $R(t)\subset V$.\\

The statements $(i)$ and $(iii)(a),(b)$ of the theorem now follow immediately from (***)(i),(ii),(iii), the distinctness claim following immediately from the fact that the functions $\eta_{j,s}$ are different on the curve $D_{r}$, by $(**)(i)$. We now consider the non-generic cases in the theorem. We abbreviate the tuple $\overline{s}$ to $s$ and let;\\

$S_{j}(t,r,s_{j})\equiv\exists s_{1}\ldots\exists s_{i\neq j}\ldots\exists s_{m}S(t,r,s_{1},\ldots,s_{i},\ldots,s_{m})$\\

and $U_{j}=pr_{j}(U)$, so that $U_{j}\subseteq S_{j}$. We consider the following varieties $W_{j},T_{j}\subseteq S_{j}\times P^{3}$ defined by;\\

$W_{j}(t,r,s_{j},x,y,z)\equiv U_{j}(t,r,s_{j})\wedge D_{r}(x,z)=0\wedge g_{s_{j}^{2}}(x,z)\neq 0\\
\indent \ \ \ \ \ \ \ \ \ \ \ \ \ \ \ \ \ \ \ \ \ \ \ \ \ \ \ \wedge {g_{s_{j}^{1}}(x,z)\over g_{s_{j}^{2}}(x,z)}=y$\\

$T_{j}(t,r,s_{j},x,y,z)\equiv U_{j}(t,r,s_{j})\wedge g_{s_{j}^{2}}(x,z)\neq 0\wedge{g_{s_{j}^{1}}(x,z)\over g_{s_{j}^{2}}(x,z)}=y$\\

where $\{g_{s_{j}^{1}},g_{s_{j}^{2}}\}$ are homogeneous polynomials, of degree $n_{j}$, represented by the parameter $s_{j}$. $T_{j}$ is irreducible, as $U_{j}\times A^{2}$ is irreducible and a generic fibre $T_{j}(t,r,s_{j},x,z)$ consists of a single point. $W_{j}$ is irreducible, as $U_{j}$ is irreducible, and a generic fibre $W_{j}(t,r,s_{j})$ corresponds to the irreducible image curve $C_{j,r,s}$ considered in case $(iii)(b)$. It follows that the Zariski closures $\overline{W_{j}},\overline{T_{j}}$, which we again denote by $W_{j},T_{j}\subseteq S_{j}\times P^{3}$, are also irreducible. We clearly have that $W_{j}\subset T_{j}$. It follows, by dimension considerations, that, for non-generic $(t_{\infty},r_{\infty},s_{j,\infty})$, the fibre $W_{j}(t_{\infty},r_{\infty},s_{j,\infty})$ defines a closed, possibly not irreducible, variety of dimension $1$ $C_{j,t_{\infty},r_{\infty}}\subset P^{3}$, (\footnote{It is conceivable that the variety could have extraneous points, that is irreducible components of dimension $0$. We will refer to the variety as a curve in the course of the proof.}). The aim of case $(iv)$ in the theorem, is to describe these limit curves $C_{j,t_{\infty},r_{\infty}}$ more explicitly, this will incidentally exclude the scenario of footnote 18.\\

We now examine the behaviour of the morphism $\eta_{j,\infty}(x,z)$, determined by the parameter $s_{j,\infty}$, on the curve $D_{t_{\infty},r_{\infty}}$, and aim to show that the curve $C_{j,t_{\infty},r_{\infty}}$ is obtained as the image of $D_{t_{\infty},r_{\infty}}$, after making certain reductions in the factorisation of $\eta_{j,\infty}$, see footnote 15, (\footnote{Until $(!)$, we make the assumption that $D_{t_{\infty},r_{\infty}}$ does not contain the line $W=0$ as a component in the projective coordinate system $\{X,Z,W\}$}).\\

We consider the closed variety $V_{j}\subset(S_{j}\times P^{2})$ defined by;\\

$V_{j}(t,r,s_{j},x,z)\equiv S_{j}(t,r,s_{j})\wedge D_{r}(x,z)$\\

$V_{j}$ is equidimensional, as the space $Q_{n}$ of degree $n$ curves is irreducible, a proof is given in the paper \cite{depiro5}, the variety $(S_{j}\times P^{2})$ is irreducible, and $V_{j}$ is obtained as the intersection of these varieties within a non-singular total projective space. Irreducibility follows easily from the fact that the generic curve $D_{r}$ is irreducible, see the similar argument in the just quoted proof. Now suppose that $(x_{0},z_{0})$ belongs to some component of the curve $D_{t_{\infty},r_{\infty}}$, then we can find $s_{j,\infty}$ with $V_{j}(t_{\infty},r_{\infty},s_{j,\infty},x_{0},z_{0})$. By Lemma 3.2 of \cite{depiro4}, we can find a generic $(t,r,s_{j},x_{1},z_{1})$, belonging to $V_{j}$, specialising to $(t_{\infty},r_{\infty},s_{j,\infty},x_{0},z_{0})$. As $dim(V_{j})=2$, $dim(x_{1},z_{1}/t,r,s_{j})=1$, in particular, we may assume that $(x_{1},z_{1})$ does not belong to the curve defined by $g_{s_{j}^{2}}(x,z)=0$, using also the fact that $g_{s_{j}^{2}}$ cannot vanish identically on the irreducible curve $D_{t,r}$, by the description $(iii)(a)$ of the theorem. It follows that we can find $y_{1}$, with $W_{j}(t,r,s_{j},x_{1},y_{1},z_{1})$, and, by specialisation, there exists $y_{0}$, with $W_{j}(t_{\infty},r_{\infty},s_{j,\infty},x_{0},y_{0},z_{0})$. This shows that the "projection" of the curve $C_{j,t_{\infty},r_{\infty}}$ contains the curve $D_{t_{\infty},r_{\infty}}$, (****).\\

As the variety $T_{j}$ is irreducible and $W_{j}\subset T_{j}$, by Lemma 3.2 of \cite{depiro4}, if a point $(x_{0},y_{0},z_{0})$ lies on the limit curve $C_{j,t_{\infty},r_{\infty}}$, there exists a generic, (*****), $(t,r,s_{j},x_{1},y_{1},z_{1})$ of $T_{j}$, specialising, (******), to $(t_{\infty},r_{\infty},s_{j,\infty},x_{0},y_{0},z_{0})$. We assume that the field $<t_{\infty},r_{\infty},s_{j,\infty},x_{0},y_{0},z_{0}>$ is contained in the base field $L$. By (****), the facts that $dim(S_{j})=1$ and $dim(T_{j})=3$, we can assume that $dim(t,r,s_{j}/L)=1$, and $dim(x_{1},z_{1}/L)=dim(x_{1},z_{1}/L<t,r,s_{j}>)=2$. We pick independent infinitesimals $\{\epsilon,\delta_{1},\delta_{2}\}$, such that $\{t,r,s_{j}\}\subset L(\epsilon)^{alg}$, $\{x_{1},z_{1}\}\subset L(\delta_{1},\delta_{2})^{alg}$, and the fields $L(\epsilon)^{alg},L(\delta_{1},\delta_{2})^{alg}$ are algebraically disjoint over $L$. By results of \cite{depiro3} or \cite{depiro4}, we can construct a sequence of specialisations;\\

 $\pi_{1}:P(L(\epsilon,\delta_{1},\delta_{2})^{alg})=P(K(\epsilon)^{alg})\rightarrow P(L(\delta_{1},\delta_{2})^{alg})=P(K)$\\

 $\pi_{2}:P(L(\delta_{1},\delta_{2})^{alg})\rightarrow P(L)$\\

 with $\pi_{1}:P(L(\epsilon)^{alg})\rightarrow P(L)$\\

 By the construction of a universal specialisation, see \cite{depiro3}, we can take the composition $\pi=\pi_{2}\circ\pi_{1}$ to satisfy (******). By the definition of $T_{j}$, we have that ${g_{s_{j}^{1}}(x_{1},z_{1})\over g_{s_{j}^{2}}(x_{1},z_{1})}=z_{1}$, we wish to compute the specialisation $\pi(z_{1})$. We first show that;\\

 $\pi_{1}({g_{s_{j}^{1}}(x_{1},z_{1})\over g_{s_{j}^{2}}(x_{1},z_{1})})=\pi_{1}({g_{s_{j,\infty}^{1}}(x_{1},z_{1})\over g_{s_{j,\infty}^{2}}(x_{1},z_{1})})$ (*******)\\

 By a straightforward computation, abbreviating the values\\ $\{g_{s_{j,\infty}^{1}}(x_{1},z_{1}),g_{s_{j,\infty}^{2}}(x_{1},z_{1})\}$ by $\{k_{1},k_{2}\}\subset K$, with $k_{2}\neq 0$, we have;\\

${g_{s_{j}^{1}}(x_{1},z_{1})\over g_{s_{j}^{2}}(x_{1},z_{1})}-{g_{s_{j,\infty}^{1}}(x_{1},z_{1})\over g_{s_{j,\infty}^{2}}(x_{1},z_{1})}={k_{1}+\epsilon_{1}\over k_{2}+\epsilon_{2}}-{k_{1}\over k_{2}}={k_{2}\epsilon_{1}-k_{1}\epsilon_{2}\over k_{2}^{2}+\epsilon_{2}k_{2}}$\\

where $\{\epsilon_{1},\epsilon_{2}\}$ are infinitesimals in $K(\epsilon)^{alg}$. Clearly, the last term in the above expression specialises to zero, showing (*******). We now show that;\\

$\pi_{2}({g_{s_{j,\infty}^{1}}(x_{1},z_{1})\over g_{s_{j,\infty}^{2}}(x_{1},z_{1})})={g_{s_{j,\infty}^{1,red}}(x_{0},z_{0})\over g_{s_{j,\infty}^{2,red}}(x_{0},z_{0})}$ (********)\\

where the $red$ terminology is introduced to indicate that the homogeneous polynomials $\{g_{s_{j,\infty}^{1}},g_{s_{j,\infty}^{2}}\}$ are $reduced$ by any homogeneous factors, vanishing entirely on a component of the projective curve $D_{t_{\infty},r_{\infty}}$. By (****), we can pick $(x_{0},z_{0})$ to lie on a unique component of $D_{t_{\infty},r_{\infty}}$. Let this component be defined by the homogeneous polynomial $h$, suppose that $g_{s_{j,\infty}^{1}}=h^{n}w_{s_{j,\infty}^{1}}$ and $g_{s_{j,\infty}^{2}}=h^{m}w_{s_{j,\infty}^{2}}$, for some integers $m,n\geq 0$. We can also assume that $(x_{0},z_{0})$ does not lie on the intersection $D_{t_{\infty},r_{\infty}}\cap (w_{s_{j,\infty}^{2}}=0)$, as this consists of finitely many points on the given component. By the above construction, $(x_{1},z_{1})$ does not lie on $h=0$, as it does not lie on the curve $D_{t_{\infty},r_{\infty}}$. It must, therefore, follow that;\\

$\pi_{2}({g_{s_{j,\infty}^{1}}(x_{1},z_{1})\over g_{s_{j,\infty}^{2}}(x_{1},z_{1})})=\pi_{2}[h^{n-m}(x_{1},z_{1})]\pi_{2}({w_{s_{j,\infty}^{1}}(x_{1},z_{1})\over w_{s_{j,\infty}^{2}}(x_{1},z_{1})})={w_{s_{j,\infty}^{1}}(x_{0},z_{0})\over w_{s_{j,\infty}^{2}}(x_{0},z_{0})}$\\

where, we have used the fact that $n=m$ and, therefore, $h^{n-m}(x_{1},z_{1})=1$; otherwise, by a simple calculation, using the property $(iii)(b)$, the specialised curve $C_{t_{\infty}}$ would contain either one of the points $[0:0:1]$\\
 or $[0:1:0]$, contradicting Defn 3.28(i). As $(x_{0},z_{0})$ does not lie on any other component of $D_{t_{\infty},r_{\infty}}$, we have that;\\

${w_{s_{j,\infty}^{1}}(x_{0},z_{0})\over w_{s_{j,\infty}^{2}}(x_{0},z_{0})}={r(x_{0},z_{0})g_{s_{j,\infty}^{1,red}}(x_{0},z_{0})\over r(x_{0},z_{0})g_{s_{j,\infty}^{2,red}}(x_{0},z_{0})}={g_{s_{j,\infty}^{1,red}}(x_{0},z_{0})\over g_{s_{j,\infty}^{2,red}}(x_{0},z_{0})}$\\

where $r$ is the maximal homogeneous factor, of both $g_{s_{j,\infty}^{1}}$ and $g_{s_{j,\infty}^{2}}$, vanishing entirely on the remaining components of $D_{t_{\infty},r_{\infty}}$. Hence, (********) is shown. Combining (*******) and (********), we then obtain;\\

$\pi({g_{s_{j}^{1}}(x_{1},z_{1})\over g_{s_{j}^{2}}(x_{1},z_{1})})={g_{s_{j,\infty}^{1,red}}(x_{0},z_{0})\over g_{s_{j,\infty}^{2,red}}(x_{0},z_{0})}$\\

Therefore, the limit curve $C_{j,t_{\infty},r_{\infty}}$ is contained in the union of image curves, defined in $(iv)(b)$. By (****), it consists exactly of the union of these image curves. This show $(iv)(a)$ and $(v)$.\\

Now, suppose that $(x_{0},y_{0},z_{0})$ belongs to the image curve $C_{i,j,\infty}$, therefore to the limit curve $C_{j,t_{\infty},r_{\infty}}$, this means that we can find a generic $(t,r,s_{j},x_{1},y_{1},z_{1})\in W_{j}$, specialising to
$(t_{\infty},r_{\infty},s_{j,\infty},x_{0},y_{0},z_{0})$. As $(x_{1},y_{1},z_{1})$ belongs to $C_{j,r,s}$, by $(iii)(b)$, the projection $pr_{1}(x_{1},y_{1},z_{1})=(x_{1},y_{1})$ belongs to the generic curve $C_{t}$. We consider the following closed variety $Z_{j}\subset(S_{j}\times P^{2})$ defined by;\\

$Z_{j}(t,r,s_{j},x,y)\equiv C_{t}(x,y)$\\

We have that $Z_{j}(t,r,s_{j},x_{1},y_{1})$, hence, by specialisation, we must have $Z_{j}(t_{\infty},r_{\infty},s_{j,\infty},x_{0},y_{0})$, that is, $(x_{0},y_{0})$ belongs to the degenerate curve $C_{\infty}$. It follows that the projection,  $pr_{1}$ of the image curve $C_{i,j,\infty}$ is contained in $C_{\infty}$, (*********).\\

 Conversely, let us suppose that $(x_{0},y_{0})$ belongs to the degenerate curve $C_{\infty}$, then, as $Z_{j}$ is irreducible, by a similar argument to the above, we can find a generic $(t,r,s_{j},x_{1},y_{1})\in Z_{j}$, specialising to $(t_{\infty},r_{\infty},s_{j,\infty},x_{0},y_{0})$. As $(x_{1},y_{1})$ belongs to a generic curve $C_{t}$, by $(iii)(b)$, we can find a lifting $(x_{1},y_{1},z_{1})$, belonging to $C_{j,r,s}$. It follows that $W_{j}(t,r,s_{j},x_{1},y_{1},z_{1})$, and, by specialisation, we must have  $W_{j}(t_{\infty},r_{\infty},s_{j,\infty},x_{0},y_{0},z_{0})$, that is $(x_{0},y_{0},z_{0})$ belongs to the limit curve $C_{j,t_{\infty},r_{\infty}}$. It follows that $C_{\infty}$ is covered by the limit curve $C_{j,t_{\infty},r_{\infty}}$, (**********). We can assume that the reduced function $\eta_{j,\infty}$ is non-constant, in which case, we have, by (**********), that the degenerate curve $C_{\infty}$ is contained in a line, contradicting the assumption on degree in Definition 3.28. It follows that the projection $pr_{1}$ of an image curve $C_{i,j,\infty}$ has dimension $1$. As $C_{i,j,\infty}$ is irreducible, then, using (*********), we obtain $(iv)(b)$. By similar reasoning to the previous argument , using (**********), we obtain $(iv)(d)$, (\footnote{(!) The exceptional case when the limit curve $D_{t_{\infty},r_{\infty}}$ contains the line $W=0$ needs to be considered separately. For a general point $[1:\lambda:0]$ on this line, we can find infinitesimals $\{\epsilon_{1},\epsilon_{2},\epsilon_{3}\}$, with $[1+\epsilon_{1}:\lambda+\epsilon_{2}:\epsilon_{3}]$, defining a general point on the generic curve $D_{t,r}$. The image of this point $[1+\epsilon_{1}:\epsilon_{3}\eta_{j,s}({1+\epsilon_{1}\over \epsilon_{3}}, {\lambda+\epsilon_{2}\over \epsilon_{3}}):\lambda+\epsilon_{2}:\epsilon_{3}]$, on $C_{j,r,s}$, projects to $[1+\epsilon_{1}:\epsilon_{3}\eta_{j,s}({1+\epsilon_{1}\over \epsilon_{3}},{\lambda+\epsilon_{2}\over \epsilon_{3}}):\epsilon_{3}]$, on the generic curve $C_{t}$, and specialises to $[1:0:0]$ on the limit curve $C_{t_{\infty}}$. We can avoid this degenerate scenario, by imposing the extra conditions specified in footnote 12.}).\\

 Now, suppose that $C_{i_{0},j_{0},\infty}$ and $C_{i_{1},j_{1},\infty}$ are distinct image curves, we aim to show $(iv)(c)$. Suppose that $i_{0}\neq i_{1}$. Then, there exist \emph{distinct} components $\{D_{r_{\infty}}^{i_{0}}, D_{r_{\infty}}^{i_{1}}\}$ of the curve $D_{r_{\infty}}$, and (reduced) birational maps;\\

 $\eta_{j_{0},s_{\infty}}:D_{r_{\infty}}^{i_{0}}\leftrightsquigarrow C_{i_{0},j_{0},\infty}$\\

 $\eta_{j_{1},s_{\infty}}:D_{r_{\infty}}^{i_{1}}\leftrightsquigarrow C_{i_{1},j_{1},\infty}$ $(\sharp)$\\

If $(xyz)\in(C_{i_{0},j_{0},\infty}\cap C_{i_{1},j_{1},\infty})$, then, in particularly, by $(\sharp)$, $(xz)\in (D_{r_{\infty}}^{i_{0}}\cap D_{r_{\infty}}^{i_{1}})$. We, therefore, obtain a map;\\

$pr_{2}:(C_{i_{0},j_{0},\infty}\cap C_{i_{1},j_{1},\infty}\cap A^{3})\rightarrow (D_{r_{\infty}}^{i_{0}}\cap D_{r_{\infty}}^{i_{1}}\cap A^{2})$ $(\sharp\sharp)$\\

The fibres of $pr_{2}$ are finite, by the birationality claim in $(\sharp)$, and the image of $pr_{2}$ is finite, by the distinctness of components claim in $(\sharp)$. The result $(iv)(c)$, then follows, by observing, (using footnote 20 to exclude the degenerate case), that the intersection of the curve $C_{i_{0},j_{0},\infty}$ with the projective plane $W=0$, is finite.\\

Now, suppose that $j_{0}\neq j_{1}$, it is sufficient to show that the intersection of the image curves $\{C_{i_{0},j_{0},\infty},C_{i_{0},j_{1},\infty}\}$ is finite. Equivalently, that the reduced functions $\eta_{j_{0},s_{\infty}}$ and $\eta_{j_{1},s_{\infty}}$ define distinct birational morphisms on the component $D_{r_{\infty}}^{i_{0}}$. Suppose that $D_{r_{\infty}}^{i_{0}}$ intersects the axis $x=0$ at a point $(0,c)$. Let $U\subset D_{r_{\infty}}^{i_{0}}$ denote an open subset on which the reduced factors $\{g_{s_{j_{0},\infty}}^{1,red}, g_{s_{j_{0},\infty}}^{2,red},g_{s_{j_{1},\infty}}^{1,red},g_{s_{j_{1},\infty}}^{2,red}\}$ are non-vanishing. As $D_{r_{\infty}}^{i_{0}}$ is irreducible, we can find $(x_{0},z_{0})\in (U\cap{\mathcal V}_{(0,c)})$. We can assume that $dim(x_{0},z_{0}/L)=1$, and find an infinitesimal $\epsilon$ such that $\{x_{0},z_{0}\}\subset L(\epsilon)^{alg}$. In particular, there exist infinitesimals $\{\epsilon_{1},\epsilon_{2}\}\subset L(\epsilon)^{alg}$, with $(x_{0},z_{0})=(\epsilon_{1},c+\epsilon_{2})$. It is sufficient to show that the unique values of the reduced functions $\{\eta_{j_{0},s_{\infty}},\eta_{j_{1},s_{\infty}}\}$ are distinct at $(x_{0},z_{0})$. We consider the irreducible variety $H\subset (S_{j_{0},j_{1}}\times P^{2})$ defined by;\\

$H(t,r,s_{j_{0}},s_{j_{1}},x,z)\equiv S_{j_{0},j_{1}}(t,r,s_{j_{1}})\wedge D_{r}(x,z)$\\

where $S_{j_{0},j_{1}}$ is the irreducible variety, defined by projection of $S$, similarly to the construction of $S_{j}$. The intersection $H_{\epsilon_{1}}\subset (S_{j_{0},j_{1}}\times P^{2})$ is defined by;\\

$H_{\epsilon_{1}}(t,r,s_{j_{0}},s_{j_{1}},x,z)\equiv H(t,r,s_{j_{0}},s_{j_{1}},x,z)\wedge x=\epsilon_{1}$\\

We have that $H_{\epsilon_{1}}(t_{\infty},r_{\infty},s_{j_{0},\infty},s_{j_{1},\infty},x_{0},z_{0})$ and claim that, for\\ $(t,r,s_{j_{0}},s_{j_{1}})\in (S_{j_{0},j_{1}}\cap{\mathcal V}_{(t_{\infty},r_{\infty},s_{j_{0},\infty},s_{j_{1},\infty})})$, there exists $(x_{1},z_{1})\in{\mathcal P}^{2}$, with $(x_{1},z_{1})\in{\mathcal V}_{(x_{0},z_{0})}$ and $H_{\epsilon_{1}}(t,r,s_{j_{0}},s_{j_{1}},x_{1},z_{1})$, $(\sharp\sharp\sharp)$. In particular, it follows that $H_{\epsilon_{1}}$ consists of a finite union $\{H_{\epsilon_{1}}^{1},\ldots,H_{\epsilon_{1}}^{k}\}$ of irreducible curves, each of which projects onto $S_{j_{0},j_{1}}$.\\

The proof of $(\sharp\sharp\sharp)$ is straightforward, using methods from the paper \cite{depiro5}. The tuple $(\overline{r}_{\infty},(1,0,-\epsilon_{1}))$ belongs to the smooth projective space, parametrising intersections between degree $n$ and degree $1$ curves. As the intersection $(D_{r_{\infty}}\cap (x=\epsilon_{1}))$ is finite, we can apply Theorem 2.4 of \cite{depiro5}.\\

We choose $(t,r,s_{j_{0}},s_{j_{1}},x_{1},z_{1})$ generically in $H_{\epsilon_{1}}$. In particular, we have that $dim(t,r,s_{j_{0}},s_{j_{1}}/L(\epsilon)^{alg})=1$. It follows that we can find an infinitesimal $\delta$, with $\{t,r,s_{j_{0}},s_{j_{1}}\}\subset L(\delta)^{alg}$ and $dim(\delta/L(\epsilon)^{alg})=dim(\epsilon/L(\delta)^{alg})=1$. We employ the notation above, to produce the commuting specialisation maps;\\

$(\pi_{1}\circ\pi_{2}):P(L(\epsilon,\delta)^{alg})\rightarrow P(L(\epsilon)^{alg})\rightarrow P(L)$.\\

$(\pi_{2}\circ\pi_{1}):P(L(\epsilon,\delta)^{alg})\rightarrow P(L(\delta)^{alg})\rightarrow P(L)$.\\

We have that $\pi_{2}(x_{1},z_{1})=(x_{0},z_{0})$ and $\pi_{1}(x_{0},z_{0})=(0,c)$. Hence, $(x_{1},z_{1})=(\epsilon_{1},c+\nu_{1})$, where $\nu_{1}$ is an infinitesimal in $L(\epsilon,\delta)^{alg}$. As the specialisations commute, we have that $(\pi_{2}\circ\pi_{1})(x_{1},z_{1})=(\pi_{1}\circ\pi_{2})(x_{1},z_{1})=(0,c)$. It follows that $\pi_{1}(x_{1},z_{1})=(0,c+\delta_{1})\in D_{t,r}$, where $\delta_{1}$ is an infinitesimal in $L(\delta)^{alg}$. Let $\pi_{1}$ define an infinitesimal neighborhood ${\mathcal V}_{(0,c+\delta_{1})}$, and let $(x_{1},z_{1})$ belong to the branch $\gamma$, centred at $(0,c+\delta_{1})$, relative to this neighborhood. Using the condition $(iii)(a)$ and the property of Definition 3.30(iii), by a straightforward adaptation of Lemma 3.27, we have that the values $val_{\gamma}(\eta_{j_{0},s})$ and $val_{\gamma}(\eta_{j_{1},s})$ are distinct, belonging to the fixed set $\{a_{1},\ldots,a_{m}\}\subset L$. For convenience of notation, we will denote them by $a_{j_{0}}$ and $a_{j_{1}}$, (\footnote{Although it is possible that some permutation of values can occur, for a given function $\eta_{j,s}$, at different branches along the axis $x=0$.}).\\

We have that;\\

$\eta_{j_{0},s}(x_{1},z_{1})-\eta_{j_{1},s}(x_{1},z_{1})=(\eta_{j_{0},s}^{\gamma}(0,c+\delta_{1})-\eta_{j_{1},s}^{\gamma}(0,c+\delta_{1}))$\\

\indent \ \ \ \ \ \ \ \ \ \ \ \ \ \ \ \ \ \ \ \ \ \ \ \ \ \ \ \ \ \ \ \ \ \ \ $+(\eta_{j_{0},s}(x_{1},z_{1})-\eta_{j_{0},s}^{\gamma}(0,c+\delta_{1}))$\\

\indent \ \ \ \ \ \ \ \ \ \ \ \ \ \ \ \ \ \ \ \ \ \ \ \ \ \ \ \ \ \ \ \ \ \ \
$-(\eta_{j_{1},s}(x_{1},z_{1})-\eta_{j_{1},s}^{\gamma}(0,c+\delta_{1}))$\\

By a simple calculation, the terms $(\eta_{j_{0},s}(x_{1},z_{1})-\eta_{j_{0},s}^{\gamma}(0,c+\delta_{1}))$ and $(\eta_{j_{1},s}(x_{1},z_{1})-\eta_{j_{1},s}^{\gamma}(0,c+\delta_{1}))$ belong to $Ker(\pi_{1})\cap L(\epsilon,\delta)^{alg}$. Hence;\\

$\pi_{1}(\eta_{j_{0},s}(x_{1},z_{1})-\eta_{j_{1},s}(x_{1},z_{1}))=a_{j_{0}}-a_{j_{1}}$\\

and\\

$\eta_{j_{0},s}(x_{1},z_{1})-\eta_{j_{1},s}(x_{1},z_{1})\ \ \ \ \ \ =a_{j_{0}}-a_{j_{1}}+\nu$\\

with $\nu$ an infinitesimal, belonging to $L(\epsilon,\delta)^{alg}$.\\

It follows, using the limit calculation above, that;\\

$\pi_{2}(\eta_{j_{0},s}(x_{1},z_{1})-\eta_{j_{1},s}(x_{1},z_{1}))=\eta_{j_{0},s_{\infty}}(x_{0},z_{0})-\eta_{j_{1},s_{\infty}}(x_{0},z_{0})$\\

$\indent \ \ \ \ \ \ \ \ \ \ \ \ \ \ \ \ \ \ \ \ \ \ \ \ \ \ \ \ \ \ \ \ \ \ \ =a_{j_{0}}-a_{j_{1}}+\epsilon_{3}$\\

where $\epsilon_{3}$ is an infinitesimal, belonging to $L(\epsilon)^{alg}$. Therefore;\\

$\pi_{1}(\eta_{j_{0},s_{\infty}}-\eta_{j_{1},s_{\infty}})=a_{j_{0}}-a_{j_{1}}\neq 0$\\

In particular, it follows that the reduced functions $\eta_{j_{0},s_{\infty}}$ and $\eta_{j_{1},s_{\infty}}$ take distinct values at the point $(x_{0},z_{0})$. This shows $(iv)(c)$. The proof of $(ii)$ follows immediately from $(iv)(b)(d)$.
\end{proof}

We now show the following fundamental properties of asymptotic degenerations;\\

\begin{lemma}
Let $G$ be the group obtained in the last section. Then the action of $G$ on a generic curve $D_{r}$, with $r\in R$, extends uniformly to almost all curves of the form $D_{r'}$ in the family, with $r'\in R$.
\end{lemma}

\begin{proof}

We can consider $G$ as acting on the curve $D_{r}$, for generic $r\in R$, replacing  $C'$ in Lemma 3.23. We enumerate the parameters needed to define the action in Lemma 3.23(*), as a tuple;\\

 $(\ddot{g}_{1}^{1},\ddot{g}_{1}^{2},\ldots,\ddot{g}_{i}^{1},\ddot{g}_{i}^{2},\ldots,\ddot{g}_{l}^{1},\ddot{g}_{l}^{2})\subset P^{2Nl}$\\

where $Card(G)=l$, which we abbreviate by $\overline{\ddot{g}}=(\ddot{g}_{1},\ldots,\ddot{g}_{i},\ldots,\ddot{g}_{l})$, (\footnote{Having obtained a set of $2l$ rational functions by the standard projective method, we can equalise their degrees, by, for example, multiplying through with factors of the form $(x^{s}g(x,z)+1)$, where $g(x,z)$ is a polynomial vanishing on the curve $D_{r}$. This is mainly technically convenient in terms of notation.}). Let\\

$\Psi_{G}(t,r,\bar h)\equiv$  (a).$\theta_{h_{i}}:D_{r}\leftrightsquigarrow D_{r}, 1\leq i\leq l$\\
\indent \ \ \ \ \ \ \ \ \ \ \ \ \ \ \ \ \ \ (b).$(\theta_{h_{i}}\circ\theta_{h_{j}})=\theta_{h_{i*j}}, 1\leq i,j\leq l$\\
\indent \ \ \ \ \ \ \ \ \ \ \ \ \ \ \ \ \ \ (c).$\theta_{h_{i}}=Id\indent iff\indent  i=1,1\leq i\leq l$\\
\indent \ \ \ \ \ \ \ \ \ \ \ \ \ \ \ \ \ \ (d).$(\theta_{h_{i}})^{-1}=\theta_{h_{i^{-1}}}, 1\leq i\leq l$\\

with $*$ and $-1$ copying the group multiplication and inversion of $G$ onto the labelled set $\{1,\ldots,l\}$, and $\bar h$ being a new variable, substituting $\overline{\ddot{g}}$. The conditions $(a)$ and $(d)$ are formulated simultaneously as follows;\\

$\noindent \exists x_{1}\ldots x_{p}\exists z_{1}\ldots z_{p}\forall x\forall z[\bigwedge_{i=1}^{p}(x\neq x_{i})\wedge \bigwedge_{i=1}^{p}(z\neq z_{i})\wedge D_{t,r}(x,z)\rightarrow \indent \ \ \ \ \ \ \ \ \ \ \ \ \ \ \ \ \ \ \ \ \ \ \ \ \ \ \ \ \ \bigwedge_{i=1}^{w}D_{t,r}(\theta_{h_{i}}(x,z))\wedge\bigwedge_{i=1}^{w}(\theta_{h_{i^{-1}}}\circ\theta_{h_{i}})(x,z)=(x,z)]$\\

The formulation of conditions $(b)$ and $(c)$is a similar exercise, which we leave to the reader. By Lemma 3.23, the condition $\psi_{G}$ holds for the tuple $(t,r,\overline{\ddot{g}})$. Imitating the construction above of the variety $S$, in Lemma 4.3, we can find an irreducible closed variety $E\subset R\times P^{2Nl}$, with $E\rightarrow R$ a finite cover, and open subsets $U_{G}\subset E$ and $V_{G}\subset D$, for which the formula $\Psi_{G}$ holds and $pr^{-1}(V_{G})\subset U_{G}$ respectively. If $(t',r')\in V_{G}$, then the above construction provides a Galois group action on the curve $D_{r'}$. By imitating the proof of Lemma 3.23, and using conditions $(c)$ and $(d)$, one can assume that such an action has the generic transitivity property.
\end{proof}

As before, we consider the non-generic case, adopting the terminology from Lemma 4.4;\\

\begin{lemma}
Let $(t_{\infty},r_{\infty},\bar h_{\infty})$ be a non-generic element of $E$, then, the reduced functions $\{\theta^{red}_{h_{1,\infty}},\ldots,\theta^{red}_{h_{l,\infty}}\}$ restrict to morphisms from the curve $D_{t_{\infty},r_{\infty}}$ to itself, satisfying conditions $(a)-(d)$ of the previous lemma, and the generic transitivity property.
\end{lemma}

\begin{proof}
We consider the closed variety $J\subset (E\times P^{2})$, defined by;\\

$J(t,r,\bar h,x,z)\equiv E(t,r,\bar h)\wedge D_{t,r}(x,z)$\\

By similar arguments to the above, $J$ is irreducible. We also consider the variety $J_{G}\subset(E\times P^{2})$, defined by;\\

$J_{G}(t,r,\bar h,x,z)\equiv V_{G}(t,r)\wedge J(t,r,\bar h,x,z)\wedge\exists z_{1}\ldots z_{i}\ldots z_{l}(z=h_{i}\centerdot z_{i})$ (\footnote{The reader should supply the extra conditions necessary to exclude the finite sets of points for which the action $\centerdot$ of $G$ is undefined, a similar argument being used in Lemma 4.3})\\

We let ${\overline J_{G}}$ be the Zariski closure of $J_{G}$ in $(E\times P^{2})$. By the construction of $V_{G}$, a generic fibre ${\overline J_{G}}(t,r,\bar h)$ is exactly the curve $D_{t,r}\subset P^{2}$. As $V_{G}$ and $D_{t,r}$ are both irreducible, it follows that $J_{G}$, and, therefore, ${\overline J_{G}}$ are both irreducible. Hence, $J={\overline J_{G}}$, $(*)$. Now suppose, for contradiction, that the reduced function $\theta^{red}_{h_{1,\infty}}$ fails to define a morphism from the curve $D_{t_{\infty},r_{\infty}}$ to itself. Then the intersection $D_{t_{\infty},r_{\infty}}\cap \overline{\theta^{red}_{h_{1,\infty}}(D_{t_{\infty},r_{\infty}})}$ defines a proper closed subset of $D_{t_{\infty},r_{\infty}}$. We can choose $(x_{0},z_{0})$ from $D_{t_{\infty},r_{\infty}}$, not lying on this intersection, $(**)$. By $(*)$, we have that   $\overline{J_{G}}(t_{\infty},r_{\infty},\bar h_{\infty},x_{0},z_{0})$. As $\overline{J_{G}}$ is irreducible, we can find a generic $(t,r,\bar h,x_{1},z_{1})$, belonging to $\overline{J_{G}}$, specialising to $(t_{\infty},r_{\infty},\bar h_{\infty},x_{0},z_{0})$. By the definition of $V_{G}$, we can find $(x_{1},z_{2})$, belonging to $D_{t,r}$, with $h_{1}\centerdot z_{2}=z_{1}$. We can assume that $(x_{1},z_{2})$ does not specialise to a point belonging to $D_{t_{\infty},r_{\infty}}\cap q^{red}_{h_{1,\infty}}=0$, $(***)$, where  $q^{red}_{h_{1,\infty}}$ is the reduced denominator term, used in the definition of $\theta^{red}_{h_{1,\infty}}$, see Lemma 3.23. This follows by observing that there are only a finite number of points on the curve $D_{t_{\infty},r_{\infty}}$, for which the reduced denominator term is zero. We can choose the coordinate $x_{0}$ so that the line $x=x_{0}$ has an empty intersection with this finite number of points. Therefore, we may assume $(***)$. It follows, by the limit calculation used in Lemma 4.3, that the specialisation $(x_{0},z_{0})$ of $\theta_{h_{1}}(x_{1},z_{2})$ is $\theta^{red}_{h_{1,\infty}}(x_{0},z_{3})$, with $(x_{0},z_{3})\in D_{t_{\infty,r_{\infty}}}$. This contradicts $(**)$. Hence, the first part of the lemma is shown. We now check conditions $(a)-(d)$. Condition $(a)$ follows from condition $(d)$ and the fact that the set of reduced functions $\{\theta^{red}_{h_{1,\infty}},\ldots,\theta^{red}_{h_{l,\infty}}\}$ define morphisms on the curve $D_{t_{\infty},r_{\infty}}$. Conditions $(b)$ and $(d)$ amount to checking that for a triple $(t,r,\bar g)$, specialising to $(t_{\infty},r_{\infty},\bar g_{\infty})$, for a generic point $(x_{1},z_{1})$ of $D_{t,r}$, specialising to $(x_{0},z_{0})$ of $D_{t_{\infty},r_{\infty}}$, we have that;\\

$sp(\theta_{h_{i}}\circ\theta_{h_{j}}(x_{1},z_{1}))=\theta^{red}_{h_{i,\infty}}\circ\theta^{red}_{h_{j,\infty}}(x_{0},z_{0})$ $(****)$\\

We assume that the three points $\{(x_{1},z_{1}),\theta_{h_{j}}(x_{1},z_{1}),\theta_{h_{i}}\circ\theta_{h_{j}}(x_{1},z_{1})\}$ satisfy $(***)$, by excluding a finite number of the coordinates $x_{0}$. In which case, we have, by repeating the limit calculation in Lemma 4.3, that;\\

$sp(\theta_{h_{i}}\circ\theta_{h_{j}}(x_{1},z_{1}))=\theta^{red}_{h_{i,\infty}}(sp(\theta_{h_{j}}(x_{1},z_{1})))\\
\indent \ \ \ \ \ \ \ \ \ \ \ \ \ \ \ \ \ \ \ \ \ \ \ \ \ =\theta^{red}_
{h_{i,\infty}}(\theta^{red}_{h_{j,\infty}}(sp(x_{1},z_{1})))\\
\indent \ \ \ \ \ \ \ \ \ \ \ \ \ \ \ \ \ \ \ \ \ \ \ \ \ =\theta^{red}_{h_{i,\infty}}\circ\theta^{red}_{h_{j,\infty}}(x_{0},z_{0})$\\

as required in $(****)$. One direction of condition $(c)$ is straightforward and left to the reader. We show the converse direction, $(*****)$. We let $E_{i}$, for $1\leq i\leq l$, be the projections of $E$ onto the $l$ factors of $G$. As before, the covers $E_{i}\rightarrow R$ are irreducible and finite. We also introduce the combined parameter spaces;\\

$S_{j,k}(t,r,s_{j},s_{k})\equiv S_{j}(t,r,s_{j})\wedge S_{k}(t,r,s_{k})$, for $1\leq j,k\leq m$\\

$F_{i,j,k}(t,r,h_{i},s_{j},s_{k})\equiv E_{i}(t,r,h_{i})\wedge S_{j,k}(t,r,s_{j},s_{k})$, for $1\leq i\leq l$,\\
\indent \ \ \ \ \ \ \ \ \ \ \ \ \ \ \ \ \ \ \ \ \ \ \ \ \ \ \ \ \ \ \ \ \ \ \ \ \ \ \ \ \ \ \ \ \ \ \ \ \ \ \ \ \ \ \ \ \ \ \ \ \ \ \ \ \  $1\leq j,k\leq m$\\

Again, the covers $S_{j,k}\rightarrow R$ and $F_{i,j,k}\rightarrow R$ are irreducible and finite. We recall, from Lemma 3.23, that, for a generic parameter $(t,r)\in R$, any $g\in G$, acting on the curve $D_{t,r}$, induces an action on the associated set of flashes $\{C_{1,r,s},\ldots, C_{j,r,s},\ldots C_{m,r,s}\}$. If $\theta_{g}\neq Id$, then, by definition of the Galois group, this action is non-trivial. By the enumeration in the previous lemma, we denote $g$ by $g_{2}\neq g_{1}$, and ,without loss of generality, we can suppose that $(g_{2}\centerdot C_{1,r,s})=C_{2,r,s}$. We recall the definition of the variety $T_{1}\subset F_{2,1,2}\times P^{3}$, considered in Lemma 4.3;\\

$T_{1}(t,r,h_{2},s_{1},s_{2},x,y,z)\equiv F_{2,1,2}(t,r,h_{2},s_{1},s_{2})\wedge D_{t,r}(x,z)\wedge y=\eta_{1,s_{1}}(x,z)$\\

\noindent and define the following variety $T_{2,g_{2}}\subset F_{2,1,2}\times P^{3}$;\\

$T_{2,g_{2}}(t,r,h_{2},s_{1},s_{2},x,y,z)\equiv F_{2,1,2}(t,r,h_{2},s_{1},s_{2})\wedge D_{t,r}(x,z)\\
\indent \ \ \ \ \ \ \ \ \ \ \ \ \ \ \ \ \ \ \ \ \ \ \ \ \ \ \ \ \ \ \ \ \ \ \ \ \wedge y=(\eta_{2,s_{2}}\circ\theta_{2,h_{2}})(x,z)$ (\footnote{It is necessary to supply extra conditions, in order to ensure both varieties are well defined. This was done in Lemma 4.3, for the variety $T_{1}$. For the variety $T_{2,g_{2}}$, one should restrict the parameter $x$, by adding the conditions that $x\notin Disc_{x'}(D_{t,r}(x',z),q_{h_{2}^{2}}(x',z)=0)$ and $x\notin Disc_{x'}(D_{t,r}(x',z),g_{s_{2}^{2}}(x',z)=0)$})\\

As we showed in Lemma 4.3, the Zariski closure $\overline{T_{1}}$ is irreducible. By the condition that $(\theta_{2,h_{2}})^{*}(\eta_{2,s_{2}})=\eta_{1,s_{1}}$ on the curve $D_{t,r}$, for a generic choice of $(t,r,h_{2},s_{1},s_{2})$, we clearly have that $T_{1}=T_{2,g_{2}}$, and the closure $\overline{T_{2,g_{2}}}$ is irreducible. By results in Lemma 4.3, a degenerate fibre  $\overline{T_{1}}(t_{\infty},r_{\infty},h_{2,\infty},s_{1,\infty},s_{2,\infty})$ is exactly the limit curve, $C_{1,t_{\infty,r_{\infty}}}$. Suppose $(x_{0},y_{0},z_{0})\in C_{1,t_{\infty,r_{\infty}}}$, then we can choose a generic point $(t,r,h_{2},s_{1},s_{2},x_{1},y_{1},z_{1})$ on the variety $\overline{T_{2,g_{2}}}$, specialising to the tuple\\ $(t_{\infty},r_{\infty},h_{2,\infty},s_{1,\infty},s_{2,\infty},x_{0},y_{0},z_{0})$. By the definition of $T_{2,g_{2}}$, we have that $y_{1}=(\eta_{2,s_{2}}\circ\theta_{2,h_{2}})(x_{1},z_{1})$. Repeating the above calculation;\\

$y_{0}=sp(y_{1})=sp((\eta_{2,s_{2}}\circ\theta_{2,h_{2}})(x_{1},z_{1}))\\
\indent \ \ \ \ \ \ \ \ \ \ \ \ \ \ \ =\eta_{2,s_{2,\infty}}^{red}(sp(\theta_{2,h_{2}}(x_{1},z_{1})))\\
\indent \ \ \ \ \ \ \ \ \ \ \ \ \ \ \ =\eta_{2,s_{2,\infty}}^{red}(\theta_{2,h_{2,\infty}}^{red}(sp(x_{1},z_{1})))\\
\indent \ \ \ \ \ \ \ \ \ \ \ \ \ \ \ =(\eta_{2,s_{2,\infty}}^{red}\circ\theta_{2,h_{2,\infty}}^{red})(x_{0},z_{0})$\\

where, we take the usual precaution, of restricting the range of the coordinate $x_{0}$, in order to apply the calculation of Lemma 4.3. The result implies that, in the limit, $(\theta_{2,h_{2,\infty}}^{red})^{*}(\eta_{2,s_{2,\infty}})=\eta_{1,s_{1,\infty}}$ on the curve $D_{t_{\infty},r_{\infty}}$. By Lemma 4.3(c), we have that the rational functions $\{\eta_{1,s_{1,\infty}}, \eta_{2,s_{2,\infty}}\}$ are distinct. Hence, we conclude, that $\theta_{2,h_{2,\infty}}^{red}\neq Id$. This shows $(*****)$.\\

We finally show the generic transitivity property, for the curve $D_{t_{\infty},r_{\infty}}$. Consider the variety $F\subset R\times P^{1}\times P^{2}$;\\

$F(t,r,x_{0},x_{1},X,Z,W)\equiv D_{t,r}(X,Z,W)\wedge x_{0}X=x_{1}W$\\

\noindent which defines a finite cover $F\rightarrow (R\times P^{1})$. Let $Q_{1}=[0:1:0]$, and let $k=mult_{(t_{\infty},r_{\infty},a,Q_{1})}(F/(R\times P^{1})$, where $a$ is generic over the parameter $(t_{\infty},r_{\infty})$, and the parameters defining $R$. For generic $(t,r,a)\in (R\times P^{1})$, we have that $I_{Q_{1}}(D_{t,r},x=a)=n-l$, where $Card(G)=l$ and $deg(D_{t,r})=n$. This follows, from a simple application of Bezout's theorem, and the previously observed fact, that the curve $D_{t,r}$ has $l$ points in finite position, along the axis $x=a$. It follows that $I_{Q_{1}}(D_{t_{\infty},r_{\infty}},x=a)=n-l+(k-1)=n-(l-k+1)$. This follows from summability of specialisation, see \cite{depiro5}. Applying Bezout's theorem again, the curve $D_{t_{\infty},r_{\infty}}$ has exactly $(l-k+1)$ points in finite position along the generic axis $x=a$. If $k>1$, we obtain $l_{1}<l$ points in finite position along a generic axis. By the counting argument, used in Lemma 3.23, properties $(a)-(d)$ shown above, and the fact that $Card(G)=l$, we can find
a reduced function $\theta^{red}_{h_{i,\infty}}$, for some $(i>2)$, acting as the identity on $D_{t_{\infty},r_{\infty}}$, this contradicts property $(c)$. Hence, $k=1$, and the curve $D_{t_{\infty},r_{\infty}}$ has exactly $l$ points in finite position along a generic axis. Transitivity of the $G$-action, then follows by applying the methods of Lemma 3.23, (\footnote{For a general $1$-dimensional family of curves of fixed projective degree, all of which are irreducible, it is not necessarily true, that $deg(C_{t_{0}})=deg(C_{t_{1}})$, for a pair $(t_{1},t_{2})$, where $deg(C)$ denotes $deg(L(C)/L(x))$. It is an interesting question to determine the conditions for which this property holds. The family parametrised by $R$ is close to having this property, depending on whether the degenerate fibres can be chosen as irreducible.})
\end{proof}

We now observe the following uniformity in results from the previous section.

\begin{lemma}
Let $g(x,z)$ be chosen, satisfying the conditions of Lemma 3.4, with corresponding $\theta_{g}$, then $\theta_{g}$ defines a birational morphism on each curve in the family $D_{r}$, for $r\in R$. Moreover, we can find a new space of parameters $R'$, satisfying conditions 4.3(i),(ii), $S'$ satisfying conditions 4.3(iii),(iv),(v), such that the algebraic family $D_{r'}$, for $r'\in R$, consists of the images of $\theta_{g}:D_{r}\leftrightsquigarrow D_{r'}$, the functions $\eta_{j,r',s'}=(\theta_{g}^{-1})^{*}(\eta_{j,r,s})$, for generic $(r,s)$, $\eta_{j,r'_{\infty},s'_{\infty}}^{red}=(\theta_{g}^{-1})^{*}(\eta_{j,r_{\infty},s_{\infty}}^{red})$, for degenerate $(r_{\infty},s_{\infty})$ and $E'$, satisfying Lemma 4.4, with $\theta_{r',\bar h'}=(\theta_{g}^{-1})^{*}(\theta_{r,\bar h})$, for generic $(r,\bar h)$, and $\theta_{r'_{\infty},\bar h'_{\infty}}^{red}=(\theta_{g}^{-1})^{*}(\theta_{r_{\infty},\bar h_{\infty}}^{red})$, for degenerate $(r_{\infty},\bar h_{\infty})$. There exists a naturally defined total morphism from the unprimed to the primed parameter spaces, $\Gamma_{g}:(R,S,E)\rightarrow (R',S',E')$, such that all relevant diagrams commute.
\end{lemma}
\begin{proof}
The proof is straightforward but lengthy. That $\theta_{g}$ defines a birational morphism on each curve in the family $D_{r}$, for $r\in R$, follows easily from the proof of Lemma 3.4, (\footnote{One can easily exclude the exceptional case that $\theta_{g}$ is not uniformly and birationally defined, namely that a degenerate curve $D_{r_{\infty}}$ has a component $l_{\infty}$, by imposing the extra conditions in footnote 12}). In order to construct the parameter space $R'$, observe that we can write the equation of a general member of the family $D_{r}$, in the form;\\

$\sum_{0\leq i+j\leq n}r_{ij}x^{i}z^{j}=0$ $(*)$\\

Now make the substitution $z=g(x,z')$ in $(*)$ and clear the denominator term, to obtain the expression;\\

$\sum_{0\leq i+j\leq n}r_{ij}x^{i}[a(x)z'+b(x)]^{j}[c(x)z'+d(x)]^{n-j}=0$ $(**)$\\

We can then expand $(**)$, obtaining an expression of the form;\\

$\sum_{0\leq i+j\leq n'}r_{ij}'x^{i}z'^{j}$ $(***)$\\

in which each $r_{ij}'$ term is a linear combination of terms $r_{ij}$, with coefficients depending on the those of the polynomials $\{a(x),b(x),c(x),d(x)\}$. Comparing $(*)$ and $(***)$, we obtain a linear projective morphism $\Gamma'$ from $R\subset P^{N}$ to $R'\subset P^{N'}$, which parametrises an irreducible $1$-dimensional family of curves, with degree $n'$. By construction, we have that $\theta_{g}:D_{r}\leftrightsquigarrow D_{r'}$. The family $R'$ satisfies condition 4.3(i), by the proof of Lemma 3.4. The construction of $S'$ is similar, by substitution of the polynomial $g^{-1}(x,z')$, defining the inverse transformation $\theta_{g}^{-1}$, for $z$, in the rational functions $\{g_{s_{j}^{1}}(x,z),g_{s_{j}^{2}}(x,z)\}$. The statement concerning the transformation of $\eta_{j,r,s}$, for generic $(r,s)$ follows immediately from this method, while the statement concerning the transformation of $\eta_{j_{\infty},r_{\infty},s_{\infty}}^{red}$, for degenerate $(r_{\infty},s_{\infty})$, is a limit calculation, of a form we have considered above. We leave the construction of $E'$ and the associated transformations of the Galois action to the reader. It is clear how the total morphism $\Gamma_{g}$ is obtained from the partial linear projective morphisms on the spaces $\{R,S,E\}$.
\end{proof}

\begin{rmk}
The above construction can be repeated for any abstract birational map $\theta_{g}$, which is uniformly birational on the curves in the family $D_{r}$, for $r\in R$, and preserves the conditions of Lemma 3.3, for a generic curve, see also footnote 3. When referring to birational maps on a degenerate fibre, we will assume that these extra conditions are satisfied. In particular, we always have that $(pr\circ\theta_{g})=pr$ uniformly.

\end{rmk}
As a consequence of the lemma, we are able to extend Definition 3.11, to the case of a degenerate curve $D_{t_{\infty},r_{\infty}}$;\\

\begin{defn}
We define a branch $\gamma$ of $D_{t_{\infty},r_{\infty}}$ to be blue, if there exists a birational map $\theta:D_{t_{\infty},r_{\infty}}\leftrightsquigarrow D_{t_{\infty},r_{\infty}'}$, see previous remark, such that the either/or clause of Definition 3.11 holds, for $D_{t_{\infty},r_{\infty}'}$ and $\theta(\gamma)$. We define a branch $\gamma$ to be silver otherwise.
\end{defn}

\begin{rmk}
This is a good definition, using the same argument as in Remarks 3.12(b). We can also give an equivalent definition of a silver branch, using Remarks 3.12(c).
\end{rmk}

We can now use uniformity of the $G$-action, to simplify the geometry of the degenerate curve.\\

\begin{lemma}
Lemma 3.13 holds for $D_{t_{\infty},r_{\infty}}$, that is the distribution of silver and blue branches along any axis is uniform.
\end{lemma}

\begin{proof}
The result follows immediately from the proof of Lemma 3.13 and uniformity of the $G$-action, Lemma 4.5.

\end{proof}

and to exclude symmetry properties;\\

\begin{lemma}
Lemma 3.34 holds for $D_{t_{\infty},r_{\infty}}$. That is, for a blue branch $\gamma$ in finite position, it is not the case that every function $\{\eta_{1,t_{\infty},r_{\infty}}^{red},\ldots,\eta_{m,t_{\infty},r_{\infty}}^{red}\}$ is symmetric at $\gamma$.
\end{lemma}

\begin{proof}
Using Lemma 3.34, Lemma 4.3(iv)(c) and Lemma 4.5.

\end{proof}
\end{section}
\begin{section}{Asymptotic Degenerations to Lines Passing through a Single Point}

We consider a special case of the families considered in the previous section, for which a degenerate fibre $C_{t_{\infty}}$ of $\{C_{t}\}_{t\in P^{1}}$, consists of $m$ lines, passing through a single point $O$. It is convenient to assign the point $O$ with the coordinates $(k,0)$, for $k\neq 0$, in the $(x,y)$ plane. The $m$ distinct lines must then correspond to $l_{[(0,a_{j}),(k,0)]}$, for $1\leq j\leq m$, with equations $(a_{j}-1)x+ky=0$. The analysis of degenerate flash intersections is easier due to the simplified geometry of this configuration. In particular, the image curves $C_{i,j,\infty}$ of Lemma 4.3, are all contained in the union of planes formed by the $m$ distinct lines and the point $Q_{1}\in P^{3}$. We make the following definition;\\

\begin{defn}
Let $\ell=\{\breve{l}_{1},\ldots,\breve{l}_{m}\}$ enumerate the branches of $C_{t_{\infty}}$, which are centred at $O$.
\end{defn}

\begin{defn}
For the branches in Definition 5.1, we define $\Upsilon_{j}=\{\breve{\gamma}_{1,j},\ldots,\breve{\gamma}_{s,j}\}$ to be the lifting of these branches to the flash $C_{j,\infty}=\bigcup_{1\leq i\leq t}C_{i,j,\infty}$, and also to the curve $D_{t_{\infty},r_{\infty}}$. We define $\Upsilon=\bigcup_{1\leq j\leq m}\Upsilon_{j}$.
\end{defn}

\begin{rmk}
We emphasise the point that, for any $j$, every branch in $\ell$ lifts to a branch in $\Upsilon_{j}$, this follows immediately from Lemma 4.3(iv)(d)
\end{rmk}

We have the same birational invariance of such branches, as in Lemma 3.27;\\

\begin{lemma}
Let $\theta:D_{t_{\infty},r_{\infty}}\leftrightsquigarrow D_{t_{\infty},r'_{\infty}}$ be a birational map, in the sense of Remarks 4.7, then the branches $\{\breve{\gamma},\eta_{j,s_{\infty},t_{\infty}}^{red}(\breve{\gamma})\}$ of $\{D_{t_{\infty},r_{\infty}},C_{j,s_{\infty}}\}$ lift the same branch $\breve{l}$ on $C_{t_{\infty}}$ as $\{\theta(\breve{\gamma}),(\eta_{j,s'_{\infty},t_{\infty}}^{red}\circ\theta\circ pr_{2})(\breve{\gamma})\}$ of $D_{t_{\infty},r'_{\infty}},C_{j,s'_{\infty}}$. In particular, the sets $\Upsilon_{j}$ and $\Upsilon$ are birationally invariant.
\end{lemma}
\begin{proof}
The proof is the same as Lemma 3.27.
\end{proof}

We have a stronger version of Lemmas 3.30 and 3.31;\\

\begin{lemma}
Let $\breve{\gamma}$ be a geometric branch of $D_{t_{\infty},r_{\infty}}$. Then;\\

$val_{\breve{\gamma}}(\eta_{j_{0},s_{\infty},t_{\infty}}^{red})=val_{\breve{\gamma}}(\eta_{j_{1},s_{\infty},t_{\infty}}^{red})$, for $j_{0}\neq j_{1}$, iff $\breve{\gamma}\in(\Upsilon_{j_{0}}\cap\Upsilon_{j_{1}})$\\
 \indent \ \ \ \ \ \ \ \ \ \ \ \ \ \ \ \ \ \ \ \ \ \ \ \ \ \ \ \ \ \ \ \ \ \ \ \ \ \ \ \ \ \ \ \ \ \ \ \ \ \ \ \ \ \ iff $\breve{\gamma}\in\bigcap_{1\leq j\leq m}\Upsilon_{j}$\\

 Moreover, $\Upsilon=\Upsilon_{j}$, for $1\leq j\leq m$.

\end{lemma}

\begin{proof}
The first two equivalences follow by using the proof of Lemma 3.30, together with the fact that $D_{t_{\infty},r_{\infty}}$ has a unique singularity at $O$. In order to show the third equivalence, suppose that the branch $\breve{\gamma}$ is in finite position along the axis $x=k$, with $val_{\breve{\gamma}}(\eta_{j_{0},s_{\infty},r_{\infty}}^{red})=val_{\breve{\gamma}}(\eta_{j_{1},s_{\infty},r_{\infty}}^{red})=0$. Suppose that $val_{\breve{\gamma}}(\eta_{j_{2},s_{\infty},r_{\infty}}^{red})=a\neq 0$, for some $j_{2}\notin\{j_{0},j_{1}\}$. We can assume that this value is finite, in which case we obtain that $(k,a)$ belongs to the curve $D_{t_{\infty},r_{\infty}}$. This is clearly impossible. Hence, $val_{\breve{\gamma}}(\eta_{j_{2},s_{\infty},r_{\infty}}^{red})=0$, and, $\breve{\gamma}$ belongs to $\Upsilon_{j_{2}}$. For the final statement, simply observe that if, without loss of generality, $\breve{\gamma}$ belongs to $\Upsilon_{1}$, then $\breve{\gamma}$ is centred along the axis $x=k$, and $val_{\breve{\gamma}}(\eta_{1,s_{\infty},r_{\infty}}^{red})=0$. If $val_{\breve{\gamma}}(\eta_{j,s_{\infty},r_{\infty}}^{red})=a\neq 0$, for $j\neq 1$, we again obtain a point $(k,a)$ belonging to the curve  $D_{t_{\infty},r_{\infty}}$, which is a contradiction. Therefore $\breve{\gamma}\in\bigcap_{1\leq j\leq m}\Upsilon_{j}$, and $\Upsilon=\Upsilon_{j}$, for $1\leq j\leq m$.

\end{proof}

We obtain the following characterisation of branches for the degenerate curve;\\

\begin{lemma}
All the branches of $D_{t_{\infty},r_{\infty}}$ are silver.
\end{lemma}

\begin{proof}
Using the proofs of Lemma 3.36 and Lemma 3.37.
\end{proof}

We now show that we can simultaneously isolate silver branches, and refer to the previous section for notation;

\begin{lemma}
Suppose that $\gamma$ is a silver branch, in finite position, of the generic curve $D_{t,r}$, centred at $(k',a')$, specialising to the point $(k,a)$ of the degenerate curve $D_{t_{\infty},r_{\infty}}$. Then, after a sequence of birational transformations, we can assume that $\gamma$ is isolated and there exists a unique silver branch of $D_{t_{\infty},r_{\infty}}$, centred at $(k,a)$.

\end{lemma}

\begin{proof}
We follow the proof of Lemma 3.14. Let $(k'+x,a'+\eta_{1}(x))$ be a parametrisation of $\gamma$, where $\eta(x)$ is a power series belonging to $K[[x]]$, with $sp(P(K))=P(L)$, of the form;\\

$\eta(x)=\sum_{j=1}^{\infty}a_{j}'x^{j}$\\

Suppose that the sequence $(a',a_{1}',\ldots,a_{j}',\ldots)$ specialises to the sequence $(a,a_{1},\ldots,a_{j},\ldots)$. We define;\\

$g(x,z)={(z-a)\over (x-k)}-{a\over k}$\\

and apply the birational map $\theta_{g}$, satisfying Lemma 4.6. Let $\{\breve{\gamma_{1}},\ldots,\breve{\gamma_{r}}\}$ enumerate the branches of $D_{t_{\infty},r_{\infty}}$, centred at $(k,a)$, then, using Lemma 5.6, the only remaining branches, in finite position along the axis $x=k$, after applying this transformation, are $\{\theta_{g}(\breve{\gamma_{1}}),\ldots,\theta_{g}(\breve{\gamma_{r}})\}$. Their centres depend only on the first coefficient $b_{1}$ of the power series $\{\eta_{1}(x),\ldots,\eta_{r}(x)\}$ defining these branches. We consider the effect of $\theta_{g}$ on $\gamma$.  The power series representation of $\theta_{g}(\gamma)$ is given by;\\

$(k'+x,{(a'-a)+\eta(x)\over (k'-k)+x}-{a\over k})=(k'+x,{\epsilon_{1}+\eta(x)\over\epsilon_{2}+x}-{a\over k})$\\

so, the centre of $\gamma_{1}=\theta_{g}(\gamma)$, specialises to $(k,a_{1}-{a\over k})$. By elementary properties of specialisations, one of the branches $\{\breve{\gamma_{1}},\ldots,\breve{\gamma_{r}}\}$ must therefore be centred at $(k,a_{1}-{a\over k})=(k,c_{1})$, requiring that $b_{1}=a_{1}$. We repeat the process for the new function;\\

$g_{1}(x,z)={(z-c_{1})\over (x-k)}-{c_{1}\over k}$\\

Now the centre of $\theta_{g_{1}}(\gamma_{1})$, specialises to $(k,a_{2}-{{c_{1}\over k}})$. Again, one of the remaining branches in $\{\breve{\gamma_{1}},\ldots,\breve{\gamma_{r}}\}$ must be centred at $(k,a_{2}-{c_{1}\over k})=(k,c_{2})$, requiring that, the second coefficient of the defining series $b_{2}=a_{2}$. We repeat the process, until we are left with a single silver branch $\breve{\gamma_{1}}$ say, along the axis $x=k$, for the curve $D_{t_{\infty},r_{\infty}}$. The original branch $\gamma$ is also isolated, using, for example, preservation of singularities, see Remarks 5.11. Hence, the lemma is shown

\end{proof}

\begin{rmk}
As an immediate consequence of the proof, we see that for the remaining silver branch $\breve{\gamma}$, the original power series representation is given by;\\

$a+\breve{\eta}(x)$, where $\breve{\eta}(x)=\sum_{j=1}^{\infty}a_{j}x^{j}$\\

By applying the finite number of birational transformations, in the previous lemma, to the series $(a'+\eta(x))$ and $(a+\breve{\eta}(x))$, we obtain, in the final situation, that the defining series of $\{\gamma,\breve{\gamma}\}$ are related by specialisation.

\end{rmk}

We make the following definition;\\

\begin{defn}
By a general nodal degeneration, we mean a family of curves $\{C_{t}:t\in P^{1}\}$, satisfying $(i)-(iv)$ of the conditions at the beginning of Section 1, together with $(v')$. $C_{\infty}$ is a union of $m$ distinct lines $\{l_{1},\ldots,l_{m}\}$.\\
 \end{defn}

A desirable property of general nodal degenerations is the following;\\

\begin{defn}
We say that a general nodal degeneration has good specialisation, if, for any given node of the generic curve $C_{t}$, centred at $(a',b')$, and specialising to a point $(a,b)$, formed by the intersection of two lines $\{l_{1},l_{2}\}$ of the degenerate curve $C_{\infty}$; if $\{c',d'\}$ denote the gradients of the tangent lines to the two branches of the node on $C_{t}$, then the specialisations $\{c,d\}$ are the gradients of the two lines $\{l_{1},l_{2}\}$.

\end{defn}

\begin{rmk}
It is always true that the point $(a',b')$ specialises to a point $(a,b)$ formed by the intersection of two lines. This is simply a consequence of the preservation of singularities. More specifcally, if the family of curves is determined by a polynomial;\\

$F(X,Y,W;t)=\sum_{i+j+k=m}a_{ijk}(t)X^{i}Y^{j}Z^{k}$\\

Then the nodes of $C_{t}$ are determined by the intersections of the varieties ${\partial F\over\partial X}={\partial F\over\partial Y}={\partial F\over\partial Z}=0$. For a point $(a',b')$ belonging to this intersection, the specialisation also belongs to this intersection at $t=\infty$. This is clearly a singular point of the variety $C_{\infty}$, hence, a point formed by the intersection of two lines. However, the stronger property concerning specialisation of gradients is not necessarily true; Zariski gives a counterexample in his paper \cite{Zariski}, refuting Severi's original proof of his conjecture.

\end{rmk}

We show how to correct this error;\\

\begin{theorem}
An asymptotic general nodal degeneration has good specialisation.
\end{theorem}
\begin{proof}
We will consider only the special case when the degenerate fibre consists of $m$ distinct lines passing through a point $O$. The more general case of $m$ distinct lines is left as an exercise for the reader. Let $\{\rho_{1},\rho_{2}\}$ be branches of a given node of the generic curve $C_{t}$. By Lemmas 3.36 and 3.43, we can find a silver branch $\gamma\in D_{t,r}$, which is isolated in finite position, and such that, without loss of generality, $\eta_{1,r,s}(\gamma)$ and $\eta_{2,r,s}(\gamma)$ lift the branches $\rho_{1}$ and $\rho_{2}$ respectively, with value $\epsilon$. We have, using, for example, Remarks 5.11, that $\gamma$ is centred at a point $(k',a')$, specialising to $(k,a)$. By Lemma 5.7, we can assume that there exists a unique silver branch $\breve{\gamma}$ of the degenerate curve $D_{t_{\infty},r_{\infty}}$, centred at $(k,a)$. The point $(t,r,s,k',\epsilon,a')$ belongs to the closed variety $W_{1,2}$, where $W_{j}$ is defined in Lemma 4.3;\\

$W_{1,2}(t,r,s,x,y,z)\equiv W_{1}(t,r,s,x,y,z)\wedge W_{2}(t,r,s,x,y,z)$\\

Hence, so does the specialisation $(t_{\infty},r_{\infty},s_{\infty},k,0,a)$. It follows from the description in Lemma 4.3(iv), that;\\

$val_{\breve{\gamma}}(\eta_{1,r_{\infty},s_{\infty}}^{red})=val_{\breve{\gamma}}(\eta_{2,r_{\infty},s_{\infty}}^{red})=0$\\

By Lemma 5.5, it follows that $\eta_{1,r_{\infty},s_{\infty}}^{red}(\breve{\gamma})$ and $\eta_{2,r_{\infty},s_{\infty}}^{red}(\breve{\gamma})$ lift the branches $\breve{l_{1}}$ and $\breve{l_{2}}$ respectively, where $\{\breve{l_{1}},\breve{l_{2}}\}$ are the branches formed by the intersection of a pair of lines at the origin $O$.\\

We now compute the slopes of the branches $\rho_{1}$ and $\rho_{2}$, denoted $\{c',d'\}$. Recall that the plane curves $D_{r}(x,z)$ are parametrised by;\\

$F(x,z;r)=\sum_{i+j\leq k}r_{ij}x^{i}z^{j}$\\

and the rational function $\eta_{1,r,s}(x,z)={g_{s_{1}^{1}}(x,z)\over g_{s_{1}^{2}}(x,z)}$. We have that;\\

$c'=D(\eta_{1,r,s})_{(k',a')}\centerdot D(F_{r})_{(k',a')}={\partial\eta_{1,r,s}\over\partial x}{\partial F_{r}\over\partial x}|_{(k',a')}+{\partial\eta_{1,r,s}\over\partial y}{\partial F_{r}\over\partial y}|_{(k',a')}$\\

$d'=D(\eta_{2,r,s})_{(k',a')}\centerdot D(F_{r})_{(k',a')}={\partial\eta_{2,r,s}\over\partial x}{\partial F_{r}\over\partial x}|_{(k',a')}+{\partial\eta_{2,r,s}\over\partial y}{\partial F_{r}\over\partial y}|_{(k',a')}$\\

Using a similar limit calculation to the above, (\footnote{The differential calculation should be checked carefully, one should compute the differential of ${g_{s_{1}^{1}}(x,z)\over g_{s_{1}^{2}}(x,z)}$ by parts, first, and then, as was done above, find the specialisation of this function, without specialising the corresponding point. It is easily checked that this specialisation is the same as the differential of $\eta_{1,r_{\infty},s_{\infty}}^{red}$}), we have;\\

$sp(c')=sp(D(\eta_{1,r,s})_{(k',a')}\centerdot D(F_{r})_{(k',a')})=D(\eta_{1,r_{\infty},s_{\infty}}^{red})_{(k,a)}\centerdot D(F_{r_{\infty}})_{(k,a)}$\\

$sp(d')=sp(D(\eta_{2,r,s})_{(k',a')}\centerdot D(F_{r})_{(k',a')})=D(\eta_{2,r_{\infty},s_{\infty}}^{red})_{(k,a)}\centerdot D(F_{r_{\infty}})_{(k,a)}$\\

By the same reasoning as above, we see that $sp(c')=c$ and $sp(d')=d$ as required.
\end{proof}

It is clearly impossible to obtain an asymptototic degeneration, in the sense of Definition 4.1, to a set of lines passing through a single point $O$, without imposing the additional assumption on $C$, that all of the tangent lines to the branches along the axis $x=0$ pass through the point $O$. However, we can find a weaker version of Definition 4.1(iii), for which the results we have shown still hold, namely;\\

(iii)'. Almost every curve $C_{t}$, has $m$ distinct non-singular branches, centred at the points $\{(0,a_{1}),\ldots,(0,a_{m})\}$. We will also refer to this condition as asymptotic.\\

The specific construction of asymptotic degenerations to lines through a given point, follows using a projective method, generally referred to as Severi's cone construction. We consider a class of degenerations, obtained by projecting a fixed irreducible curve $C\subset P^{3}$, onto a plane $H\subset P^{3}$, from a variable point $Q$. We first show the following lemma;\\

\begin{lemma}

Let $C\subset P^{3}$ be an irreducible projective curve, not contained in a hyperplane section. Let ${\omega,l}$ be a hyperplane  and a line, respectively, with $l\varsubsetneq\omega$, and $(l\cap\omega)=Q$. Let $Par_{m}=P^{N}$ denote the parameter space for plane projective curves of degree $m$, where ${N=m(m+3)\over 2}$. Then, there exist morphisms $\{\theta,\phi\}$;\\

$\theta:P^{1}\rightarrow Par_{t}, C_{t}\ defined\ by\ \theta(t)$\\

$\phi:P^{1}\cong l$, $\phi(t)=Q_{t}$, $Q=Q_{0}$\\

with the property that, for $t\neq 0$, $C_{t}$ is the plane projective curve, obtained as the projection of $C$ from $Q_{t}$ to $\omega$, see the remarks at the beginning of Section 4 of \cite{depiro6}.

\end{lemma}

\begin{proof}
Without loss of generality, by making a linear change of variables, we can suppose that the coordinates of $P^{3}$ are given by
$\{X,Y,Z,W\}$ and $\{X,Y,W\}$ denote the coordinates of the hyperplane $\omega$, embedded into $P^{3}$ by the map;\\

$i:P^{2}\rightarrow P^{3}$\\

$[X:Y:W]\mapsto [X:Y:0:W]$\\

We denote the restriction of these coordinates to $A^{2}$ and $A^{3}$ by $(x,y)$ and $(x,y,z)$ respectively, and we have a corresponding inclusion;\\

$i:A^{2}\rightarrow A^{3}$\\

$(x,y)\mapsto (x,y,0)$\\

We can assume that $Q$ corresponds to the origin $(0,0,0)$ of the coordinate system $(x,y,z)$, with projective coordinates $[0:0:0:1]$, and $l$ corresponds to the line $x=y=0$.  Let $C$ be defined by the affine equations $F_{1}(x,y,z)=F_{2}(x,y,z)=0$, and let $(0,0,t)$ parametrise a point $Q_{t}$, along the line $l$. Let;\\

$C_{t}(x,y)\equiv\exists\lambda\exists z_{1}\exists z_{2}\exists z_{3}[F_{1}(z_{1},z_{2},z_{3})=F_{2}(z_{1},z_{2},z_{3})=0\\
\indent \indent \indent \indent \indent \indent \indent \indent \indent \ \ \ \ \ \ \ \ \ \ \wedge\lambda(z_{1},z_{2},z_{3})+(1-\lambda)(0,0,t)=(x,y,0)]$ $(*)$\\

By the assumptions on $C$, the construction in Section 4 of \cite{depiro6}, and the definition $(*)$, for $t\neq 0$, $\overline{C_{t}(x,y)}=pr_{Q_{t}}(C)$, where the closure is taken inside the projective plane $\omega$ or $P^{2}$. We let $\overline{C}\subset A^{3}$ denote the closure of $C(x,y,t)$. By elementary dimension considerations, and the fact that the generic fibre $C_{t}(x,y)$ is irreducible, being an open subset of the irreducible curve $pr_{Q_{t}}(C)$, $\overline{C}$ is also irreducible. One can, therefore, find a single equation $G(x,y,t)=0$, defining the surface $\overline{C}$. By making the substitutions $x={X\over W}$ and $y={Y\over W}$, we can find a homogeneous polynomial $H(X,Y,W,t)$, of degree $m$, in the variables $\{X,Y,W\}$, whose restriction to $A^{3}$, corresponds to $G(x,y,t)$. We define $\theta$, as a morphism to $Par_{m}$, by taking the coefficients of the polynomial $H$. It is a simple exercise to check that the conditions of the lemma are satisfied.\\

\end{proof}

We observe the following further property of the degenerations, given in Lemma 5.13.

\begin{lemma}{Projective Degenerations}\\

Let $C\subset P^{3}$ be an irreducible projective curve and $\omega$ a hyperplane, intersecting $C$ transversely in $m$ distinct points $\{P_{1},\ldots,P_{m}\}$, where $m=deg(C)$. Let $l$ be a line, intersecting $\omega$ in a new point $Q$, then the curve $C_{0}$ consists of the union of lines $\bigcup_{1\leq j\leq m}l_{P_{j}Q}$.
\end{lemma}

\begin{proof}
We show that the curves $C_{t}$ are obtained as the intersection of $\omega$, with the cone $Cone_{Q_{t}}(C)$, for $t\in P^{1}$, $(*)$. We, first, consider the case, that $t\neq 0$, and, without loss of generality, assume that $Q_{t}\notin C$, the case when $Q_{t}\in C$ being left to the reader. It is a straightforward exercise, to check that;\\

$Cone_{Q_{t}}(C)=\bigcup_{w\in C}l_{Q_{t}w}$\\

is a closed, irreducible,subvariety of $P^{3}$. The fact that $pr_{Q_{t}}(C)=(Cone_{Q_{t}}(C)\cap\omega)$ follows immediately from the definitions of the projection $pr_{Q_{t}}$, and the cone $Cone_{Q_{t}}(C)$. As before, we assume that $C$ is defined by the affine equations $F_{1}(x,y,z)=F_{2}(x,y,z)=0$, $Q_{t}$ is parametrised by $(0,0,t)$ and $\omega$ is defined by $z=0$, in the coordinate system $(x,y,z)$. Let;\\

$D_{t}(x,y,z)\equiv\exists\lambda\exists z_{1}\exists z_{2}\exists z_{3}[F_{1}(z_{1},z_{2},z_{3})=F_{2}(z_{1},z_{2},z_{3})=0\\
\indent \indent \indent \indent \indent \indent \indent \indent \ \ \ \ \ \ \ \ \ \ \ \ \ \wedge\lambda(z_{1},z_{2},z_{3})+(1-\lambda)(0,0,t)=(x,y,z)]$\\

Let $\overline D\subset A^{4}$ denote the closure of $D(x,y,z,t)$. As before, using the fact that the generic fibre $D_{t}(x,y,z)$ is an open subset of the irreducible variety $Cone_{Q_{t}}(C)$, $\overline{D}$ is irreducible and we can find a single equation $G(x,y,z,t)=0$, defining the surface $\overline{D}$. By projectivizing the polynomial $G$, we can find a single homogeneous polynomial $G(X,Y,Z,W,t)$, parametrising the cones $Cone_{Q_{t}}(C)$. We let;\\

$H(X,Y,W,t)\equiv G(X,Y,0,W,t)$\\

By definition, $H$ parametrises the intersections $(Cone_{Q_{t}}(C)\cap\omega)$. It is a simple exercise, to check that this is a different definition of the same variety, given in Lemma 8.1, hence $(*)$ is shown. In order to complete the proof, it is sufficient to compute the intersection $(Cone_{Q_{0}}(C)\cap\omega)$. If $w\in (C\setminus\omega)$, then $(l_{wQ_{0}}\cap\omega)=Q_{0}$, whereas, if $w\in (C\cap\omega)$, then $(l_{wQ_{0}}\cap\omega)=l_{wQ_{0}}$, using the fact that $Q_{0}\in\omega$. The result follows immediately, from the fact that $(C\cap\omega)=\{P_{1},\ldots,P_{m}\}$, and distinctness of the point $Q=Q_{0}$.
\end{proof}

The above lemmas allow us to easily construct degenerations to lines through a given point. In order to obtain the extra asymptotic conditions, we require the following variant;\\

\begin{lemma}{Mirror Construction}\\

Let $\{\omega_{1},\omega_{2}\}\subset P^{3}$ be two planes intersecting in a line $l$. Let $\overline{C}\subset\omega_{1}$, with the property that $pr_{P}(\overline C)=C$, relative to the hyperplane $\omega_{2}\subset P^{3}$. Let $l'$ be a line, passing through $P$, and intersecting $\omega$ in $Q$. The set of projections;\\

$pr_{x}(\overline{C}):x\in (l'\setminus Q)\}$\\

together with the limit curve $"pr_{Q}(\overline{C})"$, forms a projective degeneration, in the sense of the previous lemma and remarks at the beginning of Section 1. Moreover, if $\{P_{1},\ldots,P_{m}\}$ denote the points of intersections $({\overline C\cap l})=(C\cap l)$, then the degeneration is asymptotic relative to the axis $l$.

\end{lemma}

\begin{proof}

The proof that this gives a degeneration, satisfying the previous lemma and remarks at the beginning of Section 1, is given in the previous lemma and the paper \cite{depiro6}, using Lemma 8.1 to show that degree is preserved, and Lemma 8.3, to show that the genus may be preserved for a suitable choice of $P$ (a birational projection will preserve the genus of the original curve). The limit curve $"pr_{Q}(C)"$ is not formally defined in terms of projections, however, is easily shown to be a union $\bigcup_{y\in(\overline{C}\cap\omega_{1})}l_{Qy}$, by the previous lemma. Again, the reader should look at the paper \cite{depiro6}, Lemma 8.2. The asymptotic condition follows straightforwardly from the fact that each curve in the degeneration contains the points $\{P_{1},\ldots,P_{m}\}$, which are fixed by each projection, being in the plane  $\omega_{2}$.\\

\end{proof}

The above results correct the erroneous step of Severi's conjecture, found in \cite{Severi}, as pointed out by Zariski in \cite{Zariski}. Zariski's counterexample fails to apply to Severi's cone construction, as it is not asymptotic. We will leave a closer examination of Severi's proof to another paper.

\end{section}


\begin{thebibliography}{99}

\bibitem{Aby1} S.S. Abhyankar, Algebraic Geometry for Scientists
and Engineers, AMS Mathematical Surveys 35, (1990)\\

\bibitem{painton} Painton Cowen, "The Rose Window, Splendour and Symbol", Thames and Hudson, (2005)\\

\bibitem{depiro3} T. de Piro, Infinitesimals in a Recursively Enumerable Prime Model, available at http://www.magneticstrix.net\\

\bibitem{depiro4} T. de Piro, Zariski Structures and Algebraic Curves, available at http://www.magneticstrix.net\\


\bibitem{depiro5} T. de Piro, A Non-Standard Bezout Theorem for Curves, available at http://www.magneticstrix.net\\


\bibitem{depiro6} T. de Piro, A Theory of Branches for Algebraic Curves, available at http://www.magneticstrix.net\\

\bibitem{depiro7} T. de Piro, A Theory of Divisors for Algebraic Curves, available at http://www.magneticstrix.net\\


\bibitem{depiro8} T. de Piro, Some Geometry of Nodal Curves, available at http://www.magneticstrix.net\\


\bibitem{depiro11} T. de Piro, A Note on the Etale Topology, available at http://www.magneticstrix.net\\


\bibitem{depiro13} T.de Piro, A Theory of Duality for Algebraic Curves, available at http://www.magneticstrix.net\\

\bibitem{GG} John Gribbin and Mary Gribbin, From Here to Infinity, National Maritime Museum, Royal Observatory, Greenwich, (2008)\\

\bibitem{Guralnick} R. Guralnick, Monodromy Groups of Coverings of Curves, Galois Groups and Fundamental Groups, 1-46, Math. Sci. Res. Inst. Publ., 41, Cambridge University Press, Cambridge, (2003).\\

\bibitem{Harris}, J. Harris, On the Severi Problem,\\

\bibitem{Ha} R. Hartshorne, Algebraic Geometry, Springer (1977)\\

\bibitem{K} S. Kleiman, Bertini and his Two Fundamental Theorems, AG arXiv, 9704018.\\

\bibitem{helga} Helga Meyerbroker, "Rose Windows and How to Make Them", Coloured tissue paper crafts, Floris Books, (1989)\\

\bibitem{M} J.S Milne, Lectures on Etale Cohomology, available at http://www.math.lsa.umich.edu/$\sim$jmilne/\\

\bibitem{Mum} D. Mumford, Red Book of Varieties and Schemes,
Springer (1999)\\

\bibitem{Sernesi} E. Sernesi, Deformations of Algebraic Schemes, Volume 334, Springer, (2006)\\

\bibitem{Severi} F. Severi, Vorlesungen uber algebraische Geometrie, Anhang F, Teubner, Leipzig, 1921.\\

\bibitem{Tracing} E.H Neville, The Tracing of Cubic Curves, The Mathematical Gazette, Vol 18, No. 230, pp 258-266, (1934)\\

\bibitem{Anal} Isaac Newton, Analysis by Equations of an Infinite Number of Terms, (written in 1669, a published version (in Latin) appearing in 1711, and a translated published version appearing in 1745)\\

\bibitem{Cubics} Isaac Newton, Enumeration of Lines of the Third Order.\\

\bibitem{Flux} Isaac Newton, The Method of Fluxions and Infinite Series.\\

\bibitem{Nunemacher} J. Nunemacher, Asymptotes, Cubic Curves, and the Projective Plane, Mathematics Magazine, Volume 72, No. 3, June 1999.\\

\bibitem{Smart} E.H Smart, The Nature of an Algebraic Curve at Infinity and the Conic of Closest Contact at Infinity, The Mathematical Gazette, Vol 13, No. 188, pp 357-361, (1927)\\

\bibitem{Talbot} C.R.M Talbot, Notes and Examples on Sir Isaac Newton's Enumeration of Lines of the Third Order, (H.G Bohn, Covent Garden, London), (1861)\\

\bibitem{Zariski} Dimension Theoretic Characterization of Maximal Irreducible Algebraic Systems, American Journal of Mathematics, Vol 104, No 1, pp 209-226, (1979)\\

\end{thebibliography}
\end{document}